\documentclass[10pt]{amsart}
\usepackage{amsmath,amssymb}
\usepackage[ruled,algosection]{algorithm2e}
\usepackage[pdftex]{graphicx,color}
\usepackage{subfigure}
\usepackage{hyperref} 
\hypersetup{
  colorlinks=true,
  linktocpage=true,
  bookmarksnumbered=true
}
\newtheorem{theorem}{Theorem}[section]
\newtheorem{lemma}[theorem]{Lemma}
\newtheorem{remark}[theorem]{Remark}
\newtheorem{corollary}[theorem]{Corollary}
\newtheorem{example}[theorem]{Example}
\newtheorem{assumption}[theorem]{Assumption}
\newcommand{\ie}{\emph{i.e.}}
\newcommand{\eg}{\emph{e.g.}}
\newcommand{\cf}{\emph{cf.}}
\newcommand{\ud}{\mathrm{d}}
\newcommand{\prob}{\mathrm{P}}
\newcommand{\expt}{\mathrm{E}}
\newcommand{\var}{\mathrm{Var}}
\newcommand{\rset}{\mathbb{R}}

\newcommand{\zset}{\mathbb{Z}}
\newcommand{\one}{\mathbf{1}}
\newcommand{\poisson}{\mathcal{P}}
\newcommand{\normal}{\mathcal{N}}
\newcommand{\binomial}{\mathcal{B}}
\newcommand{\Ordo}[1]{{\mathcal{O}}\left(#1\right)}
\newcommand{\ordo}[1]{{o}\left(#1\right)}

\newcommand{\abstr}[1]{\begin{abstract}#1\end{abstract}}
\newcommand{\thx}[1]{\thanks{#1}}

\title[Computable Weak Error Expansion for the Tau-Leap Method] {Towards
  Automatic Global Error Control: Computable Weak Error Expansion for
  the Tau-Leap Method}

\author{J. Karlsson}

\address{Applied Mathematics and Computational Sciences, KAUST, Saudi
  Arabia}

\email{jesper.karlsson@kaust.edu.sa}

\author{R. Tempone}

\address{Applied Mathematics and Computational Sciences, KAUST, Saudi Arabia}

\email{raul.tempone@kaust.edu.sa}

\subjclass[2000]{Primary: 60H10, 60H35; Secondary: 60J75, 65Y20,
  65L50}

\keywords{Tau-leap, weak approximation, reaction networks, markov
  chain, error estimation, a posteriori error estimates, backward dual
  functions}

\thx{
  The authors would like to thank Markos Katsoulakis, Fabio Nobile and
  Anders Szepessy for their valuable ideas and comments on this
  work. The authors would also like to recognize the support of the
  PECOS center at ICES, University of Texas at Austin (Project Number
  024550, Center for Predictive Computational Science). Support from the
  VR project "Effektiva numeriska metoder f\"or stokastiska
  differentialekvationer med till\"ampningar" and King Abdullah
  University of Science and Technology (KAUST) is also acknowledged.
}

\numberwithin{equation}{section}
\numberwithin{figure}{section}
\numberwithin{table}{section}

\begin{document}

\abstr{
  This work develops novel error expansions with computable leading
  order terms for the global weak error in the tau-leap discretization
  of pure jump processes arising in kinetic Monte Carlo models.
  Accurate computable a posteriori error approximations are the basis
  for adaptive algorithms; a fundamental tool for numerical simulation
  of both deterministic and stochastic dynamical systems. These pure
  jump processes are simulated either by the tau-leap method, or by
  exact simulation, also referred to as dynamic Monte Carlo, the
  Gillespie algorithm or the Stochastic simulation algorithm. Two types
  of estimates are presented: an \emph{a priori} estimate for the
  relative error that gives a comparison between the work for the two
  methods depending on the propensity regime, and an \emph{a posteriori}
  estimate with computable leading order term.
}


\maketitle
\setcounter{tocdepth}{1}
\tableofcontents

\section{Introduction}\label{sec:introduction}
In this work we derive a global weak error expansion with computable
leading order term for the tau-leap method.  The tau-leap method was originally
proposed by Gillespie in \cite{gillespie_tau_leap}, for approximating
homogeneous and well stirred stochastically reacting chemical systems.
In such systems different species undergo reactions at random times to
form new species or to decay. Each reaction can be modeled by a
\emph{propensity function}, directly related to the number of
particles, indicating the probability of a reaction to happen per unit
of time.  The notion \emph{well stirred} here means that the number of
reactive particle collisions are low compared to the total number of
collisions.  Depending on the number of particles in the system and
the time to next reaction the reaction process can be modeled
differently:

For all possible regimes, a homogeneous well stirred system in thermal
equilibrium can be modeled by the \emph{chemical master equation}
(CME). This ordinary differential equation describes the time
evolution of the probability of each particle configuration, \cf{}
\cite{gillespie_cme}. Since the dimension of the solution space is of
the order of all the possible configurations, the CME is in practice
impossible to solve.  On the other hand, the system can still be
simulated \emph{exactly} using the \emph{Stochastic simulation
  algorithm} (SSA), also introduced by Gillespie in
\cite{gillespie_ssa}.  The SSA numerically simulates the Markov
process described by the CME by using dynamic Monte Carlo sampling,
also referred to as \emph{kinetic Monte Carlo}. Although the SSA
generates exact path realizations for the Markov process it is only
tractable for low propensities. Indeed, since each reaction is
simulated exactly by sampling the next reaction to happen and the time
to this reaction, the total computational work becomes roughly
inversely proportional to the total propensity.  The tau-leap method,
on the other hand, approximates the SSA by evolving the chemical
system with fixed time steps, keeping the propensity fixed in each
time step, and can be seen as a forward Euler method for a stochastic
differential equation driven by Poisson random measures, \cf{}
\cite{li}.  In the limit, as the time steps go to zero, the tau-leap
solution converges to the SSA, see \cite{rathinam}.

As the number of particles in the system grows the SSA is
sometimes approximated by the \emph{chemical Langevin} diffusion
equation. Further, if both the number of particles and the total
volume of the system goes to infinity at the same speed, \ie{} the
randomness of the system becomes negligible, then the concentration of
each species over time can be modeled by a deterministic system of
ordinary differential equations, the \emph{reaction rate equations}
(RREs).
Because of different time scales in the reactions, RREs are often are
\emph{stiff}, indicating the need for stable tau-leap methods and adaptive
time-stepping.  Implicit stable tau-leap methods that deal with stiff cases
have been proposed by \cite{cao_stab, rathinam_stab}. Several authors
have discussed the importance of leap selection procedures to increase
efficiency \cite{cao_eff_step_size, gillespie_leap_size} and to avoid
negative populations \cite{anderson, cao_negative_pop,
  chatterjee_binomial, tian_burrage}.

Representing the number of particles for the species in the system by
the stochastic vector $X(t)$, the goal in this work is to approximate
the real valued quantity, $\expt [g(X(T))]$, for some given function
$g$ and initial configuration $X(0)$. We here derive an \emph{a
  posteriori} estimate for the global weak error $\expt
[g(X(T))-g(\bar X(T))]$ between exact (SSA) solution $X$ and the tau-leap
solution $\bar X$, based on an exact global error representation,
cf. Lemma \ref{lem:tau_exp_error_rep}.  The error representation uses
the value function defined by the Kolmogorov backward equation
\eqref{eq:kolmogorov}.  Also, an \emph{a priori} estimate for the
relative global weak error that is independent of the number of
particles is derived.  Both the a priori and a posteriori error
estimates are based on a continuous extension and point wise bounds of
the discretely defined value function and its derivatives. In
\cite{katsoulakis}, similar $L$-infinity bounds, exponentially
dependent on the propensity function, were derived. This work improves
those estimates showing that the value function and its derivatives
only grows polynomially with the population. This result is based on
weighted estimates and a stochastic representation of the weighted
derivatives in terms of pure jump processes with modified
propensities.

Adaptive time stepping algorithms based on computable a posteriori
error estimates are fundamental tools for numerical simulation of
both deterministic and stochastic dynamical systems. In our
case, the leading order term of the error expansion is approximated by
a discrete dual weighted propensity residual, similar to \cite{mstz3,
  mstz2} for deterministic differential equations, and \cite{mstz,
  msst, mordecki, stz} for stochastic differential equations of
diffusion and jump-diffusion type.  To the best of the authors'
knowledge there are still no results on optimal adaptive time stepping
algorithms based on a posteriori error estimates for the global weak
error in the tau-leap method.

The outline of this work is as follows: first, Section
\ref{sec:problem} presents the mathematical setting together with the
main assumptions and the relation with the Kolmogorov backward
equation upon which our results are based.  Next, Section
\ref{sec:num_approx} introduces the tau-leap method and a numerical
scheme for avoiding negative populations (following \cite{anderson}),
here called the \emph{Poisson bridge} tau-leap method.
Section \ref{sec:accuracy} presents the main results, namely an
\emph{a priori} bound on the relative error of the tau-leap method
(Theorem \ref{thm:rel_error_bnd}), and a computable \emph{a
  posteriori} error approximation using a discrete dual (Theorem
\ref{thm:aposteriori}).  
Also, from the a priori bound in Theorem \ref{thm:rel_error_bnd} a
work comparison between the tau-leap method and SSA is made.  Finally,
Section \ref{sec:examples} provides a numerical verification of our a
posteriori error estimate in Theorem \ref{thm:aposteriori}.

\section{Problem}\label{sec:problem}
We consider a well stirred system of $d$ chemical species, which
interact through $M$ chemical reaction channels.  The system is
assumed to be confined to a constant volume $\Omega$ and to be in
thermal, but not necessarily chemical, equilibrium at some constant
temperature. With $X_{t}^{(i)}$ denoting the number of molecules of
species $i$ in the system at time $t$, we want to study the evolution
of the state vector $X_t = (X_{t}^{(1)}, \ldots, X_{t}^{(d)})$, given
that the system was initially in some state $X_{t_0} = x_0$.  It is
here assumed that $X_t, x_0 \in \zset_+^d$ where $\zset_+$ denotes the
set of non-negative integers.  The goal of the computation is the
approximation of the quantity
\begin{equation*}
  \expt [ g(X_T) ],
\end{equation*}
where $g:\rset^d\to \rset$ is a given function.

Each of the $j=1,\ldots,M$ reactions is characterized by the change
\begin{equation*}
  x \to x+ \nu_j  
\end{equation*}
where $\nu_j\in \zset^d$ is a stoichiometric vector that indicates the
change in the state vector $x\in \zset_+^d$, produced by a single
firing of the reaction $j$.  Now that we have defined what happens
when each reaction takes place, we also need to indicate how often
each of those reactions may occur in time. This information is carried
by the \emph{propensity functions}, $a_j:\zset_+^d \to \rset$,
$j=1,\ldots,M$, which tell the probability of the reaction $j$
happening during the infinitesimal interval $(t,t+\ud t)$ \ie{}
\begin{equation*}
  \prob\Bigl(j \text{  fires during } (t,t+\ud t) \bigm| X_t = x\Bigr) =
  a_j(x) \ud t, \quad j=1,\ldots,M.
\end{equation*}
Usually the functions $a_j$ are just polynomials, see Remark
\ref{rem:propensity}. The compensator of our reaction
process is state dependent, and the Kolmogorov backward equation for
this pure jump process and a smooth observable $g:\rset^d\to \rset$ is
\begin{equation}\label{eq:kolmogorov}
  \begin{aligned}
    \partial_t u + \sum_{j=1}^M a_j(x) \Bigl( u(x+ \nu_j,t) - u(x,t)
    \Bigr) &= 0, && \text{ in } \zset_+^d \times [0,T),\\
    u &= g, && \text{ on } \zset_+^d \times \{T\},\\
  \end{aligned}
\end{equation}
with the corresponding value function
\begin{equation*}
  u(x,t) := \expt \bigl[ g(X_T) \bigm| X_t=x \bigr].
\end{equation*}
Here, natural boundary conditions restrict the species to the set
$\zset_+^d$, see Remark \ref{rem:bnd_val}.

Let $\mathcal{Z} := \{\nu_1,\ldots,\nu_M\}$. It is interesting to see
that the previous equation can be rewritten in the following way
\begin{equation}\label{eq:chem_equiv}
  \begin{aligned}
    \partial_t u + a_0(x)
    \int_{\mathcal{Z}}\Bigl( u(x+ z,t) - u(x,t) \Bigr) q(x,\ud z) &=
    0, && \text{ in } \zset_+^d \times [0,T),\\ 
    u(\cdot,T) &= g, && \text{ on } \zset_+^d \times \{T\},\\
  \end{aligned}
\end{equation}
\begin{equation*}
  a_0(x):= \sum_{j=1}^M a_j(x),
\end{equation*}
where the compensator measure $q(x,\ud z)$ is atomistic in the
stoichiometric vector set $\mathcal{Z}$ and such that
\begin{equation*}
  q(x,z=\nu_j) = \prob ( z=\nu_j \,|\, X_t=x ) = 
  \frac{a_j(x)}{a_0(x)}, \quad j=1,\ldots,M.  
\end{equation*}
Although \eqref{eq:kolmogorov} and \eqref{eq:chem_equiv} are
equivalent they may lead to different discretizations.  In
particular, the stochastic simulation algorithm (SSA) introduced by
Gillespie \cite{gillespie_ssa} is the natural discretization scheme
for \eqref{eq:chem_equiv}. Moreover, the corresponding error estimates
and error control algorithms behave differently.

Since \eqref{eq:kolmogorov} is driven by pure jumps with finite
activity it is in principle possible to simulate trajectories of $X_t$
exactly; the SSA does precisely that. It only requires the sample of
two random variables per time step: one to find the time of the next
reaction and another to decide which is the reaction that is firing at
that time. The drawback of this algorithm appears clearly as the sum
of the intensity of all reactions, $a_0(x)$, gets large: since all the
jump times have to be included in the time discretization the
corresponding computational work may become unaffordable. Indeed, we
have that the mean value of the number of jumps on the interval
$(t,t+\tau)$ is approximately $a_0(X_t) \tau + \ordo{\tau}$.

\begin{assumption}[Bounded population]\label{assu:boundedness}

  We here assume that species can only be transformed into other
  species or be consumed.
  This means that the state $X_t \in \zset_+$ will be bounded from
  above by a hyperplane intersecting the point $X_0$ and the
  coordinate axes, \ie{}
  \begin{equation*}
    X_t \in \Pi(X_0,n):= \{x \in \zset_+^d: n \cdot (x-X_0) \leq 0 \}.
  \end{equation*}
  for some vector $n \in \rset_+^d$ with strictly positive
  coordinates, \ie{} $n>0$.
  
  In other words, all the stoichiometric vectors should satisfy
  \begin{equation*}
  n \cdot \nu_j \le 0, \, j=1,\ldots,M.
  \end{equation*}
\end{assumption}

\begin{remark}[Preventing negative populations]\label{rem:propensity}

  To prevent the state from becoming negative after a jump we must have
  that $a_j(x)=0$ for $x+\nu_j \notin \zset_+^d$. For the same reason
  we also have that $a_j(0)=0$.  For common chemical reactions the
  propensity functions can be modeled as polynomials of the form
  \begin{equation}
    \label{eq:propensity}
    a_j(x) := c_j \prod_{i=1}^d \frac{x_i!}{(x_i+\nu_{ij}^-)!}
    \one_{(x_i + \nu_{ij}^-) \geq 0},
  \end{equation}
  which satisfies the above criteria, see
  \cite{agk,gillespie_ssa}. Here, $\nu_{ij}^- :=
  \min(\nu_{ij},0)$. From the expression \eqref{eq:propensity} it also
  follows that the propensity is monotone in $\zset_+^d$, \ie{} the
  gradient of the polynomial function $a_j$ is non-negative.
 
\end{remark}

\begin{remark}[Boundary values]\label{rem:bnd_val}

  From Remark \ref{rem:propensity} it follows that Equation
  \eqref{eq:kolmogorov} has a natural boundary condition at the
  boundaries where any component of the $x$ vector is zero.

\end{remark}


\section{The tau-leap method}\label{sec:num_approx}
To avoid the computational drawback of the SSA, \ie{} when many
reactions occur during a short time interval, the \emph{tau-leap} method
was proposed in \cite{gillespie_tau_leap}: Given a population $\bar
X_t$, and a time step $\tau>0$, the population at time $t+\tau$ is
generated by
\begin{equation}
  \label{eq:tau_leap}
  \bar X_{t+\tau} = \bar X_t + \sum_{j=1}^M \nu_j 
  \poisson_j( a_j( \bar X_t )\tau ),
\end{equation}
where $\poisson( a_j( \bar X_t )\tau )$ is a sample value from the
Poisson distribution with parameter $a_j( \bar X_t )\tau$, indicating the
sampled number of times that the reaction $j$ fires during the
$(t,t+\tau)$ interval. 

The tau-leap method \eqref{eq:tau_leap} is based on the observation
that if the propensity is constant between $t$ and $t+\tau$, the
firing probability in one reaction channel is independent of the other
reaction channels. The total number of firings in each channel is then
a Poisson distributed stochastic variable depending only on the
initial population $\bar X_t$.  Also, Equation \eqref{eq:tau_leap} is
nothing else than a forward Euler discretization of the SDE
corresponding to the Kolmogorov backward equation
\eqref{eq:kolmogorov}, (or the chemical master equation), driven by
the Poisson random measure, \cf{} \cite{li}.

In the following we let $X_t$ denote the exact process and $\bar X_t$
the tau-leap approximation.

\begin{remark}\label{rem:a_bar}
  Although the path generated by \eqref{eq:tau_leap} seems to only be
  defined for discrete time steps, it can be extended to the continuous
  time interval. Indeed, any intermediate time steps in the
  $(t,t+\tau)$ interval can be defined as a Poisson bridge with fixed
  endpoints $\bar X_t$, $\bar X_{t+\tau}$ and fixed intensity $\bar a_j :=
  a_j(\bar X_t)$.
\end{remark}

After freezing the propensity to $a(\bar X_t)$ in each time step
$[t,t+\tau]$, the error in the tau-leap method \eqref{eq:tau_leap}
comes from the variation of $a(X_s)$ for $s\in(t,t+\tau)$, where $X_s$
is the true process starting at $X_t=\bar X_t$.  To take this into
account it was in \cite{gillespie_tau_leap} proposed choosing local
time steps following the leap condition
\begin{equation}
  \label{eq:leap_condition}
  |\Delta a_j(\bar X_t)| := 
  |a_j(\bar X_{t+\tau})-a_j(\bar X_t) | \leq \epsilon a_0(\bar X_t),
\end{equation}
for a given control parameter $0 < \epsilon \ll 1$.
In order to avoid unnecessary sampling of the Poisson random variable
a pre-leap check can be done by first doing a Taylor expansion 
\begin{equation*}
  \begin{aligned}
    \Delta a_j(\bar X_t) 
    &= a_j \left( \bar X_t + 
      \sum_{i=1}^M \nu_i \poisson_i ( a_i(\bar X_t)\tau )
    \right) - a_j(\bar X_t)\\
    &\approx \nabla a_j(\bar X_t) \cdot 
    \sum_{i=1}^M \nu_i \poisson_i ( a_i(\bar X_t)\tau ),
  \end{aligned}
\end{equation*}
and approximating the mean and variance of $\Delta a_j(x)$ by
\begin{equation*}
  \frac{1}{\tau} \expt [\Delta a_j(x)] \approx \sum_{i=1}^M a_i(x)
  \left( \nabla a_j(x) \cdot \nu_i \right) =: \mu_{j}(x), \quad j = 1,\ldots,M,
\end{equation*}
and
\begin{equation*}
  \frac{1}{\tau} \var[\Delta a_j(x)] \approx \sum_{i=1}^M a_i(x) \left( \nabla
    a_j(x) \cdot \nu_i \right)^2 =:  \sigma_{j}(x)^2, \quad j = 1,\ldots,M,
\end{equation*}
respectively, see \cite{gillespie_leap_check}.
The leap size is then chosen as
\begin{equation}
  \label{eq:leap_size}
  \tau  = \min_{j=1,\ldots,M} \min \left(
    \frac{\epsilon a_0(\bar X_{t})}{|\mu_j(\bar X_{t})|}, 
    \frac{\epsilon^2 a_0(\bar X_{t})^2}{\sigma_j(\bar X_{t})^2} \right),
\end{equation}
which implies that $|\expt [\Delta a_j]|\leq \epsilon a_o$ and
$\var[\Delta a_j]\leq \epsilon^2 a_o^2$, so that the leap condition
\eqref{eq:leap_condition} holds in some statistical sense.  The leap
size control \eqref{eq:leap_size} has the advantage that it can be
computed before each step is taken; however, it relies on the \emph{ad
  hoc} parameter $\epsilon$ and only controls the local error.

Our goal is here to develop a rigorous \emph{a posteriori} global
weak error estimate for \eqref{eq:tau_leap}, with computable leading
order term, that can be used to control the error in a more systematic
way than the leap condition \eqref{eq:leap_condition}, and that can in
future work be used for efficient adaptive time stepping algorithms.
In Section \ref{sec:poisson_bridge} we present a procedure that
eliminates the possibility of negative populations $\bar X_t$,
similarly to \cite{anderson}, and in Section \ref{sec:accuracy} we
develop error estimates that are finally tested numerically in Section
\ref{sec:examples}.

\subsection{Avoiding negative population} \label{sec:poisson_bridge}
The regular tau-leap approximation \eqref{eq:tau_leap} suffers from
the undesirable property that $\bar X_t$ can become negative.  To
prevent such unphysical behavior, which is a result of the
approximation and not the process itself, the idea is to adaptively
adjust the time step $\tau$ to avoid negative populations and at the
same time leave the distribution of $\bar X_t$ unchanged. Similarly to
\cite{anderson}, this is done by post-leap checks such that if any
component of $\bar X_{t+\tau}$ becomes negative, a new step $\bar
X_{t+\tau/2}$ will be sampled using a conditional Poisson
distribution, \ie{} a \emph{Poisson bridge}. To describe this
procedure, we first introduce independent unit rate Poisson processes
$Y_j(\cdot)$, and define their corresponding \emph{internal times} as
\begin{equation*}
  \lambda_t^j := \int_0^t a_j(X_s) \, \ud s, \quad j=1,\ldots,M.
\end{equation*}
The state of the chemical system at time $t$
then satisfies
\begin{equation*}
  X_t = X_0 + \sum_{j=1}^M \nu_j Y_j ( \lambda_t^j ),
\end{equation*}
see \cite{anderson}, \ie{} each reaction channel is described by a
unit rate Poisson process with an internal time given by
$\lambda^j_t$. 
From the definition of $Y_j$ we have Poisson distributed increments
\begin{equation*}
  \begin{aligned}
    Y_j( \lambda_{t+\tau}^j ) - Y_j(
    \lambda_t^j )
    &= Y_j \left(  \lambda_t^j + \int_t^{t+\tau} a_j(X_s) \, \ud s \right) - 
    Y_j( \lambda_t^j )\\
    &= \poisson_j \left( \int_t^{t+\tau} a_j(X_s) \, \ud s \right),
  \end{aligned}
\end{equation*}
and since we have for each tau-leap step that
\begin{equation}\label{eq:delta_lambda}
  \begin{aligned}
    \int_t^{t+\tau} a_j(\bar X_s) \, \ud s = \bar{a}_j(\bar X_t) \tau
    := \Delta \lambda_t^j,
  \end{aligned}
\end{equation}
the tau-leap method can be written in
terms of increments in $Y_j$, \ie{}
\begin{equation}\label{eq:tau_lambda}
  \bar X_{t+\tau} = \bar X_t + 
  \sum_{j=1}^M \nu_j \Delta Y_j(\Delta \lambda_t^j).
\end{equation}
where
\begin{equation}\label{eq:delta_y}
  \Delta Y_j(\cdot) := \poisson_j \left( \cdot \right).
\end{equation}

As the tau-leap method steps forward in time we build up a history of
samples for the driving process $Y$, $\{(\lambda_i^j,
Y_i^j)\}_{i=0}^{n_j}$, such that $0 = \lambda_0^j < \lambda_1^j <
\ldots < \lambda_{n_j}^j$, and $0 = Y_0^j \leq Y_1^j \leq \ldots \leq
Y_{n_j}^j$, by summing increments $(\Delta \lambda_t^j, \Delta
Y_j(\Delta \lambda_t^j))$.  Starting at $X_t$ and at the last value
$(\lambda_k^j, Y_k^j)$ in the history, a regular tau-leap step samples an
increment $\Delta Y_t^j:=\Delta Y_j(\Delta \lambda_t^j)$ and saves
$(\lambda_k^j + \Delta \lambda_t^j, Y_k^j + \Delta Y_t^j)$ to the
history. Let $\binomial(n,p)$ denote the binomial distribution. As
soon as a negative population $\bar X_{t+\tau}<0$ is encountered, a new
step $\bar X_{t+\tau/2}$ is calculated by the Poisson bridge
\begin{equation*}
  \bar X_{t+\tau/2} = \bar X_t + 
  \sum_{j=1}^M \nu_j \binomial(\Delta Y_t^j, 0.5),
\end{equation*}
and $(\lambda_k^j + 0.5 \Delta \lambda_t^j, Y_k^j + \binomial(\Delta
Y_t^j, 0.5))$ is added to the saved history. If $\bar X_{t+\tau/2}$ is
non-negative we move forward along the $t$-axis, otherwise the step is
halved once again. When traversing the physical time axis $t$ we have
to check if there already exist samples to the right on the internal
time axis $\lambda$, \eg{} given the position $\lambda_k^j$ and a
future value $\lambda_{k+1}^j$ the physical time step $\tau$ must be
adjusted such that we end up on $\lambda_{k+1}^j$ for at least one $j$
and to the left of $\lambda_{k+1}^j$ for the remaining $j$. For the
reaction $j$ that ends up on $\lambda_{k+1}^j$ we use the
corresponding $Y_{k+1}^j$, and for the other reactions, a bridge
between $(\lambda_k^j, Y_k^j)$ and $(\lambda_{k+1}^j, Y_{k+1}^j)$ is
sampled. Algorithm \ref{alg:PBTL} describes in detail one step of the
Poisson bridge tau-leap method in the interval $(t,t+\tau)$.  To speed up
the algorithm we approximate, for large $np$, the binomial
distribution $\binomial(n,p)$ with the normal distribution
$\normal(np, np(1-p))$ rounded to integers and multiplied by the
indicator function with support in $[0,n]$. We here apply this approximation whenever
$np>10^4$.

\begin{algorithm}
 \caption{One step for the Poisson Bridge Tau-Leap method. Unless
   stated otherwise, it is assumed that steps in the algorithm containing the
   reaction index $j$ are performed for all reactions $j=1,\ldots,M$.}
 \label{alg:PBTL}
 \dontprintsemicolon
 \SetAlgoSkip{bigskip}
 \SetArgSty{text}
 \SetKwInOut{Input}{Input}
 \SetKwInOut{Output}{Output}
 \Input
 { 
   Initial values $(t, \bar X)$ and $(\lambda^j,
   Y^j)$. Coefficients $\nu_j$ and $a_j(x)$,
   step size $\tau$, and an ordered history
   of samples $\Lambda_j:=\{(\lambda_{i}^j, Y_{i}^j)\}_{i=0}^{n_j}$, where 
   $0 = \lambda_0^j  < \lambda_1^j < \ldots < \lambda_{n_j}^j$ and 
   $0 = Y_0^j \leq Y_1^j \leq \ldots \leq Y_{n_j}^j$.
 }
 \Output
 {
   An ordered history of samples $\Lambda_j$, the current value of
   values for $(\lambda^j, Y^j)$, and the steps $\{(t_i,\bar X_i)\}_{i=0}^m$ taken in
   the interval $[t,t+\tau]$.
 }
 \BlankLine
 
 Set local time $t_{loc} = 0$.\;
 \While{($t_{loc}<\tau$)}
 {
   Let $\bar a_j := a_j(\bar X)$, and
   let $k_j$ be the index of current $\lambda^j$ in the sample history $\Lambda_j$. 
   Set local time step $\tau_{loc} := \tau-t_{loc}$ and
   corresponding internal time step $\Delta\lambda^j := \bar a_j\tau_{loc}$.\;

   \For{($j=1,\ldots,M$)}
   {
     Save a temporary local time step $\tau_{loc}^j
     := \tau_{loc}$ for reaction $j$.\; 
     \uIf{($k_j$ is last index in $\Lambda_j$)}
     {
       Sample a new Poisson value $\Delta Y^j := \poisson( 
       \Delta\lambda^j )$, and add 
       $(\lambda_{k_j}^j+\Delta\lambda^j, Y_{k_j}^j+\Delta Y^j)$ to the history $\Lambda_j$.
     }
     \Else
     {
       \uIf{($\lambda_{k_j}^j + \Delta\lambda^j > \lambda_{k_j+1}^j$)}
       {
         Set local time step for reaction $j$:
         $\tau_{loc}^j :=  \frac{\lambda_{k_j+1}^j -
           \lambda_{k_j}^j}{\bar a_j}$.\;
         Set $\Delta\lambda^j := \lambda_{k_j+1}^j - \lambda_{k_j}^j$ and 
         $\Delta Y^j := Y_{k_j+1}^j - Y_{k_j}^j$.\;
       }
       \Else
       {
         Sample a Poisson bridge value 
         $\Delta Y^j := \binomial(Y_{k_j+1}^j - Y_{k_j}^j, \frac{
           \Delta\lambda^j}{\lambda_{k_j+1}^j-\lambda_{k_j}^j})$ 
         from the binomial distribution, and add 
         $(\lambda_{k_j}^j+\Delta\lambda^j, Y_{k_j}^j+\Delta Y^j)$ to the
         history $\Lambda_j$.
       }
     }   
   }
   Set local time step to be the
   minimum physical time to next sampled value in history: 
   $\tau_{loc} :=  \min_j \tau_{loc}^j$.\;
   \While{(any component of $\bar X+ \sum_j \nu_j \Delta Y^j$ is negative)}
   {
     Divide step size by two, \ie{} $\tau_{loc} := 0.5 \tau_{loc}$,
     and $\Delta\lambda^j := 0.5 \Delta\lambda^j$.\;

     Sample a bridge value $\Delta Y^j := \binomial( \Delta Y^j, 0.5 )$.\;

     Add $(\lambda_{k_j}^j+\Delta\lambda^j, Y_{k_j}^j+\Delta Y^j)$ to $\Lambda_j$.
   }   

   Accept new value $\bar X \leftarrow \bar X+\sum_j \nu_j \Delta Y^j$,
   increase internal time $\lambda^j \leftarrow \lambda_{k_j}^j + \Delta\lambda^j$, and
   increase local time $t_{loc} \leftarrow t_{loc} + \tau_{loc}$.\;

}
\end{algorithm}


\begin{remark}\label{rem:pre-leap}
  The post-leap check performed by the Poisson bridge tau-leap method
  in Algorithm \ref{alg:PBTL} will guarantee non-negative
  \emph{sampled} steps. There is however still a probability that
  components of the continuous tau-leap process may become negative at
  points between the sampled steps, as discussed in Remark
  \ref{rem:a_bar}.  To limit this effect we can introduce a pre-leap
  check that adjusts the time step $\tau$ such that
  \begin{equation}\label{eq:exit_prob}
    \prob \bigl(\bar X^{(i)}_{t+\tau} < 0 \bigm| \bar X_t \bigr) <
    \epsilon, \quad i=1,\ldots,d,
  \end{equation}
  for some small $\epsilon>0$, see \cite{kkst} where the exit
  probability \eqref{eq:exit_prob} is approximated by a normal
  approximation of the tau-leap step \eqref{eq:tau_leap}, leading to a
  quadratic inequality in $\tau$.

\end{remark}

\section{Accuracy and error estimation}\label{sec:accuracy}


To develop computable global error estimates for the \emph{Poisson
  bridge} tau-leap method we start with the following (non-computable)
error representation based on the backward Kolmogorov equation
\eqref{eq:kolmogorov}, along the continuous-time approximate tau-leap
paths $\bar X_t$, defined in \eqref{eq:tau_leap} and Remark
\ref{rem:a_bar}, and the difference in propensities:

\begin{lemma}[Error representation for the tau-leap method]\label{lem:tau_exp_error_rep}
  Assume that $a_j$ and $u$ are defined for $x\in\zset^d$ and that $u$
  solves the backward Kolmogorov equation \eqref{eq:kolmogorov} in
  $\zset_+^d \times [0,T]$.  For deterministic time steps we then have
  \begin{equation*}
    \begin{aligned}
      \expt \left[ g ( X_T ) - g ( \bar X_T ) \right] 
      =& \
      \sum_{j=1}^M \expt \Biggl[ \int_0^T  (a_j-\bar a_j) ( \bar X_t )       
      \Bigl( u ( \bar X_t + \nu_j, t ) - u ( \bar X_t, t ) \Bigr) \,
      \ud t \Biggr]\\
      &- \expt \left[ \int_0^T
        \phi(\bar X_t,t) \one_{\{ \bar  X_t\in\zset_-^d \}} 
       \, \ud t \right],
    \end{aligned}
  \end{equation*}
  where
  \begin{equation*}
  \phi(\bar X_t,t) := \partial_t u(\bar X_t,t) + 
  \sum_{j=1}^M \bar a_j(\bar X_t) \Bigl( u ( \bar X_t + \nu_j, t ) - u ( \bar X_t, t ) \Bigr),
  \end{equation*}
  and $\zset_-^d := \zset^d \setminus \zset_+^d$ is the set where at least one
  component is negative.

\end{lemma}

\begin{proof}
  Since $\bar X_0 = X_0$ we have
  \begin{equation}\label{eq:tau_exp_error_proof}
    \begin{aligned}
      \expt \left[ g ( X_T ) - g ( \bar X_T ) \right]
      &= \expt \left[ u (\bar X_0, 0) - u (\bar X_T, T) \right]\\
      &= \expt \left[ \int_0^T -(\partial_t + \mathcal{L}_{\bar X})
        u(\bar X_t, t) \, \ud t \right],
    \end{aligned}
  \end{equation}
  where the "infinitesimal generator" for $\bar X_t$ is defined as 
  \begin{equation*}
    \mathcal{L}_{\bar X} u(\bar X_t, t) := 
    \sum_{j=1}^M \bar a_j(\bar X_t) \bigl( u(\bar X_t+\nu_j,t)-u(\bar X_t,t) \bigr). 
  \end{equation*}
  Note the abuse of notion here: $\bar a_j$, and correspondingly
  $\mathcal{L}_{\bar X} u$, is from Remark \ref{rem:a_bar} \emph{only} defined along paths.
  Adding and subtracting
  \begin{equation*}
    \mathcal{L}_{X} u(\bar X_t, t) := 
    \sum_{j=1}^M a_j(\bar X_t) \bigl( u(\bar X_t+\nu_j,t)-u(\bar X_t,t) \bigr),
  \end{equation*}
  gives
  \begin{equation*}
    \begin{aligned}
      \expt \left[ g ( X_T ) - g ( \bar X_T ) \right] 
      =& \
      \expt \Biggl[ \int_0^T 
      \left( \mathcal{L}_{X}-\mathcal{L}_{\bar X} \right) u(\bar X_t, t)
      - \phi(\bar X_t,t) \ud t \Biggr]\\
    \end{aligned}
  \end{equation*}
  From the Kolmogorov backward equation \eqref{eq:kolmogorov} we have
  $\phi = 0$ for $\bar X_t \in \zset_+^d$ which gives Lemma
  \ref{lem:tau_exp_error_rep}.

\end{proof}

\subsection{Relative error}\label{sec:rel_error}
Assume that we have $x \in \zset_+^d$ particles and propensities
$a_j(x)$. 
Introduce the scaling $z:=x/\gamma$ where the scaling factor
$\gamma>0$ is related to the initial number of particles $X_0$, \eg{}
by taking the Euclidean norm $\gamma:=\|X_0\|$.  The relative error
can now be defined as
\begin{equation}
  \label{eq:rel_error}
  \varepsilon := \expt \bigl[ g(Z_T) - g(\bar Z_T) \bigr],
\end{equation}
where the scaled $Z_t$-process is represented by the Kolmogorov
backward equation
\begin{equation}
  \label{eq:kolmogorov_rescaled}
  \begin{aligned}
    \partial_t \tilde u + \sum_{j=1}^M \tilde a_j( z ) 
    \underbrace{
      \Bigl( \tilde u(z + \tilde \nu_j,t) - \tilde u(z,t) \Bigr) 
    }_{D_j \tilde u} &= 0, &&\text{ in } \zset_+^d/\gamma \times [0,T),\\
    \tilde u &= g, &&\text{ on } \zset_+^d/\gamma \times \{T\},\\
  \end{aligned}    
\end{equation}
and approximated by the rescaled tau-leap method
\begin{equation}
  \label{eq:tau_leap_rescaled}
  \bar Z_{t+\tau} = \bar Z_t + \sum_{j=1}^M \tilde \nu_j 
  \poisson_j( \tilde a_j( \bar Z_t )\tau ),
\end{equation}
with the scaled propensities $\tilde a_j(z):=a_j(x)=a_j(\gamma z)$ 
and jumps $\tilde \nu_j:=\nu_j/\gamma$. 

\begin{example}

  Let $g(x)=x^r:=\prod_{i=1}^d x_i^{r_i}$ and $|r| := \sum_{i=1}^d
  |r_i| \geq 1$ where $r\in\zset_+^d$ is a $d$-dimensional
  multi-index. Then
  \begin{equation*}
     \expt \bigl[ g(Z_T) - g(\bar Z_T) \bigr] = \frac{ \expt \bigl[
       g(X_T) - g(\bar X_T) \bigr] }{ g(X_0) },
  \end{equation*}
  for some $\min_i X_0^{(i)} \leq \gamma \leq \max_i X_0^{(i)}$.

\end{example}

\begin{assumption}\label{assu:basic_assumption}

  Following Assumption \ref{assu:boundedness} and Remark \ref{rem:propensity}
  we here assume that:
  \begin{enumerate}

  \item 
    The process $Z_t\in\zset_+^d/\gamma$ is bounded from above by a
    hyperplane passing through the point $Z_0$ and with normal $n>0$, \ie{}
    \begin{equation*}
      Z_t \in \tilde \Pi(Z_0,n):= \Pi(\gamma X_0,n)/\gamma = 
      \{z \in \zset_+^d/\gamma: n \cdot (z-Z_0) \leq 0 \}.
    \end{equation*}
    Consequently, the tau-leap process takes values in the larger set
    \begin{equation*}
      \bar Z_t \in \hat \Pi(Z_0,n) :=
      \{z \in \zset^d/\gamma: n \cdot (z-Z_0) \leq 0 \}.
    \end{equation*}
    
  \item \label{item:prop} The propensity $a_j:\zset^d_+\to\rset$ is
    non-negative and non-decreasing in each component.  Also, we have
    that $a_j(0)=0$, and $a_j(x)=0$ if $x+\nu_j \notin \zset_+^d$.
        
  \end{enumerate}

\end{assumption}

The goal is now to show an \emph{a priori} bound for the
discretization error $\varepsilon$ that is independent of the scaling
factor $\gamma$, see Theorem \ref{thm:rel_error_bnd}. The proof is
based on a Taylor expansion and bounds on the value function
$\tilde u$ and its derivatives. 

\begin{theorem}[Bound on relative error] \label{thm:rel_error_bnd}

  Let the process $Z_t$ and the propensities $a_j$ satisfy Assumption
  \ref{assu:basic_assumption}. Assume polynomial propensity functions
  of order $|p_j|$, with $p_j\in\zset_+^d$, and let $g \in
  C^2(\rset^d)$. Also, assume that the value function $\tilde u$
  given by \eqref{eq:kolmogorov_rescaled} and its first two spacial
  derivatives are bounded in $\tilde \Pi(Z_0,n)\times[0,T]$, independently of
  the scaling factor $\gamma$.


  For a time step $\tau=h\gamma^{-\delta}$, with $h>0$ and
  $\delta=\max_j (2|p_j|-2)$, the relative error \eqref{eq:rel_error}
  for the Poisson bridge tau-leap method is then bounded by
  \begin{equation*}
    \left| \varepsilon \right| \leq C h,
  \end{equation*}
  for some sufficiently large $\gamma$ and where the constant $C>0$ is
  independent of $h$ and the scaling factor $\gamma$.
  
\end{theorem}

\begin{proof}

  Throughout the proof $C$ will denote a non-negative constant value, not
  necessarily the same at each instance.  
  To show the theorem we note that from the error
  representation in Lemma \ref{lem:tau_exp_error_rep},  
  the relative error can be expressed as 
  \begin{equation}\label{eq:tau_exp_error_rep} 
    \begin{aligned}
      \varepsilon =& 
      \sum_{j=1}^M \sum_{n=0}^{N-1} \int_{t_{n}}^{t_{n+1}} 
      \underbrace{
        \expt \Bigl[
        \bigl( \tilde a_j( \bar{Z}_t ) - \tilde a_j( \bar{Z}_{t_n} ) \bigr)
        D_j \tilde u( \bar{Z}_t, t )
        \bigm| \bar{Z}_{t_n} \Bigr]
      }_{I_{j,n}}
      \ \ud t \\
      &- \sum_{n=0}^{N-1} \int_{t_{n}}^{t_{n+1}}  \expt \Bigl[ 
        \left( \Lambda_n \tilde u( \bar{Z}_t, t ) \right)
        \one_{\{ \bar  {Z}_t \in \zset_-^d/\gamma \}} 
        \bigm| \bar{Z}_{t_n} \Bigr]
        \ud t,
    \end{aligned}    
  \end{equation}
  where
  \begin{equation*}
    \Lambda_n := 
    \partial_t + \sum_{j=1}^M \tilde a_j(\bar{Z}_{t_n}) D_j.
  \end{equation*}

  For the first term in \eqref{eq:tau_exp_error_rep},
  Taylor expanding the propensity $\tilde a$ around $\bar{Z}_{t_n}$ gives
  \begin{equation*}
    I_{j,n} =
    \sum_{|\alpha_j| = 1}^{|p_j|} 
    \frac{ \tilde a_j^{(\alpha_j)}( \bar{Z}_{t_n} ) }{ \alpha_j! }
    \underbrace{
      \expt \Bigl[
      ( \bar{Z}_{t}-\bar{Z}_{t_n} )^{\alpha_j}      
      D_j \tilde u( \bar Z_t, t )
      \Bigm| \bar{Z}_{t_n} \Bigr],
    }_{\bar I_{\alpha_j,n}}
  \end{equation*}
  where $\alpha_j \in \zset_+^d$.

  According to Lemma \ref{lem:cont_ext} the value function and its two
  first spacial derivatives can be continuously extended to
  $\rset^d\times[0,T]$. The mean value theorem then yields
  \begin{equation*}
    \begin{aligned}
      \bar I_{\alpha_j,n}
      =&\
      \expt \biggl[
      ( \bar{Z}_{t}-\bar{Z}_{t_n} )^{\alpha_j}      
      \Bigl( \tilde u( \bar Z_t + \tilde \nu_j, t ) - 
      \tilde u( \bar{Z}_t, t ) \Bigr)
      \biggm| \bar{Z}_{t_n} \biggr]\\
      =&\
      \expt \biggl[
      ( \bar{Z}_{t}-\bar{Z}_{t_n} )^{\alpha_j}      
      \biggl( 
      \partial_z \tilde u( \bar{Z}_t, t ) \tilde \nu_j+ 
      \frac{1}{2} \tilde \nu_j^T \partial_z^2 \tilde u( \xi_1, t ) \tilde \nu_j
      \biggr)
      \biggm| \bar{Z}_{t_n} \biggr]\\
      =&\
      \expt \biggl[
      ( \bar{Z}_{t}-\bar{Z}_{t_n} )^{\alpha_j}      
      \partial_z \tilde u( \bar{Z}_{t_n}, s )\tilde \nu_j
      \biggm| \bar{Z}_{t_n} \biggr]\\
      &+ 
      \frac{1}{2}
      \expt \biggl[
      ( \bar{Z}_{t}-\bar{Z}_{t_n} )^{\alpha_j}      
      ( \bar{Z}_{t}-\bar{Z}_{t_n} )
      \partial_z^2 \tilde u( \xi_2, t ) \tilde \nu_j
      \biggm| \bar{Z}_{t_n} \biggr]\\
      &+ 
      \frac{1}{2}
      \expt \biggl[
      \tilde \nu_j^T
      ( \bar{Z}_{t}-\bar{Z}_{t_n} )^{\alpha_j}      
      \partial_z^2 \tilde u( \xi_1, t ) \tilde \nu_j
      \biggm| \bar{Z}_{t_n} \biggr]  \\
      =&\
      \bar I^\mathrm{I}_{\alpha_j,n} + 
      \bar I^\mathrm{II}_{\alpha_j,n} + 
      \bar I^\mathrm{III}_{\alpha_j,n},
    \end{aligned}    
  \end{equation*}
  for $\xi_1 = \alpha\bar{Z}_t + (1-\alpha)\left( \bar{Z}_t +\tilde\nu_j
  \right)$ for some $\alpha \in [0, 1]$ and $\xi_2 = \beta
  \bar{Z}_{t_n} + (1-\beta) \bar{Z}_{t}$ for some $\beta\in[0,1]$.
  
  From the tau-leap approximation of the process $Z_t$ we have
  \begin{equation*}
    \begin{aligned}
      \expt \Bigl[
      \left|\left( \bar{Z}_{t}-\bar{Z}_{t_n} \right)^{\alpha_j}\right|
      \Bigm| \bar{Z}_{t_n} \Bigr] 
      =&\
      \sum_{j=1}^M 
      \expt \Bigl[
      \poisson_j\left( \tilde a_j( \bar Z_{t_n} )(t-t_n) \right)^{|\alpha_j|}
      \bigl|\tilde \nu_j^{\alpha_j}\bigr| 
      \Bigm| \bar{Z}_{t_n} \Bigr]\\
      =&\
      \sum_{j=1}^M
      \sum_{\beta_j=1}^{|\alpha_j|}
      C_{\beta_j} \tilde a_j( \bar{Z}_{t_n})^{\beta_j} (t-t_n)^{\beta_j}
      \bigl|\tilde \nu_j^{\alpha_j}\bigr|\\
      \leq&\
      \sum_{j=1}^M
      \sum_{\beta_j=1}^{|\alpha_j|}
      C_{\beta_j} \tilde a_j( \bar{Z}_{t_n})^{\beta_j} \tau^{\beta_j}
      \bigl|\tilde \nu_j^{\alpha_j}\bigr|,
    \end{aligned}
  \end{equation*}
  for some positive constants $C_{\beta_j}$ that do not depend on
  $\gamma$.  Choosing $\tau=h\gamma^{-\delta}$ for $h,\delta>0$, and
  without loss of generality assuming that the propensity is a
  monomial $a_j(x) = Cx^{p_j}$ with $C>0$, $p_j\in\zset_+^d$ and
  $|p_j| \geq 1$, then
  \begin{equation*}
    \begin{aligned}
      \tilde a_j(\bar{Z}_{t_n})^{\beta_j} \tau^{\beta_j} \bigl|\tilde \nu_j^{\alpha_j}\bigr|
      =&\
      a_j(\gamma \bar{Z}_{t_n})^{\beta_j}
      \tau^{\beta_j} \bigl|\tilde \nu_j^{\alpha_j}\bigr|
      \\
      =&\
      a_j(\gamma \bar{Z}_{t_n})^{\beta_j} 
      h^{\beta_j} \bigl|\nu_j^{\alpha_j}\bigr| \gamma^{-\delta {\beta_j} - |\alpha_j|} \\
      \leq&\
      C \bar{Z}_{t_n}^{p_j {\beta_j}}
      h^{\beta_j} \bigl|\nu_j^{\alpha_j}\bigr| \gamma^{\beta_j (|p_j|-\delta)-|\alpha_j|}. \\ 
    \end{aligned}
  \end{equation*}

  Using the continuous extension in Remark \ref{rem:neg_ext} and
  assuming the value function $\tilde u$ and its first two derivatives
  are bounded independently of $\gamma$ inside $\tilde
  \Pi(Z_0,n)\times[0,T]$, \ie{} neglecting the probability terms in
  the growth condition in Lemma \ref{lem:growth}, leads to continuity
  and boundedness in the extended domain $\{z \in \rset^d: n \cdot
  (z-Z_0) \leq 0 \}$.
  Together with $\delta = \max_j (2|p_j|-2)$ and $h<1$, this gives
  \begin{equation*}
    \begin{aligned}
      \left|\sum_{|\alpha_j| = 1}^{|p_j|} 
      \frac{ \tilde a_j^{(\alpha_j)}( \bar{Z}_{t_n} ) }{ \alpha_j! }
      \bar I^\mathrm{I}_{\alpha_j,n} \right|
      &\leq
      C \sum_{|\alpha_j|=1}^{|p_j|} \sum_{\beta_j=1}^{|\alpha_j|} 
      \gamma^{|p_j|-|\alpha_j|+\beta_j(|p_j|-\delta)-1} h^{\beta_j}\\
      &\leq
      C \sum_{|\alpha_j|=1}^{|p_j|} \sum_{\beta_j=1}^{|\alpha_j|} 
      \gamma^{\max_j(|p_j|)-2+\beta_j(\max_j(|p_j|)-\delta)} h\\      
      &\leq
      C \sum_{|\alpha_j|=1}^{|p_j|} \sum_{\beta_j=1}^{|\alpha_j|} 
      \gamma^{\max_j(|p_j|)-2+(2-\max_j(|p_j|))} h\\      
      &\leq
      C h,\\      
    \end{aligned}
  \end{equation*}
  for $\bar{Z}_{t_n} \in \tilde \Pi(Z_0,n)$.  

  Performing the same analysis on the error contribution from the
  terms $\bar I^\mathrm{II}_{\alpha_j,n}$ and $\bar
  I^\mathrm{III}_{\alpha_j,n}$ gives the estimate in Theorem
  \ref{thm:rel_error_bnd}. Note that the additional $(\bar Z_t - \bar
  Z_{t_n})$ term in $\bar I^\mathrm{II}_{\alpha_j,n}$ and the extra
  $\tilde \nu_j$ in $\bar I^\mathrm{III}_{\alpha_j,n}$ yields that
  $\bar I^\mathrm{I}_{\alpha_j,n}=\Ordo{1}$, $\bar
  I^\mathrm{II}_{\alpha_j,n}=\Ordo{\gamma^{-1}}$ and $\bar
  I^\mathrm{III}_{\alpha_j,n}=\Ordo{\gamma^{-1}}$, so for large
  $\gamma$ the first term in the Taylor expansion is dominant.

  For the second term in \eqref{eq:tau_exp_error_rep} we have that
  since the value function and its derivatives are continuously
  extended and bounded in $\hat \Pi(Z_0,n)$
  \begin{equation*}
    \Lambda_n \tilde u(\bar{Z}_t, t) = \left( \partial_t +
      \sum_{j=1}^M \tilde a_j(\bar{Z}_{t_n}) D_j \right) \tilde
    u(\bar{Z}_t, t) \leq  \gamma^{-1} \sum_{j=1}^M C_j \tilde a_j(\bar{Z}_{t_n}).
  \end{equation*}
  and thus
  \begin{equation}\label{eq:remainder}
    \begin{aligned}
      \expt \left[ \Lambda_n \tilde u(
        \bar{Z}_t, t ) \one_{\{ \bar {Z}_t \in \zset_-^d/\gamma \}}
        \bigm| \bar Z_{t_n}\right]
      &\leq
      \gamma^{-1}
      \expt \left[ \one_{\{ \bar {Z}_t \in \zset_-^d/\gamma \}}\bigm| \bar Z_{t_n}
      \right]
      \sum_{j=1}^M C_j \tilde a_j(\bar{Z}_{t_n})\\
      &= 
      \gamma^{-1}
      \prob \left( \bar {Z}_t \in \zset_-^d/\gamma\bigm| \bar Z_{t_n}
      \right) \sum_{j=1}^M C_j \tilde a_j(\bar{Z}_{t_n})
      \\
      &\leq
      \gamma^{-1}
      \sum_{i=1}^d \prob \left( \bar {Z}_t^{(i)} < 0 \bigm| \bar
        Z_{t_n}\right)
      \sum_{j=1}^M C_j \tilde a_j(\bar{Z}_{t_n}).      
    \end{aligned}
  \end{equation}

  Neglecting positive jumps we obtain
  \begin{equation*}
    \begin{aligned}
      \bar Z_{t}^{(i)}
      &= \bar Z_{t_n}^{(i)} + \sum_{j=1}^M \tilde \nu_j^{(i)} 
      \poisson_j( \tilde a_j( \bar Z_{t_n} )(t-t_n) )\\
      &\geq
      \bar Z_{t_n}^{(i)} + \tilde \nu^{(i)}
      \poisson( \tilde a_0( \bar Z_{t_n} )\tau ) =: \hat Z_t^{(i)}, \quad t\in(t_n,t_{n+1}),
    \end{aligned}
  \end{equation*}
  where $\tilde \nu^{(i)} := \min_j \{\min(\tilde \nu_j^{(i)},0)\}
  <0$, and $\tilde a_0:=\sum_{j=1}^M \tilde a_j$.  For $\bar
  {Z}_{t_n}^{(i)}>0$ the probability of negative populations in
  \eqref{eq:remainder} can then be approximated by
  \begin{equation}
    \label{eq:prob}
    \begin{aligned}
      \prob \left( \bar {Z}_t^{(i)} <0 \bigm| \bar Z_{t_n} \right)
      &\leq
      \prob \left( \hat {Z}_t^{(i)} <0 \bigm| \bar Z_{t_n} \right)\\
      &\leq \prob \left( \poisson( \tilde a_0( \bar Z_{t_n} )\tau ) >
        \underbrace{-\bar Z_{t_n}^{(i)}/\tilde\nu^{(i)}}_{=:q>0}
        \bigm| \bar
        Z_{t_n} \right)\\
      &= \sum_{k=\lceil q \rceil}^\infty \frac{\left(\tilde a_0( \bar
          Z_{t_n} )\tau\right)^k}{k!}
      e^{-\tilde a_0( \bar Z_{t_n} )\tau}\\
      &= \sum_{k=0}^\infty \frac{\left(\tilde a_0( \bar Z_{t_n}
          )\tau\right)^{k+\lceil q \rceil}}{\left(k+\lceil q \rceil\right)!}
      e^{-\tilde a_0( \bar Z_{t_n} )\tau}\\
      &< \sum_{k=0}^\infty \frac{\left(\tilde a_0( \bar Z_{t_n}
          )\tau\right)^{\lceil q \rceil}}{\lceil q \rceil !}
      \frac{\left(\tilde a_0( \bar Z_{t_n} )\tau\right)^k}{k!}
      e^{-\tilde a_0( \bar Z_{t_n} )\tau}\\
      &\leq \frac{\left(\tilde a_0( \bar Z_{t_n}
          )\tau\right)^{C\gamma}}{(C\gamma)!},
    \end{aligned}
  \end{equation}
  where $\lceil\cdot\rceil$ denotes the ceiling and $C$ is independent
  of $\gamma$. For $\bar {Z}_{t_n}^{(i)}=0$, the probability $\prob
  \left( \bar {Z}_t^{(i)} <0 \bigm| \bar Z_{t_n} \right)$ is zero
  since $a_j(\bar {Z}_{t_n})=0$ if $\tilde \nu_j^{(i)}<0$.
  
  Combining \eqref{eq:remainder} and \eqref{eq:prob}, and using
  $a_0(x) \leq C x^{\max_j |p_j|}$ and $\tau=h\gamma^{2-2\max_j|p_j|}$ gives
  \begin{equation}
    \begin{aligned}
      \expt \left[ \Lambda_n \tilde u(
        \bar{Z}_t, t ) \one_{\{ \bar {Z}_t \in \zset_-^d/\gamma \}} \right]
      &\leq
      C \gamma^{-1}
      \tilde a_0(\bar{Z}_{t_n})
      \frac{\left(a_0( \bar Z_{t_n}
          )\tau\right)^{C\gamma}}{(C\gamma)!}\\
      &\leq
      C \gamma^{\max_j |p_j|-1}
      \frac{\left( C h \gamma^{2-\max_j |p_j|}
        \right)^{C\gamma}}{(C\gamma)!}\\
      &\approx
      C \gamma^{\max_j |p_j|-1}
      \left( C h \gamma^{1-\max_j |p_j|} \right)^{C\gamma},
    \end{aligned}
  \end{equation}
  so the remainder term \eqref{eq:remainder} can be neglected.

\end{proof}


In the theorem above we used that the relative value function $\tilde
u$ defined by \eqref{eq:kolmogorov_rescaled} and its derivatives can
be continuously extended to $\rset^d \times [0,T]$, and are bounded
independently of $\gamma$ in the domain $\hat \Pi(Z_0,n)\times[0,T]$.
In the remaining part of this section we will motivate why those
assumptions make sense.

\begin{lemma}[Extension of the value function and its derivatives]\label{lem:cont_ext}

  Assume that the process $Z_t$ and the propensities $a_j$ satisfy
  Assumption \ref{assu:basic_assumption}, and that $a_j$ and $g$ are
  $C^2$ on $\rset^d$.  The solution $\tilde u$ to the Kolmogorov
  backward equation \eqref{eq:kolmogorov_rescaled} is then locally
  bounded and smooth in time on $\zset_+^d/\gamma \times [0,T]$ and
  can be continuously extended to $\rset^d \times [0,T]$. Also, the
  derivatives $\partial_z \tilde u$ and $\partial_z^2 \tilde u$ are
  locally bounded and smooth in time on the lattice $\zset_+^d/\gamma$
  and can be continuously extended to $\rset^d \times [0,T]$.

\end{lemma}


\begin{proof}

  First, we continuously extend $\tilde a_j$ such that it is
  non-negative, monotone in each component and $C^2$ on $\rset^d$, see
  \eg{} Example \ref{ex:propensity_extension1} and
  \ref{ex:propensity_extension2} below.  Given such an extension of
  the propensity, the value function $\tilde u$ can be continuously
  extended to the domain $\rset^d \times [0,T]$ by solving the
  Kolmogorov backward equation \eqref{eq:kolmogorov_rescaled} on the
  shifted lattice $(\zset_+^d+\epsilon)/\gamma \times [0,T]$,
  $\epsilon \in \rset^d$.

  Note that, for any point $Z_t\in\zset^d_+/\gamma$ we know that
  $Z_s\in\tilde \Pi(Z_t,n)$ for $s \geq t$, and
  \eqref{eq:kolmogorov_rescaled} can be understood as a linear constant
  coefficient system of ordinary differential equations on a finite
  lattice, with a unique bounded smooth solution $\tilde u$.


  

  The derivatives $\tilde v:=\partial_z \tilde u \in \rset^d$ and 
  $\tilde w:=\partial^2_z \tilde u \in \rset^{d \times d}$ can be defined on 
  $\zset_+^d/\gamma \times [0,T]$ and satisfy the equations
  \begin{equation}
    \label{eq:kolmogorov_rescaled2}
    \begin{aligned}
      \partial_t \tilde v + \sum_{j=1}^M \tilde a_j(z) D_j \tilde v
        &= -\sum_{j=1}^M \nabla \tilde a_j(z) D_j \tilde u, &&\text{ in }
        \zset_+^d/\gamma \times [0,T),\\ 
      \tilde v &= \nabla g, &&\text{ in } \zset_+^d/\gamma \times \{T\},\\
    \end{aligned}
  \end{equation}
  and
  \begin{equation}
    \label{eq:kolmogorov_rescaled3}
    \begin{aligned}
      \partial_t \tilde w + \sum_{j=1}^M \tilde a_j(z) D_j \tilde w &=\\ 
      \sum_{j=1}^M \Bigl(
      -2 \nabla \tilde a_j(z) &\cdot D_j \tilde v - \nabla^2 \tilde a_j(z) D_j \tilde u
      \Bigr), &&\text{ in } \zset_+^d/\gamma \times [0,T),\\
      \tilde w &= \nabla^2 g,&&\text{ in } \zset_+^d/\gamma \times \{T\}.\\
    \end{aligned}
  \end{equation}
  Using that $g\in C^2(\rset^d)$ and that $\tilde u$ is locally bounded and
  smooth, 
  we can by the same argument as for $\tilde u$ above see that $\tilde
  v$ and $\tilde w$ are locally bounded and smooth in time.

  By shifting the grid in the same manner as
  was done for $\tilde u$, the solutions $\tilde v$ and $\tilde w$ can
  be continuously extended to
  $\rset^d\times[0,T]$.

\end{proof}

\begin{remark}\label{rem:neg_ext}
  
  In addition to the above lemma we can also continuously extend the
  value function $\tilde u(z,\cdot)$ to real negative values $z\in\rset_-^d$ such that
  $\tilde u, \partial_z \tilde u$ and $\partial_z^2 \tilde u$ vanish
  for $z_i<-l<0$. This extension is used for the Taylor expansion in
  Lemma \ref{thm:rel_error_bnd}.

  Without loss of generality consider the
  extension to the domain $\rset_-^2$, with subdomains $A:=\{z \in
  (-l,0)\times(0,\infty)\}$, $B:=\{z\in (-l,0)^2 \}$ and $C:=\{z \in
  (0,\infty)\times(-l,0)\}$.

  First let $\tilde u(z,\cdot)=0$ on $z\in\rset_-^2\setminus \{A\cup
  B\cup C\}$. Next, extend the value function to the domain $A$
  by third degree polynomials in the $z_1$ direction, for all
  $z_2\geq0$, matching $\tilde u$, $\partial_{z_1} \tilde u$ and
  $\partial_{z_1}^2 \tilde u$ on $z_1=-L$ and $z_1=0$. Similarly, fit
  polynomials in the $z_2$ direction in domain $C$. In the domain $B$
  a $C^2$-extension can be made by transfinite interpolation using
  $\tilde u$, $\partial_{z}$ and $\partial_{z}^2 \tilde u$ on
  $\partial B$, see \eg{} \cite{barnhill, gordon, worsey2}.

\end{remark}


\begin{example}\label{ex:propensity_extension1}

  A natural extension to the propensity function $a_j(x):=x$, $x\in\zset$, is
  \begin{equation*}
    a_j(x) :=  
    \begin{cases}
      0, & x < 0,\\
      \hat a_j(x), & 0 \leq x \leq 1,\\
      x, & x > 1,
    \end{cases}
  \end{equation*}
  for $x\in\rset$, where $\hat a_j$ is a positive and
  monotone $C^2$ function chosen such that $a_j\in C^2(\rset)$, \eg{}
  $\hat a_j(x) := 6x^3 - 8x^4 + 3x^5$, see Figure \ref{fig:propensity}.

\end{example}

\begin{example}\label{ex:propensity_extension2}

  The propensity function $a(x):=x(x-1)$ is negative in $(0,1)$. A
  natural extension is thus to let 
  \begin{equation*}
    a(x) :=
    \begin{cases}
      0, & x < 1,\\
      \hat a(x), & 1 \leq x < 2,\\
      x(x-1), & x \geq 2,
    \end{cases}
  \end{equation*}
  for $x\in\rset$, where $\hat a_j$ is a positive and
  monotone $C^2$ function chosen such that $a_j\in C^2(\rset)$, \eg{}
  $\hat a_j(x) := 9(x-1)^3 - 11(x-1)^4 + 4(x-1)^5$, see Figure \ref{fig:propensity}.

\end{example}

\begin{figure}[hbpt]
  \centering
  \subfigure[Example \ref{ex:propensity_extension1}]{
    \includegraphics*[width=0.45\textwidth,viewport=80 200 540 600]
    {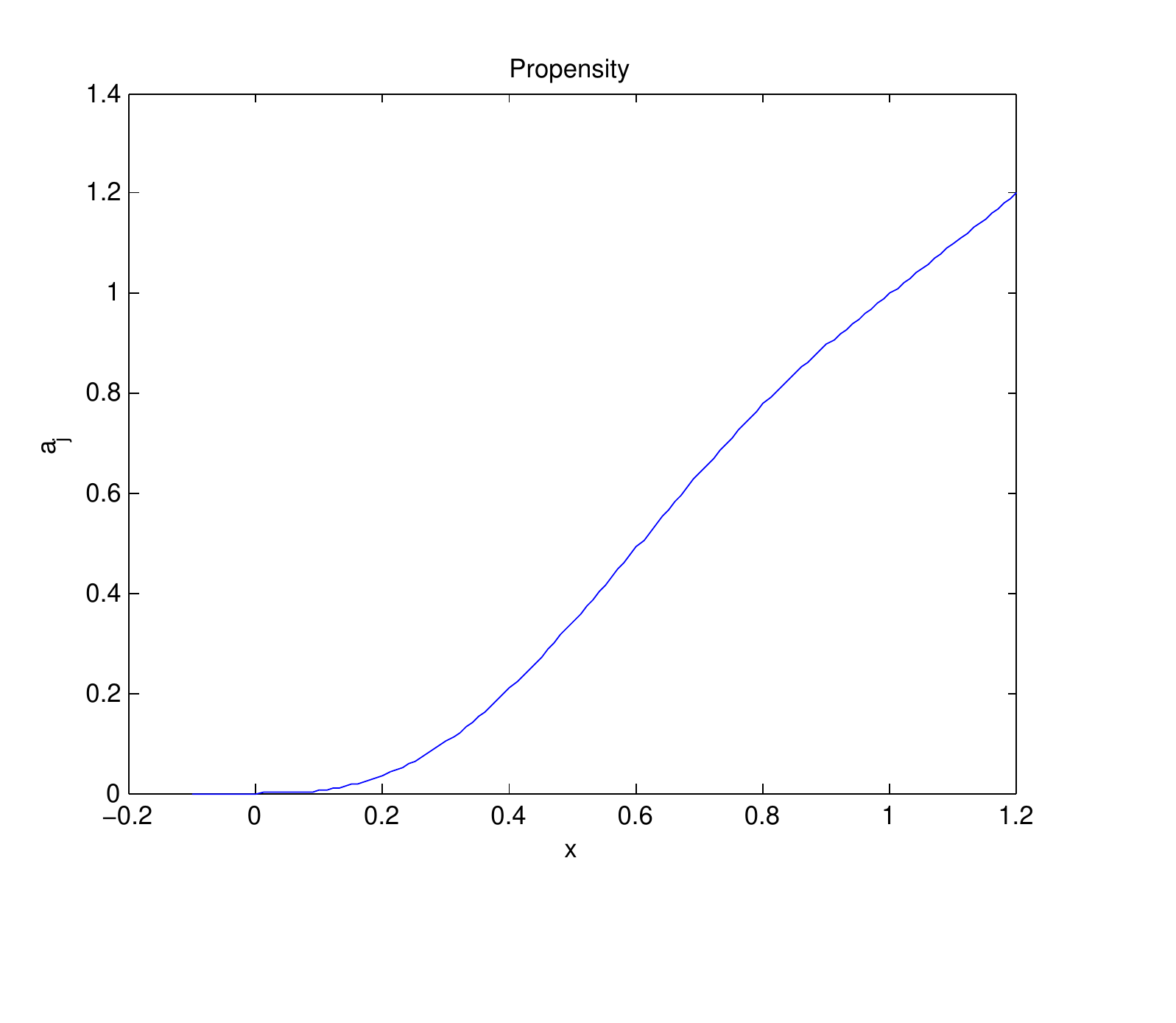}}
  \subfigure[Example \ref{ex:propensity_extension2}]{
    \includegraphics*[width=0.45\textwidth,viewport=80 200 540 600]
    {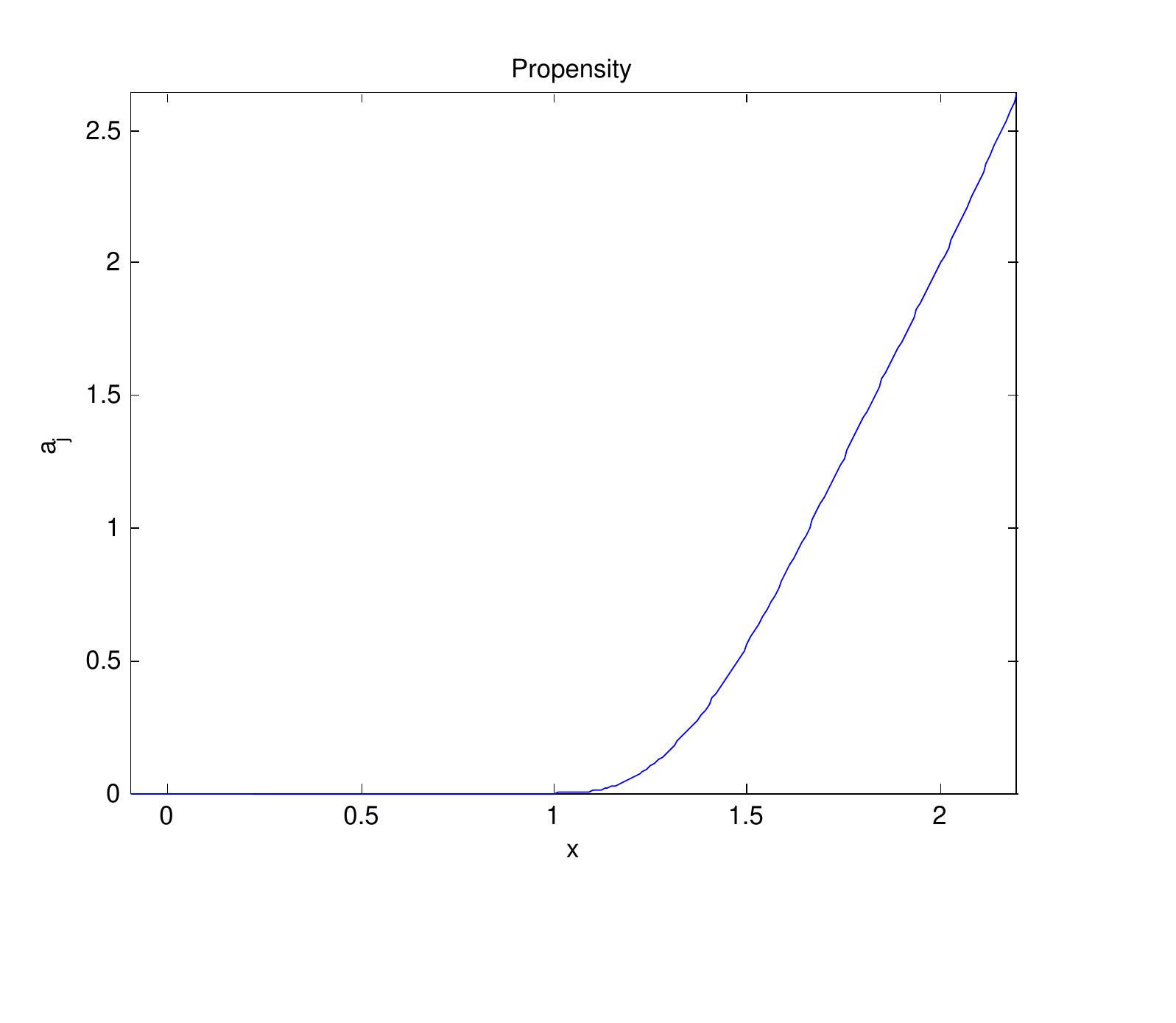}}
  \caption{Extended propensity functions in Example
    \ref{ex:propensity_extension1} and \ref{ex:propensity_extension2}.}
  \label{fig:propensity}
\end{figure}

In Lemma \ref{lem:cont_ext} we saw that $\tilde u$ and its first
two derivatives are locally bounded in $\rset^d \times
[0,T]$, however, this bound depends on $\gamma$.
The next step is to show under what conditions we can find a bound
that does not depend on $\gamma$. First, we start by generalizing the
polynomial form of the propensity \eqref{eq:propensity} in Remark
\ref{rem:propensity} into bounds:

\begin{assumption}\label{assu:prop_bnd}

  Given some multi-index $p_j \in \zset_+^d$ and
  positive constants $0 < C_L^i \leq C_U^i$, $0 < C_L' \leq C_U'$ and $0
  < C_L'' \leq C_U''$,
  assume that the propensity 
  has the bounds
  \begin{equation}
    \label{eq:prop_bnd}
    \begin{aligned}
      C_L (x+\zeta_j)_+^{p_j} &\leq a_j(x) \leq C_{U} (x+\zeta_j)_+^{p_j}, \\
      C_L' (x+\zeta_j)_+^{p_j-e_i} &\leq \partial_{x_i} a_j(x) \leq C_{U}' (x+\zeta_j)_+^{p_j-e_i}, \\
      C_L'' (x+\zeta_j)_+^{p_j-e_i-e_k} &\leq \partial_{x_{ik}} a_j(x)
      \leq C_{U}'' (x+\zeta_j)_+^{p_j-e_i-e_k},\\
    \end{aligned}
  \end{equation}
  for $i=1,\ldots,d$, $j=1,\ldots,M$ and $x\in\zset^d_+$.  Here,
  $(\cdot)_+:=\max(\cdot,0)$,
  $\zeta_{ji}:=(\nu_{ij}+1)\one_{\nu_{ij}<0}$, and $e_i$ indicate the
  unit basis vectors in $\rset^d$.
  
\end{assumption}

In the next step, we introduce a weight definition that in Lemma
\ref{lem:growth} will allow us to do find weighted estimates for
the value function and its derivatives.

\begin{remark}[Weight definition]\label{rem:weight_def}
  From the Assumption \ref{assu:prop_bnd} on boundedness of $Z_t$,
  there exists a vector $n>0$ in $\rset^d$ such that $\nu_j \cdot n <
  0$ for $j=1,\ldots,M$.  Take a weight function $y(z):=\psi(n \cdot
  z)>0$ where $\psi:(t_{\min},+\infty) \to \rset^+$ is a smooth,
  strictly decreasing function.  Observe that we have by construction
  $y(z+\nu_j) - y(z) \ge 0$, for all $z \in \tilde \Pi(Z_0,n)$.
  In particular, if $g$ has polynomial growth $g(z) \le C (z\cdot
  n)^{r_0}(z\cdot n+ 1)^{r_1}$ we can have the product $yg$
  uniformly bounded, for instance by taking $\psi(t) = (t +t_0)^{-r_0}
  (t + 1)^{-r_1}$, with some $-t_{\min} = t_0>0$ and $r_0,r_1\ge 0.$
\end{remark}

Now we are ready to present weighted estimates that are uniform with
respect to the scaling parameter $\gamma$, and on which Theorem
\ref{thm:rel_error_bnd} for the a priori relative error bound is
based.

\begin{lemma}[Growth condition]\label{lem:growth}
  
  Assume that the assumptions on boundedness of $Z_t$ and the conditions
  on $a_j$ in Assumption \ref{assu:basic_assumption} hold.
  Also, assume that the bounds \eqref{eq:prop_bnd} on the propensity
  and its derivatives in Assumption \ref{assu:prop_bnd} are satisfied.
  Let $|p|$ be the maximum reaction order in the system, \ie{}
  $|p|:=\max_j |p_j|$ for the multi-indices in \eqref{eq:prop_bnd},
  and assume that for each of the components $x_i,$ there exist a
  single component reaction whose propensity has order $|p|$.
  Finally, assume that $g$ is $C^2$ on $\rset^d$, non-negative
  with $g(0)=0$, and that
  for the weight functions
  \begin{equation*}
    \begin{aligned}
      y(z) := &\
      \left( z \cdot n + \frac{1}{\gamma}\max_j |\nu_j\cdot n| \right)^{-r_0}
      \left( 1+ z \cdot n\right)^{- r_1},\\
      y_1(z) := &\ y(z) \left( z \cdot n + \frac{1}{\gamma}\max_j | \nu_j\cdot n| \right),\\
      y_2(z) := &\ y_1(z) \left( z \cdot n + \frac{1}{\gamma}\max_j | \nu_j\cdot n| \right),\\
    \end{aligned}
  \end{equation*}
  the quantities $|yg|$, $|y_1\nabla g|$ and $|y_2\nabla^2 g|$ are bounded.

  We then have the bounds
  \begin{equation}
    \label{eq:order}
    \begin{aligned}
      |\tilde u(z,t)| 
      \leq&\ \frac{C}{y(z)},\\
      |\partial_z \tilde u(z,t)| 
      \leq&\ 
      \frac{C}{y_1(z)}+\frac{C'}{y_1(z)}
      (\gamma z\cdot n)^{|p|}
      \prob\left(
        \frac{\theta}{\gamma}< Z_T \cdot n<\eta  \biggm| Z_t=z 
      \right),\\
      |\partial^2_z \tilde u(z,t)| 
      \leq&\ 
      \frac{C}{y_2(z)}
      +\frac{C''}{y_2(z)}
      (\gamma z\cdot n)^{2|p|}
      \prob\left(
        \frac{\theta}{\gamma}< Z_T \cdot n<\eta  \biggm| Z_t=z 
      \right)^2,\\
    \end{aligned}
  \end{equation}
  for $(z,t) \in \tilde\Pi(Z_0,n) \times [0,T]$
  , for some $\theta,\eta>0$ and positive constants 
  $C$, $C'$ and $C''$ that are independent of $\gamma$.
  
\end{lemma}

\begin{proof}

  Throughout the proof $C$ will denote a constant value, not
  necessarily the same at each instance. Also, for simplicity we will
  assume that the propensity has the form $a_j(x):=Cx^{p_j}$. In the
  general case with bounds \eqref{eq:prop_bnd}, the proof is
  essentially the same. The proof is divided into three sections
  corresponding to the bounds in \eqref{eq:order}.
  
  \subsection*{Bound on $\tilde u$:}

  Let $U(z,t) := y(z) \tilde u(z,t)$.
  Since $g$ has polynomial growth we can take $y$ as in Remark \ref{rem:weight_def}
  with 
  \begin{equation*}
    t_0 = \max_j |\tilde \nu_j\cdot n| := |\tilde \nu^*\cdot n|.
  \end{equation*}
  Also, since $g(0) = 0$ we have $r_0>0$.
  Then we have
  \begin{equation*}
    \begin{aligned}
      \partial_t U + \sum_{j=1}^M \tilde a_j
        \Bigl(
        y \tilde u(z + \tilde\nu_j,t) - U
        \Bigr)
        &= 0, &&\text{ in } \tilde\Pi(Z_0,n) \times [0,T),\\
        U &= yg, &&\text{ on } \tilde\Pi(Z_0,n) \times \{T\},\\
    \end{aligned}    
  \end{equation*}  
  where $\| U(\cdot,T) \|_\infty = \| y \tilde u(\cdot,T) \|_\infty =
  \| y g \|_\infty$ by definition is bounded.  Adding and subtracting
  $\sum_{j=1}^M \tilde a_j( z ) U(z + \tilde\nu_j,t)$ yields
  \begin{equation}
    \label{eq:kolmogorov_modified}
    \begin{aligned}
      \partial_t U +\sum_{j=1}^M \tilde a_j D_j U &= \sum_{j=1}^M
      \tilde a_j \tilde u(z + \tilde\nu_j,t) D_j y,
    \end{aligned}    
  \end{equation}  
  with the stochastic representation
  \begin{equation*}
    U(z,t) = \expt\left[ y(Z_T)g(Z_T) - 
      \int_t^T \sum_{j=1}^M   \tilde a_j(Z_s) \tilde u(Z_s + \tilde\nu_j,s) 
      D_j y(Z_s) \Biggm| Z_t=z \, \ud s \right].
  \end{equation*}
  Since the right hand side of \eqref{eq:kolmogorov_modified} is
  non-negative we obtain
  \begin{equation*}
    \|U(\cdot,t)\|_\infty \leq \| U(\cdot,T) \|_\infty \leq C,
  \end{equation*}
  and therefore
  \begin{equation*}
    \tilde u(z,t) \leq C y(z)^{-1} \le C \left( z \cdot n +
      \frac{|\nu^*\cdot n|}{\gamma}\right)^{r_0} \left( 1+ z \cdot
      n\right)^{r_1},
  \end{equation*}
  which is the first estimate in \eqref{eq:order}.
   
  \subsection*{Bound on $\partial \tilde u$:}

  For the bound on the first derivative $\tilde v:=\partial_z \tilde u
  \in \rset^d$, 
  we consider the weighted function $V(z,t) := y_1(z) \tilde v(z,t) \in \rset^d$
  with the weight 
  \begin{equation*}
    y_1(z) := y(z) \left( z \cdot n + \frac{|\nu^*\cdot n|}{\gamma}\right)
    = \left( z \cdot n + \frac{|\nu^*\cdot n|}{\gamma}\right)^{1-r_0}
    \left( 1+ z \cdot n\right)^{- r_1} \in \rset,
  \end{equation*}
  such that the product $y_1 \nabla g$ is bounded uniformly in $z$ and $\gamma$.  
  Define an auxiliary function
  \begin{equation*}
    R(z,t) =  \sum_{j=1}^M      \tilde a_j(z) \tilde v(z + \tilde\nu_j,t) D_j y_1(z).
  \end{equation*}
  This yields
  \begin{equation*}
    \begin{aligned}
      \partial_t V + \sum_{j=1}^M  \tilde a_j D_j V 
      &=
      R(z,t) -
     \sum_{j=1}^M \nabla\tilde a_j(z) y_1(z) D_j \tilde u(z,t),
      &&\text{ in } \tilde \Pi(Z_0,n) \times [0,T),\\
      V &= y_1 \nabla g, &&\text{ on } \tilde\Pi(Z_0,n) \times \{T\}.\\
    \end{aligned}    
  \end{equation*}  
  Denote 
  \begin{equation}\label{eq:beta_def}
  \aligned
    \beta_j(z) :=&  \tilde a_j(z) D_jy_1(z)
    y_1(z+\tilde\nu_j)^{-1},  \\
    \beta(z) :=&  \sum_{j=1}^M \beta_j(z),
    \endaligned
  \end{equation}
  and observe that $\beta_j(0) =0$ and $\beta_j \ge 0$.
  Since
    \begin{equation*}
    R(z,t) = \sum_{j=1}^M V(z+\tilde\nu_j,t) \beta_j(z),
  \end{equation*}
  we have
  \begin{equation*}
    \begin{aligned}
      \partial_t V +\sum_{j=1}^M  (\tilde a_j - \beta_j) D_j V -\beta V
      &=
      -y_1\,\sum_{j=1}^M  \nabla \tilde a_j D_j \tilde u,
      &&\text{ in } \tilde\Pi(Z_0,n) \times [0,T),\\
      V &= y_1 \nabla g, &&\text{ on } \tilde\Pi(Z_0,n) \times \{T\}.\\
    \end{aligned}    
  \end{equation*}  
  The solution $V$ can now be stochastically represented as
  \begin{equation}
    \label{eq:Vint}
    \begin{aligned}
      V(z,t) = &\
      \underbrace{
        \expt \Bigl[ y_1(\hat Z_T) \nabla g(\hat Z_T) e^{-\int_t^T \beta(\hat Z_{s}) \ud s}
        \Bigm| \hat Z_t=z \Bigr]
      }_{= V_1}\\
      &\ + 
      \underbrace{
        \expt \biggl[ \int_t^T y_1(\hat Z_s) \,e^{-\int_s^T \beta(\hat
          Z_{t'}) \ud t'}\,  \sum_{j=1}^M \,\nabla\tilde a_j(\hat Z_s)  D_j
        \tilde u(\hat Z_s,s)  
        \ud s  \biggm| \hat Z_t=z \biggr]
      }_{= V_2},
    \end{aligned}
  \end{equation}
  where $\hat Z_t$ is a modified jump process with smaller propensities
  \begin{equation*}
    \hat a_j(z) := (\tilde a_j-\beta_j)(z) = \tilde
    a_j(z)\Bigl(1-D_jy_1(z)y_1(z+\tilde\nu_j)^{-1}\Bigr)
    = \tilde a_j(z)y_1(z)y_1(z+\tilde\nu_j)^{-1}.
  \end{equation*}
  Indeed, since $y_1(z+\tilde\nu_j) \geq y_1(z) \geq 0$ by definition, we have $0 \leq \hat a_j(z)
  \leq \tilde a_j(z)$.

  Observe that since the product $|y_1\nabla g|$ is bounded, we have $|V_1|\le
  C$. We now consider $V_2$, and first focus on a lower bound for the
  functions $\beta_j$ defined in \eqref{eq:beta_def}. From the
  assumption $a_j(x):=Cx^{p_j}$ we obtain
    \begin{equation*}
    \beta_j(z) = C \gamma^{|p_j|} z^{p_j} \xi_j(z),
  \end{equation*}
  and
  \begin{equation*}
    \begin{aligned}
      \xi_j(z\cdot n) :=&\ D_jy_1(z)y_1(z+\tilde\nu_j)^{-1}\\
       =&\ 1 - \left(1 -
        \frac{|\tilde \nu_j\cdot n|}{z\cdot n + |\tilde \nu^*\cdot
          n|}\right)^{r_0-1} \left(1 -
        \frac{|\tilde \nu_j \cdot n|}{1+z\cdot n}\right)^{r_1}.
    \end{aligned}
  \end{equation*}
  Observe that $\xi_j$ is strictly decreasing with respect to the
  product $z\cdot n$ on
  $(0,+\infty)$, 
  and that $\xi(+\infty) = 0^+$.  Since the values
  $z\in\tilde\Pi(Z_0,n)$ and $\tilde \nu$ are bounded by assumption
  \ref{assu:boundedness}, we conclude that $\xi_j(z)\ge \xi_0>0$ and
  then
  \begin{equation*}
    \beta_j(z) \ge C \xi_0 \gamma^{|p_j|} z^{p_j},
  \end{equation*}
  for all $z\in\tilde\Pi(Z_0,n)$.
  
  Given $\eta>0$, let us now introduce a family of time intervals,
  \begin{equation*}
    I_\eta^t := \{ s\in[t,T] : \hat Z_s\cdot n \geq \eta\}.
  \end{equation*}
  Let   
  \begin{equation*}
    |p| = \max_j |p_j|,
  \end{equation*}
  and observe that, thanks to the assumption on the existence of all
  the single component reactions with order $|p|$,
  \begin{equation}\label{eq:lower_bound_intbeta}
    \int_t^T \beta(\hat Z_s) ds 
    \ge  \xi_0 C \, |I_\eta^t| (\gamma\eta)^{|p|}.
  \end{equation}
  On the other hand, we have 
  \begin{equation}\label{eq;bound_forcing}
    \begin{aligned}
      | \partial_{z_i} \tilde a_j(z) y_1(z) D_j
      \tilde u(z,s) | \le & C |  \partial_{z_i} a_j(\gamma z) y_1(z) y(z)^{-1}| \\
      = & C \gamma^{|p_j|} z^{p_j-e_i}
      \left( z \cdot n + \frac{|\nu^* \cdot n|}{\gamma}\right)\\
      \le & C \gamma^{|p_j|} ( z \cdot n )^{{|p_j|}-1}  \left( z
        \cdot n + \frac{|\nu^* \cdot n|}{\gamma}\right).
    \end{aligned}
  \end{equation}
  Let $0<\kappa <1$ denote a time fraction.  Then thanks to \eqref
  {eq:lower_bound_intbeta} and \eqref{eq;bound_forcing} the term $V_2$
  in \eqref{eq:Vint} can be estimated as
  \begin{equation*}
    \begin{aligned}
      |V_2(z,t)| 
      \leq &\
      C\, \, \int_t^T  \expt \biggl[ \gamma^{|p|} (\hat Z_s \cdot n)^{|p|}
      e^{- \xi_0 C\, \,(\gamma\eta)^{|p|} \,
        |I_\eta^s|   }    \biggm| \hat Z_t=z \biggr]\ud s\\
      \leq &\ 
      C\, \gamma^{|p|}\, \int_t^T 
      e^{- \xi_0\, C\, \,(\gamma\eta)^{|p|} \,
        \kappa (T-s) }  \ud s \\
      & +
      C\, \gamma^{|p|} \expt \biggl[  \int_t^T  ( \hat Z_s \cdot n)^{|p|} 
      1_{
        |I_\eta^s| < \kappa (T-s)  }    \ud s  \biggm| \hat Z_t=z \biggr] \\
      \leq  &   
      C\, \frac{1-e^{- \xi_0 C\, \,(\gamma\eta)^{|p|} \, 
          \kappa (T-t) }}{ \xi_0 C\,\eta^{|p|} \, 
        \kappa } +
      C\, \gamma^{|p|}  \int_t^T \prob( |I_\eta^s| < \kappa (T-s) | \hat Z_t=z  ) \ud s.
    \end{aligned}
  \end{equation*}
  Observe, that since the process $\hat Z_s\cdot n$ is non increasing, we have 
  \begin{equation*}  1_{
      |I_\eta^{s_1}| < \kappa (T-s_1)  } \le 1_{
      |I_\eta^{s_2}| < \kappa (T-s_2)  }\le 1_{
      \hat Z_T \cdot n < \eta  }, \text{  for }s_1 \le s_2<T,
  \end{equation*}
  and then, for a given constant $\theta>1$,
  \begin{equation}
    \label{eq:bnd_non_incr}
    \begin{aligned}
      \gamma^{|p|} & \expt \biggl[  \int_t^T  (\hat Z_s \cdot n)^{|p|}
      1_{
        |I_\eta^s| < \kappa (T-s)  }    \ud s  \biggm| \hat Z_t=z \biggr] \\
      &\le 
      \gamma^{|p|} \expt \biggl[  \int_t^T   (\hat Z_s \cdot n)^{|p|}
      1_{
        \hat Z_T \cdot n<\eta  } \left( 1_{
          \hat Z_s \cdot n\le \frac{\theta}{\gamma} }
        + 1_{
          \hat Z_s \cdot n> \frac{\theta}{\gamma}} \right)  \ud
      s  \biggm| \hat Z_t=z \biggr] \\ 
      &\le 
      \theta^{|p|} \expt \biggl[  \int_t^T   
      1_{
        \hat Z_T \cdot n\le \frac{\theta}{\gamma} } \ud s
      \biggm| \hat Z_t=z \biggr] \\ 
      &+ 
      \gamma^{|p|} (z \cdot n)^{|p|} \expt \biggl[  \int_t^T  1_{
        \frac{\theta}{\gamma}<\hat Z_T \cdot n<\eta }  \ud s
      \biggm| \hat Z_t=z \biggr] \\ 
      &\le  (T-t)
      \theta^{|p|}
      \prob\left(  \hat Z_T \cdot n\le \frac{\theta}{\gamma}
        \biggm| \hat Z_t=z \right) \\ 
      &+ 
      (T-t)    \gamma^{|p|} (z\cdot n)^{|p|} \prob\left(
        \frac{\theta}{\gamma}<\hat Z_T \cdot n<\eta  \biggm|
        \hat Z_t=z \right). \\ 
    \end{aligned}
  \end{equation}

  Observe that when the process $\hat Z_s \cdot n$ is within a
  distance $\mathcal{O}(1/\gamma)$ from zero there are only bounded
  contributions to the derivative of $\tilde u$.  Therefore, what really
  deteriorates the derivative estimate is the time spent between
  $\theta/\gamma$ and $\eta.$

  \subsection*{Bound on $\partial ^2 \tilde u$:}
  Similar to the bound on $\partial \tilde u$ we consider the weighted
  function $W(z,t) := y_2(z) \tilde w(z,t) = y_2(z) \partial^2_z \tilde
  u(z,t) \in \rset^{d\times d}$ with the weight 
  \begin{equation*}
    y_2(z) = y_1(z) \left( z \cdot n + \frac{|\nu^*\cdot
        n|}{\gamma}\right) \in \rset,
  \end{equation*}
  such that the product $y_2 \nabla^2 g$ is bounded uniformly in $z$
  and $\gamma$.  Following the same procedure as for $V$ we obtain
  from \eqref{eq:kolmogorov_rescaled3} that
  \begin{equation*}
    \begin{aligned}
      \partial_t W + \sum_{j=1}^M  (\tilde a_j - \beta_j) D_j W -
      \beta_j W
      &= - R,
      &&\text{ in } \tilde\Pi(Z_0,n) \times [0,T),\\
      W &= y_2 \nabla^2 g, &&\text{ on } \tilde\Pi(Z_0,n) \times \{T\},\\
    \end{aligned}    
  \end{equation*}  
  where
  \begin{equation}
    \label{eq:beta_new}
    \beta(z) :=  \sum_{j=1}^M  \beta_j(z), \quad 
    \beta_j(z) := \tilde a_j(z) D_j y_2(z) y_2(z + \tilde\nu_j)^{-1},
  \end{equation}
  and  
  \begin{equation*}
    R(z,t) =  \sum_{j=1}^M  
    2y_2(z) \nabla \tilde a_j(z) \otimes D_j\tilde v(z,t)
    + \nabla^2 \tilde a_j(z) y_2(z) D_j\tilde u(z,t),
  \end{equation*}
  with '$\otimes$' denoting the tensor product.
  This gives the
  stochastic representation formula
  \begin{equation}
    \label{eq:Wint}
    \begin{aligned}
      W(z,t) =&\
      \underbrace{
        \expt \Bigl[ y_2(\hat Z_T) \nabla^2 g(\hat Z_T) e^{-\int_t^T \beta(\hat Z_{s}) \ud s}
        \Bigm| \hat Z_t=z \Bigr]
      }_{=:W_1}\\
      &+ 
      \underbrace{
        \expt \biggl[ \int_t^T y_2(\hat Z_s) \,
        e^{-\int_s^T \beta(\hat Z_{t'}) \ud t'}\,
        R(\hat Z_s,s)
        \ud s  \biggm| \hat Z_t=z \biggr]
      }_{=:W_2},
    \end{aligned}
  \end{equation}
  where the modified process $\hat Z_t$ has propensity $\hat a_j(z) :=
  \tilde a_j(z) - \beta_j(z)$.
  
  Since $|y_2\nabla^2 g|$ is bounded we have that $|W_1| \leq C$. For $W_2$
  we note that \eqref{eq:lower_bound_intbeta} holds also for the new
  $\beta$ in \eqref{eq:beta_new} and the new modified process $\hat
  Z_t$.  Also, we have the bounds
  \begin{equation*}
    \begin{aligned}
    | \partial_{z_i} \tilde a_j(z) y_2(z) D_j
    \tilde v(z,s) | 
    &\le \ C |  \partial_{z_i} a_j(\gamma z) y_2(z) |\\
    \cdot &\left| \gamma^{|p|} (z \cdot n)^{|p|} 
    \prob \left(
    \theta\gamma^{-1}<\hat Z_T \cdot n < \eta \bigm| \hat Z_s=z
    \right) 
    y_1(z)^{-1} \right| \\
    &= \ C \gamma^{|p_j|} z^{p_j-e_i}
    \left( z \cdot n + \frac{|\nu^* \cdot n|}{\gamma}\right)\\
    \cdot & \left| \gamma^{|p|} (z \cdot n)^{|p|} 
    \prob \left(
    \theta\gamma^{-1}<\hat Z_T \cdot n < \eta \bigm| \hat Z_s=z
    \right) 
    y_1(z)^{-1} \right| \\
    &\le \ C \gamma^{2|p_j|} ( z \cdot n )^{2|p_j|} 
    \prob \left(
    \theta\gamma^{-1}<\hat Z_T \cdot n < \eta \bigm| \hat Z_s=z
    \right), 
    \end{aligned}
  \end{equation*}
  and similarly
  \begin{equation*}
    \begin{aligned}
    | \partial_{z_i z_k} \tilde a_j(z) y_2(z) D_j
    \tilde u(z,s) | \le &\ C |  \partial_{z_i z_k} a_j(\gamma z) y_2(z) / y(z)| \\
    = &\ C \gamma^{|p_j|} z^{p_j-e_i-e_k}
    \left( z \cdot n + \frac{|\nu^* \cdot n|}{\gamma}\right)^2\\
    \le &\ C \gamma^{|p_j|} ( z \cdot n )^{{|p_j|}-2}  \left( z
      \cdot n + \frac{|\nu^* \cdot n|}{\gamma}\right)^2\\
    \le &\ C \gamma^{|p_j|} ( z \cdot n )^{|p_j|}.
    \end{aligned}
  \end{equation*}
  In the same manner as in \eqref{eq:bnd_non_incr} we obtain
  \begin{equation*}
    \begin{aligned}
      |W_2(z,t)| 
      \leq &\
      C  \int_t^T  \expt \biggl[ \gamma^{|p|} (\hat Z_s \cdot n)^{|p|}
      e^{- \xi_0 C \,(\gamma\eta)^{|p|} \, |I_\eta^s|} \biggm| \hat Z_t=z \biggr]\ud s\\
      & + C  \int_t^T  \expt \biggl[
      \gamma^{2|p_j|} ( \hat Z_s \cdot n )^{2|p_j|}
      \prob \left(
        \theta\gamma^{-1}<\hat Z_T \cdot n < \eta \bigm| \hat 
        Z_{s'}=\hat Z_s
      \right)\\
      &\qquad\qquad\qquad \times e^{- \xi_0 C \,(\gamma\eta)^{|p|} \,
        |I_\eta^s|}    \biggm| \hat Z_t=z \biggr]\ud s\\
      \leq &\
      C  + C \gamma^{|p|} (z \cdot n)^{|p|} \prob \left(
        \theta\gamma^{-1}<\hat Z_T \cdot n < \eta \bigm| \hat 
        Z_{s}=z
      \right)
      \\
      & + C \gamma^{2|p|} (z \cdot n)^{2|p|} \prob \left(
        \theta\gamma^{-1}<\hat Z_T \cdot n < \eta \bigm| \hat 
        Z_{s}=z
      \right)^2,\\
    \end{aligned}
  \end{equation*}
  and in particular if the quantity \eqref{eq:crossing_assu} is uniformly bounded in
  $\gamma$ we have that $|W| \leq C$.

\end{proof}

\begin{corollary}\label{cor:growth}
  
  Assume, in addition to the assumptions made in Lemma
  \ref{lem:growth}, that given $z>0$ there exists $\eta(z)$ s.t. for a
  sufficiently small constant $\theta>0$ the quantity
  \begin{equation}\label{eq:crossing_assu}
    \gamma^{|p|} (z \cdot n)^{|p|} \prob\left(
      \frac{\theta}{\gamma}<\hat Z_T \cdot n<\eta  \biggm| \hat Z_t=z \right) <C,
  \end{equation}
  is uniformly bounded in $\gamma$.  Then
  \begin{equation*}
    \label{eq:order2}
    \begin{aligned}
      |\tilde u(z,t)| 
      \leq&\ 
      C \left( z \cdot n + \max_j |\tilde \nu_j\cdot n| \right)^{r_0}
      \left( 1+ z \cdot n\right)^{r_1},\\
      |\partial_z \tilde u(z,t)| 
      \leq&\ 
      C \left( z \cdot n + \max_j |\tilde \nu_j\cdot n| \right)^{r_0-1}
      \left( 1+ z \cdot n\right)^{r_1},\\ 
      |\partial^2_z \tilde u(z,t)| 
      \leq&\ 
      C \left( z \cdot n + \max_j |\tilde \nu_j\cdot n| \right)^{r_0-2}
      \left( 1+ z \cdot n\right)^{r_1},\\
    \end{aligned}
  \end{equation*}
  for $(z,t) \in \tilde\Pi(Z_0,n) \times [0,T]$.

  The assumption \eqref{eq:crossing_assu} is trivially true whenever
  $n\cdot \nu_j =0 \,$ for all $j=1,\ldots,M$, \eg{} for a system
  with only reversible reactions of the form $X_1 \rightleftharpoons
  X_2$.  In that case the product $\hat Z_t\cdot n = z \cdot n$
  remains constant for all $t\le s\le T.$


  
\end{corollary}




\subsection{Computational work}\label{sec:work}

The results in the previous section now give us the possibility to
judge in which regimes and for which propensities the Posson bridge
tau-leap method is expected to be more efficient than the SSA. As
mentioned in Section \ref{sec:introduction}, the computational work of
the SSA is roughly inversely proportional to the total propensity, and
becomes intractable as the number of particles grow. The tau-leap
method, on the other hand, approximates the process by using fixed
time steps, and may lose accuracy as the number of particles grow
unless the step size is adjusted accordingly. A reasonable way to
compare the two methods is thus to keep the required accuracy of the
tau-leap method fixed.

From Theorem \ref{thm:rel_error_bnd} we see that the computational
work for the Poisson bridge tau-leap method to achieve a relative error
$\varepsilon:=\expt \bigl[ g(\bar Z_T) - g(Z_T) \bigr]$, using the time
step $\tau = h\gamma^{2-2p}$, is
\begin{equation}\label{eq:tl_work}
  \it{Work}_{TL} \approx \frac{C}{\tau} \approx \frac{C\gamma^{2p-2}}{\varepsilon}.
\end{equation}
The comparable work for the SSA is
\begin{equation*}
  \it{Work}_{SSA} \approx \frac{C}{\tau_{SSA}} \approx C\tilde a(Z_0) =
  C a(\gamma Z_0) \approx C \gamma^p Z_0^p,  
\end{equation*}
and then we have
\begin{equation}\label{eq:work_comparison}
  \frac{\it{Work}_{SSA}}{\it{Work}_{TL}} \approx
  C \gamma^{2-p} Z_0^p \varepsilon.
\end{equation}
Thus, asymptotically as $\gamma\to\infty$, for $p=1$ the tau-leap method
outperforms the SSA. For $p=2$ the methods are comparable and for
$p>2$ the SSA seems to be the right choice.  Note, that for $p=1$ the
step size $\tau=h$ is independent of the number of particles even
though the time between reactions is of order $\gamma^{-p}$.


\begin{remark}
  The estimated work \eqref{eq:tl_work} and consequently the
  comparison \eqref{eq:work_comparison} is a worst case scenario.  For
  a simple decaying reaction, $X\to\emptyset$, with propensity $a(x) =
  x^p$ for $p=1,2,3$, stoichiometric number $\nu=-1$ and initial
  number of particles $\gamma$, the tau-leap method is
  \begin{equation*}
    \begin{aligned}
      \bar X_{t_{n+1}} &= \bar X_{t_n} - \poisson( \bar X_{t_n}^p 
      \tau_n ), \quad p=1,2,3,\\
      \bar X_{t_0} &= \gamma.
    \end{aligned}
  \end{equation*}
  Rewrite
  \begin{equation*}
    \poisson( \bar X_{t_n}^p \tau_n ) = \bar X_{t_n}^p \tau_n +
    \sqrt{\bar X_{t_n}^p \tau_n} \Delta W_{t_n},
  \end{equation*}
  where
  \begin{equation*}
    \Delta W_{t_n} := \frac{\poisson( \bar X_{t_n}^p
      \tau_n ) - \bar X_{t_n}^p \tau_n }{\sqrt{\bar X_{t_n}^p
        \tau_n} }.
  \end{equation*}
  The Berry-Ess\'en theorem, see \cite{durrett}, implies that $\Delta
  W_{t_n}$ approaches a $N(0,\sqrt{\tau_n})$ distributed variable as
  the number of particles grow. Neglecting the (relatively decreasing)
  stochastic term the tau-leap method can thus be approximated by the
  mean field equation
  \begin{equation*}
    \begin{aligned}
    \bar X_t' &= - \bar X_t^p, \quad t\in(0,T], \quad p=1,2,3,\\
    \bar X_0 &= \gamma.
    \end{aligned}
  \end{equation*}
  with scaled solutions
  \begin{equation*}
     \bar Z_t = e^{-t}, \quad 
     \bar Z_t = \frac{1}{\gamma t+1}, \quad 
     \bar Z_t = \frac{1}{\sqrt{ 2\gamma^2 t + 1 }},
  \end{equation*}
  for $p=1,2,3$, respectively. Since $g(\bar Z_T)$ is essentially independent of
  $\gamma$ for $p=1$, the work is also expected to be independent of
  $\gamma$, as indicated in \eqref{eq:tl_work}. For $p=2,3$ the
  relative solution $\bar Z_T$ decays as $\gamma^{-1}$ and most of the
  decay happen in $\gamma^{-p}$ time, so for the interval
  $t\in(\gamma^{-p},T)$ less work is required and using the time-step
  $\tau=h\gamma^{2-2p}$ in the interval $t\in(0,\gamma^{-p})$ implies
  that the estimate \eqref{eq:tl_work} gains at least one order of
  $p$.

\end{remark}
\begin{remark}
  In \cite{agk}, the authors consider the case of constant density,
  by assuming that
  \begin{equation*}
    c_j = C\gamma^{1-|\nu_j^-|}, \quad |\nu_j^-|:=\sum_{i=1}^{d} |\nu_{ji}^-|,
  \end{equation*}
  in \eqref{eq:propensity}.
  This implies that the propensity can be written as
  \begin{equation*}
    \tilde a_j(z) = a_j(\gamma x) = \gamma A(x),
  \end{equation*}
  for some uniformly bounded function $A$, and the 'effective'
  propensity is thus linear in terms of $\gamma$.  Using this
  assumption, the authors of \cite{agk} show that choosing $\tau =
  \gamma^{-\beta}$, for some $\beta\in(0,1)$, gives a relative error
  of the same order.  Using the same constant density assumption would,
  with our analysis, lead to a step size $\tau=h$, for some $h>0$, and a
  relative error of the same order, which has the additional advantage
  that the step size is independent of $\gamma$ and the unknown
  constant $\beta$.

\end{remark}


\subsection{Dual approximation}

In Theorem \ref{thm:rel_error_bnd} we showed an \emph{a priori} estimate for the
relative global error.  For an \emph{a posteriori} error estimate we
have the following result:
\begin{theorem}[A posteriori error]\label{thm:aposteriori}


  Assume that Lemma \ref{lem:cont_ext} and the growth condition in
  Corollary \ref{cor:growth} hold.  Given $M_\omega$ sample paths $\{
  \bar X_{t_n}(\omega_i)\}_{i=1}^{M_\omega}$ from the Poisson bridge
  tau-leap method, with deterministic time steps $\{t_n\}_{n=0}^{N}$,
  the relative error can be approximated by
  
  
  \begin{equation}\label{eq:aposteriori}
    \begin{aligned}
      \expt \bigl[ g(Z_T)-g(\bar Z_T) \bigr]
      = &\
      \frac{1}{M_\omega} \sum_{i=1}^{M_\omega} 
      I(\omega_i)
      + \Ordo{\gamma^{-1}\tau_{max}^2} + \varepsilon_d + \varepsilon_\omega,
    \end{aligned}
  \end{equation}
  where
  \begin{equation*}
    I :=
    \sum_{n=0}^{N-1}
    \frac{\tau_n}{2}
    \tilde \varphi_{t_{n+1}} \cdot
    \biggl(
    \sum_{j=1}^M
    \Bigl( \tilde a_j(\bar Z_{t_{n+1}}) - 
    \tilde a_j(\bar Z_{t_n})
    \Bigr) 
    \tilde \nu_j
    \biggr).
  \end{equation*}
  Here, the time steps are $\tau_n:=t_{n+1}-t_n$ and the dual weight
  $\tilde \varphi_{t_n} \in \rset^d$ is defined by the backward
  problem
  \begin{equation}\label{eq:dual_equations}
    \begin{aligned}
      \tilde \varphi_{t_n} &= J_{t_n}
      \tilde \varphi_{t_{n+1}} 
      , \quad t_n=0,\ldots,N-1,\\
      \tilde \varphi_T &= g'(Z_T),
    \end{aligned}
  \end{equation}
  with $J_{t_n} \in \rset^{d\times d}$ defined by
  \begin{equation}\label{eq:variation}
    J_{t_n} := 
    Id + \sum_{j=1}^M \tilde \nu_j \cdot \nabla \tilde a_j(\bar Z_{t_n}) 
    \left( 
      \tau_n +  \frac{1}{2\sqrt{\tilde a_j(\bar Z_{t_n})}} \Delta \bar W_{t_n}^j  
    \right) {\one}_{\tilde a(\bar Z_{t_n})>0},
  \end{equation}    
  where
  \begin{equation}\label{eq:aux_increment}
    \Delta \bar W_{t_n}^j := \frac{ \Delta Y_{t_n}^j/\gamma - \tau_n
      \tilde a_j ( \bar
      Z_{t_n} ) }{\sqrt{\tilde a_j ( \bar Z_{t_n} ) }},
  \end{equation}
  are approximate Wiener increments, defined for $\tilde a(\bar Z_{t_n})>0$.
  The process $\bar Z_{t_n}$ has increments $\Delta Y_{t_n}^j$, as
  described by \eqref{eq:delta_y} in Section \ref{sec:poisson_bridge}.
  
  Finally, $\varepsilon_\omega$ is the standard Monte Carlo error with
  $\var [\,\varepsilon_\omega\,] = M_\omega^{-1}\var[\,I\,]$, and
  $\varepsilon_d$ is a remainder of diffusion type
  \begin{equation*}
    \varepsilon_d =     
    \sum_{n=0}^{N-1} 
    \frac{\tau_n}{2} 
    \Bigl(
    \partial_z \tilde u(\bar Z_{t_{n+1}},t_{n+1}) -
    \expt \bigl[ \ \tilde \varphi_{t_{n+1}} \bigm| \mathcal{F}_{t_{n+1}} \bigr]
    \Bigr)
    \cdot
    \sum_{j=1}^{M}    
    \Bigl( \tilde a_j(\bar Z_{t_{n+1}}) - 
    \tilde a_j(\bar Z_{t_n})
    \Bigr) 
    \tilde \nu_j,
  \end{equation*}
  where the $\sigma$-algebra $\mathcal{F}_{t}$ is generated by the history
  of $\bar Z_{t}$ up to time $t$.

  
\end{theorem}


\begin{proof}
  
  The proof of the above theorem comes from applying the
  Trapezoidal rule to the error representation formula
  in Lemma \ref{lem:tau_exp_error_rep}, Taylor expanding $\tilde u$ as in the
  first part of the proof of Theorem \ref{thm:rel_error_bnd}, and
  using Corollary \ref{cor:growth}. For given
  time steps $\tau_n$ this gives
  \begin{equation}\label{eq:trapezoidal}
    \begin{aligned}
      \expt \bigl[ & g(Z_T)-g(\bar Z_T) \bigr] =\\
      &= \sum_{j=1}^M \sum_{n=0}^{N-1} \int_{t_{n}}^{t_{n+1}} 
      \expt \Bigl[ 
      \bigl( \tilde a_j( \bar{Z}_s ) - \tilde a_j( \bar{Z}_{t_n} ) \bigr)
      \bigl( \tilde u( \bar Z_s + \tilde \nu_j, s ) - 
      \tilde u( \bar{Z}_s, s ) \bigr)
      \Bigr]
      \ \ud s\\
      &= \sum_{j=1}^M \sum_{n=0}^{N-1} \frac{\tau_n}{2} 
      \expt \Bigl[ \bigl( \tilde a_j( \bar{Z}_{t_{n+1}} ) -
      \tilde a_j( \bar{Z}_{t_n} ) \bigr) 
      \partial_z \tilde u( \bar Z_{t_{n+1}}, t_{n+1} ) \cdot \tilde \nu_j
      \Bigr]\\
      &\qquad\qquad + \Ordo{\gamma^{-1}\tau_n^3}\\
      &=  \sum_{j=1}^M \sum_{n=0}^{N-1} \frac{\tau_n}{2} 
      \expt \Bigl[ \bigl( \tilde a_j( \bar{Z}_{t_{n+1}} ) -
      \tilde a_j( \bar{Z}_{t_n} ) \bigr) 
      \expt \bigl[ \ \tilde \varphi_{t_{n+1}} \bigm|
      \mathcal{F}_{t_{n+1}} \bigr] \cdot \tilde \nu_j
      \Bigr]\\
      &\qquad\qquad + \Ordo{\gamma^{-1}\tau_n^3} + \varepsilon_d\\
      &= \sum_{j=1}^M \sum_{n=0}^{N-1} \frac{\tau_n}{2} 
      \expt \Bigl[ \bigl( \tilde a_j( \bar{Z}_{t_{n+1}} ) -
      \tilde a_j( \bar{Z}_{t_n} ) \bigr) 
      \tilde \varphi_{t_{n+1}} \cdot \tilde \nu_j
      \Bigr] + \Ordo{\gamma^{-1}\tau_n^3} + \varepsilon_d.
    \end{aligned}
  \end{equation}
  
\end{proof}

\begin{remark}

  Note that \eqref{eq:variation} is the approximation of the first
  variation $\partial X_{n+1} / \partial X_{n}$ for the chemical
  Langevin equation,
  \begin{equation}\label{eq:langevin}
     X_{{t_{n+1}}} =  X_{t_n} + \sum_{j=1}^M \nu_j 
    \Big( \tau_n a_j (  X_{t_n} ) + \Delta W_{t_n}^j \sqrt{a_j(
      X_{t_n})} \Big),
  \end{equation}
  where $\Delta W_{t_n}^j \sim N(0,\tau_n)$ are standard Wiener
  increments, and $\tilde \varphi_{t_n}$ is thus an approximation to
  the discrete dual for the corresponding value function. This means
  that the $\varepsilon_d$ term contains errors from both
  approximating the value function and the dual for the Langevin equation.

\end{remark}

\section{Examples}\label{sec:examples}
The goal is here to numerically verify the error representation in
Lemma \ref{lem:tau_exp_error_rep} and the a posteriori error estimate
in Theorem \ref{thm:aposteriori}. This is a fundamental step for
future work developing an appropriate adaptive algorithm for the
Poisson bridge tau-leap method in the spirit of \cite{mstz}.

\subsection{Testing the estimates}\label{sec:testing_estimates}

To test the error representation in Lemma \ref{lem:tau_exp_error_rep} we use
the approximation
\begin{equation*}
  \underbrace{ \expt [ g(X_T) - g( \bar X_T) ] }_{=:lhs} =
  \underbrace{ \sum_{n=0}^{N-1} \sum_{j=1}^M  \expt  \left[ I_{j,n}
    \right] }_{:=rhs} + \Ordo{\tau_{max}^2},
\end{equation*}
where
\begin{equation}
  \begin{aligned}
    I_{j,n} = &\ \frac{\tau_n}{2} 
    \Bigl( a_j ( \bar X_{t_{n+1}} ) - a_j ( \bar X_{t_n} ) \Bigr) \\
    & \times \Bigl( 
    u( \bar X_{t_{n+1}} + \nu_j, t_{n+1}) -
    u( \bar X_{t_{n+1}}, t_{n+1}) \Bigr),
  \end{aligned}
\end{equation}
and show that both $lhs$ and $rhs$ decay as $\Ordo{\tau_{max}}$
asymptotically, and that the efficiency index $lhs/rhs$ approaches
$1$.

To estimate the true value function $u$, we solve the Kolmogorov
backward equation \eqref{eq:kolmogorov}, which by Assumption
\ref{assu:boundedness} will become a $\prod_{i=1}^d ( X^{(i)}_{max}+1
)$ dimensional (stiff) system of ordinary differential equations,
since the number of particles in the chemical system will either be
constant or decrease over time, \eg{} $X^{(i)}_t \in
\{0,1,\ldots,X^{(i)}_{max}\}$ for some upper bound $X_{max}\in
\rset^d$.

 
For large values of $X_{max}$ the discretization is chosen such that
the system is solved for a subset of logarithmically distributed
integers in $[0,X^{(i)}_{max}]$, see Figure
\ref{fig:logarithmic_grid}. Of course, this distribution is in no
sense optimal and ideally a spatially adaptive algorithm should be
used, but it can be expected that points close to any $X^{(i)}=0$ will
have a greater contribution to the error.

\begin{figure}[hbpt]
  \centering
  \includegraphics*[width=0.7\textwidth]{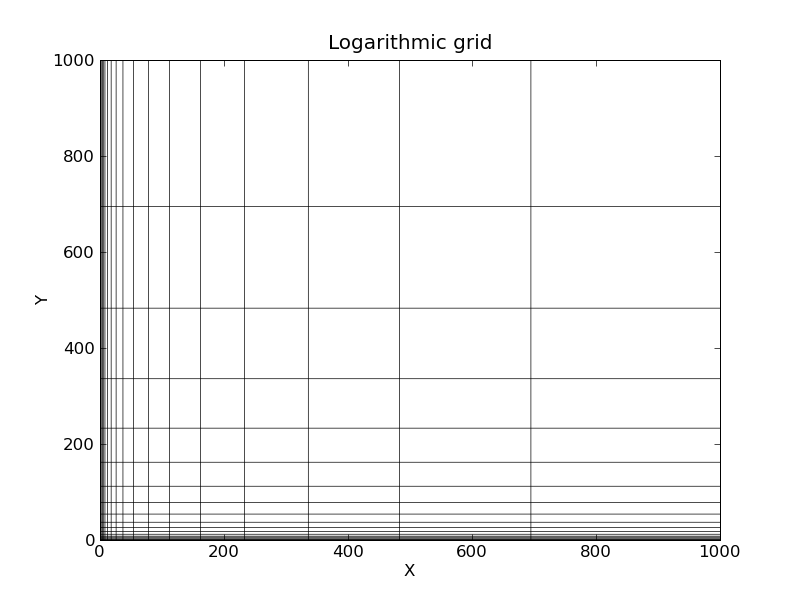}
  \caption{Example of logarithmic grid used for solving the backward
    Kolmogorov equation in $\zset_+^2$.}
  \label{fig:logarithmic_grid}
\end{figure}

Calculating $lhs$ with sufficient accuracy can be very demanding for
problems with many species and a high number of particles. A less
demanding way is to use the approximation
\begin{equation*}
  \begin{aligned}
    lhs &= \expt [ g ( X_T ) - g ( \tilde X_T ) ] 
    + \expt [ g ( \tilde X_T ) - g ( \bar X_T ) ] \\
    &= \underbrace{2 \expt [ g ( \tilde X_T ) -
      g ( \bar X_T ) ] }_{lhs_{approx}} +  \Ordo{\tau_{max}^2},    
  \end{aligned}
\end{equation*}
where the process $\tilde X_t$ is generated by using half the step
size of $\bar X_t$ and the sample path $(\lambda,Y)$ generated by
$\bar X_t$. The estimate $lhs_{approx}$ does not use the value
function at all, and has the sample variance of order $\tau M^{-1}$
compared to
$M^{-1}$
for $lhs$, see \cite{li}.

In practice, the estimate $rhs$ in itself is not much of use since it
requires knowledge of the value function $u$, and thus $rhs$ must be
approximated by computable quantities, \eg{} as in Theorem
\ref{thm:aposteriori}.  For this purpose, let
\begin{equation*}
  rhs_{dual} := 
  \expt\left[ \sum_{n=0}^{N-1} \sum_{j=1}^M \frac{\tau_n}{2} 
    \left( a_j \left( \bar X_{n+1} \right)
      - a_j \left( \bar X_n \right) \right)
    \varphi_n \cdot \nu_j \right],
\end{equation*}
where the approximate discrete dual is given as in Theorem
\ref{thm:aposteriori} but without the scaling factor $\gamma$.  The
accuracy of $rhs_{dual}$ is of great importance to construct a proper
adaptive algorithm, see \eg{} \cite{mstz,msst,stz}. To show this we
use the error density
\begin{equation}\label{eq:error_indicator}
  \rho_{j,n} := \left| \frac{1}{2\tau_n} \expt [ \left( a_j \left( X_{n+1} \right)
      - a_j \left( X_n \right) \right) \varphi_n \cdot \nu_j ] \right|,
\end{equation}
defined for the initial deterministic time steps, \ie{} not the time
steps given by the Poisson bridge tau-leap method in Section
\ref{sec:poisson_bridge}, see \cite{mordecki}. This gives the total
error
\begin{equation*}
  \varepsilon := rhs_{dual} = \sum_{n=0}^{N-1} \sum_{j=1}^M \tau_n^2 \rho_{j,n},
\end{equation*}
and the current work $\it{Work}_{TL} := N$ is then compared with the estimated
work to achieve the same error $\varepsilon$ for an optimal adaptive
mesh
\begin{equation}\label{eq:adaptive_work}
  \it{Work}_a = \left( \sum_{n=0}^{N-1} \sqrt{\rho_n} \tau_n \right)^2 ,
\end{equation}
and a uniform mesh
\begin{equation}\label{eq:uniform_work}
  \it{Work}_u = T \sum_{n=0}^{N-1} \rho_n \tau_n .
\end{equation}

\subsection{Reactive decay}
In this example we have an irreversible reaction where molecules of a
single species spontaneously disappear, possibly into another particle
type.  The chemical reaction for reactive decay (or isomerization) can
be written as
\begin{equation}\label{eq:decay}
  X \to \emptyset
\end{equation}
and is here described by the linear propensity function 
\begin{equation}
  \label{eq:decay_prop}
  a(x) = cx,  
\end{equation}
initial value $X(0) = X_0 \in \zset_+$, and stoichiometric number
$\nu=-1$.  We test four different cases: a high number of particles or
a low number of particles, that becomes negative often or not so often,
see Table \ref{tab:decay}. In Figure \ref{fig:decay_realizations} a
few realizations of $X_t$ for $t\in[0,1]$ are shown for each example.

\begin{table}[hbpt]
  \centering
  \begin{tabular}{c|cccc}
    Example & 1 & 2 & 3 & 4 \\
    \hline
    $X_0$ & 10 & $10^6$ & 10 & $10^6$\\ 
    $c$ & 0.2 & 1.0 & 2.0 & 7.0
  \end{tabular}
  \caption{Reactive decay, \cf{} \eqref{eq:decay} and \eqref{eq:decay_prop}.}
  \label{tab:decay}
\end{table}

As quantities of interest we take the first moments of $X$, \ie{}
$g(x)=x^{mom}$ for $mom=1,2,3$ and with final time $T=1$.  In Figure
\ref{fig:decay_convergence_mom1}, \ref{fig:decay_convergence_mom2} and
\ref{fig:decay_convergence_mom3} the convergence of $lhs_{approx}$,
$rhs$ and $rhs_{dual}$ can be seen for the different moments.  In all
cases we see a linear $O(\tau)$ convergence and in Figure
\ref{fig:decay_efficiency_mom1}, \ref{fig:decay_efficiency_mom2} and
\ref{fig:decay_efficiency_mom3}, we see that the corresponding
efficiency indices stay close to $1$, \ie{} the computable error
approximation in Theorem \ref{thm:aposteriori} agrees well with the
error representation formula in Lemma \ref{lem:tau_exp_error_rep}. Comparing
the current work $W$ with the work estimates for an adaptive mesh in
\eqref{eq:adaptive_work}, and a completely uniform mesh in
\eqref{eq:uniform_work}, shows that for these relatively non-stiff
examples adaptivity will not make any improvement, see Table
\ref{tab:decay_steps}.
In this case, a uniform mesh is thus most
suitable, and the examples are only run for verification purposes.
Also note that the error in Example 2 and 4, \ie{} for a high number of
particles, is almost completely governed by the deterministic error
and require few realizations to achieve a low statistical error of the
estimates $lhs$, $rhs$ and $rhs_{dual}$.

\begin{figure}[hbpt]
  \centering
  \subfigure[Example 1]{
    \includegraphics*[width=0.45\textwidth,viewport=80 260 480 590]
    {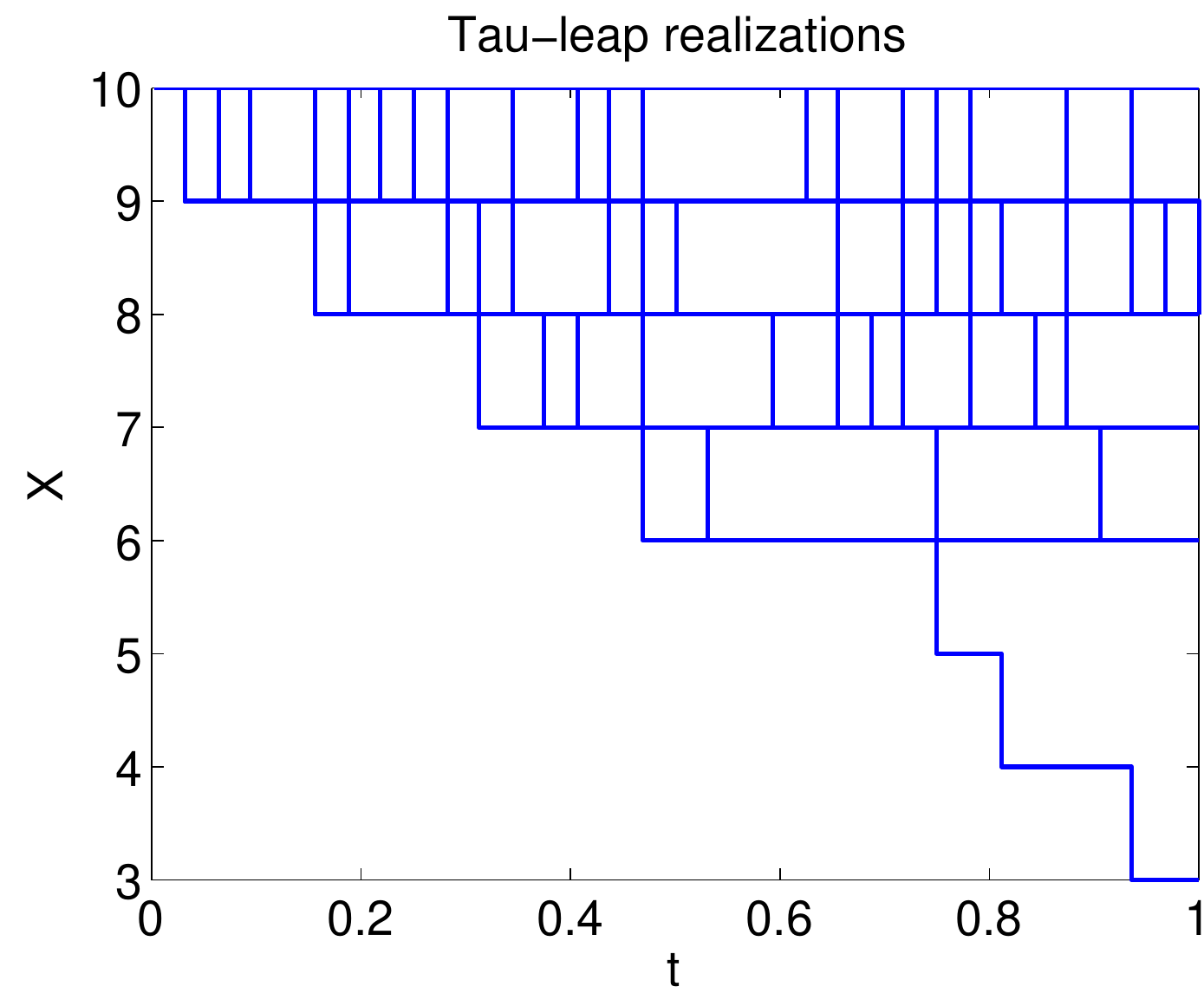}}
  \subfigure[Example 2]{
    \includegraphics*[width=0.45\textwidth,viewport=80 260 480 590]
    {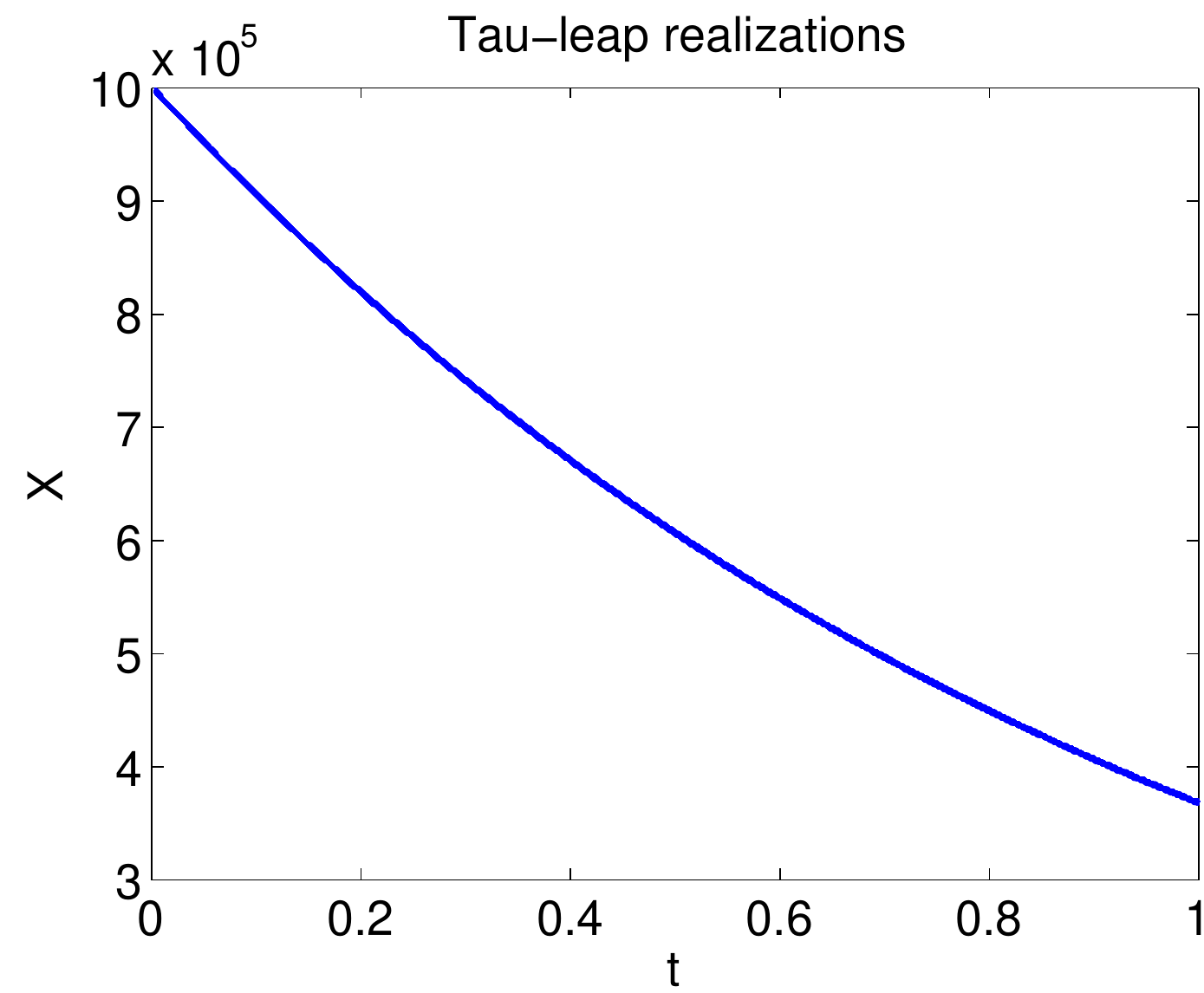}}
  \subfigure[Example 3]{
    \includegraphics[width=0.45\textwidth,viewport=80 260 480 590]
    {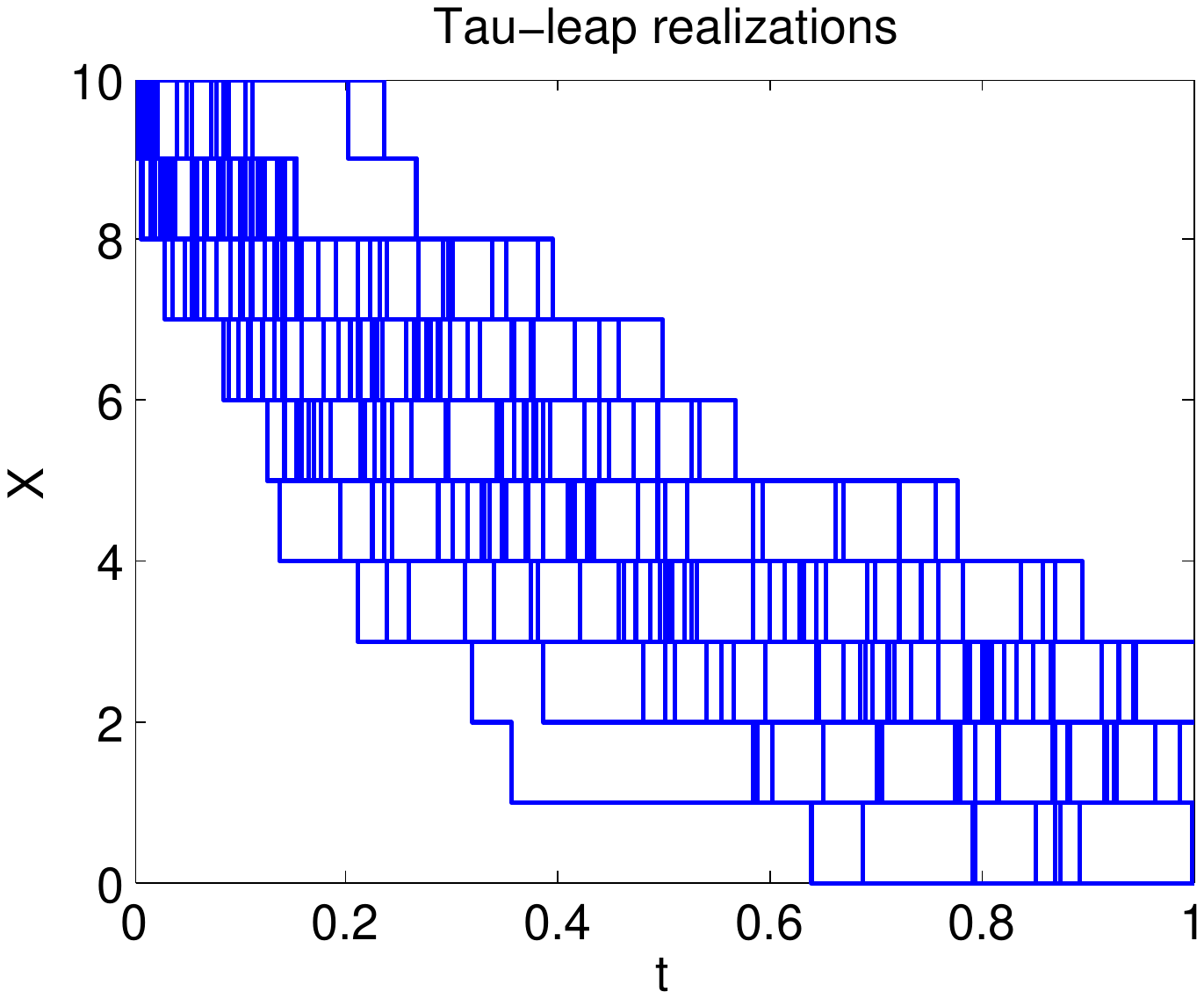}}
  \subfigure[Example 4]{
    \includegraphics[width=0.45\textwidth,viewport=80 260 480 590]
    {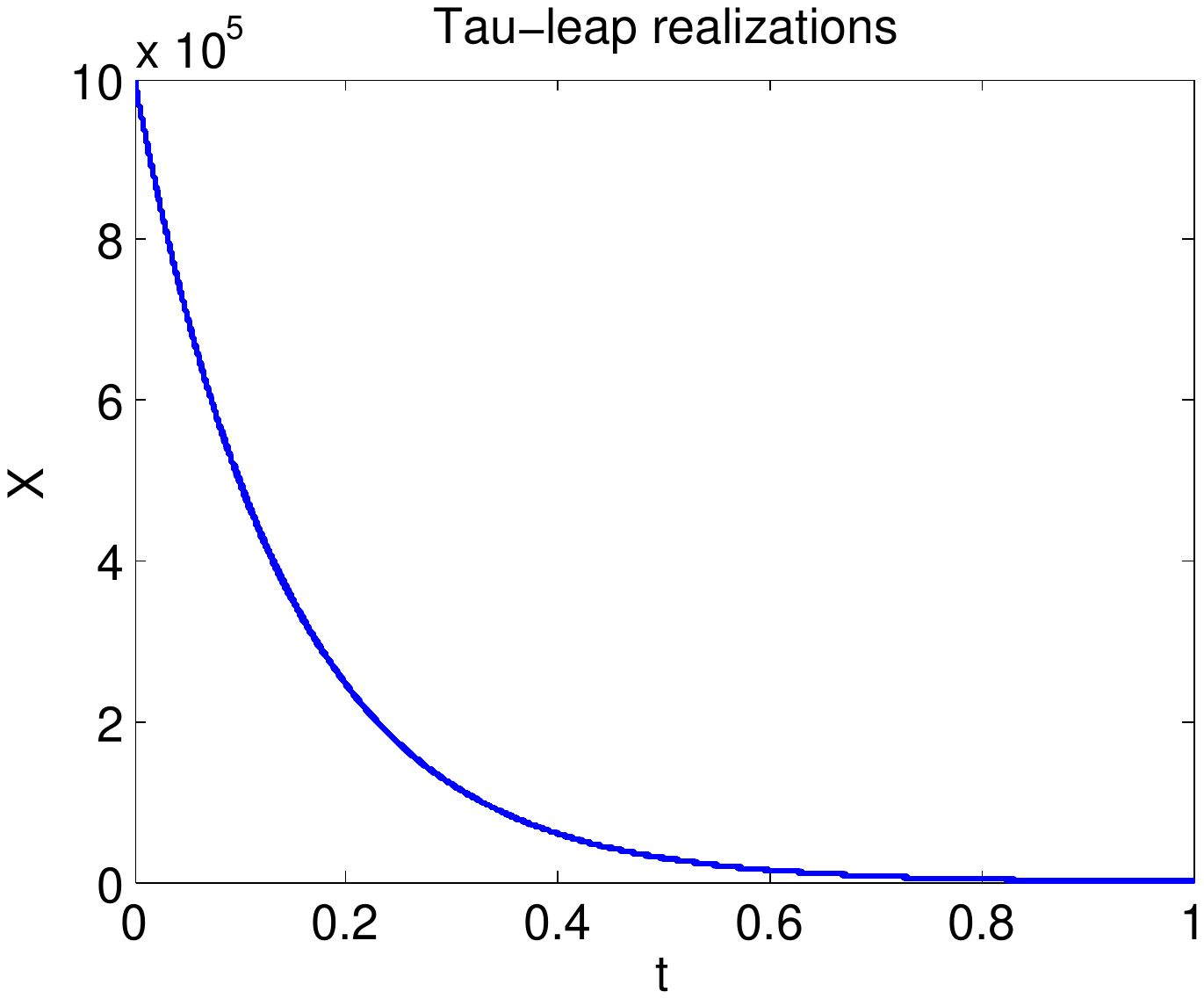}}
  \caption{A few realizations of reactive decay for Example 1 to 4. These
    examples are mostly suited for uniform meshes and adaptivity will
    thus not make any improvement.}
  \label{fig:decay_realizations}
\end{figure}

\begin{figure}[hbpt]
  \centering
  \subfigure[Example 1]{
    \includegraphics*[width=0.45\textwidth,viewport=70 220 500 580]
    {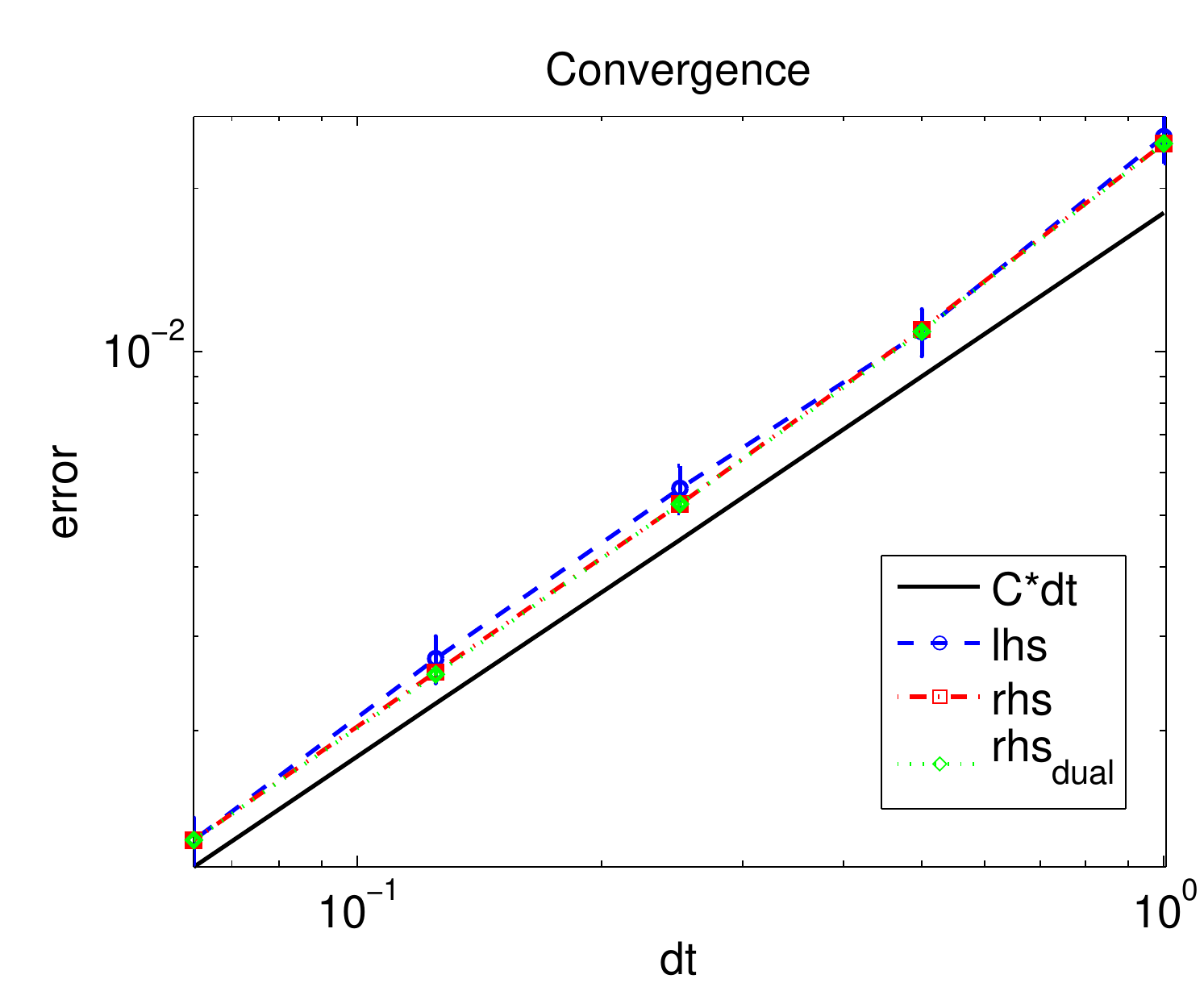}}
  \subfigure[Example 2]{
    \includegraphics*[width=0.45\textwidth,viewport=70 220 500 580]
    {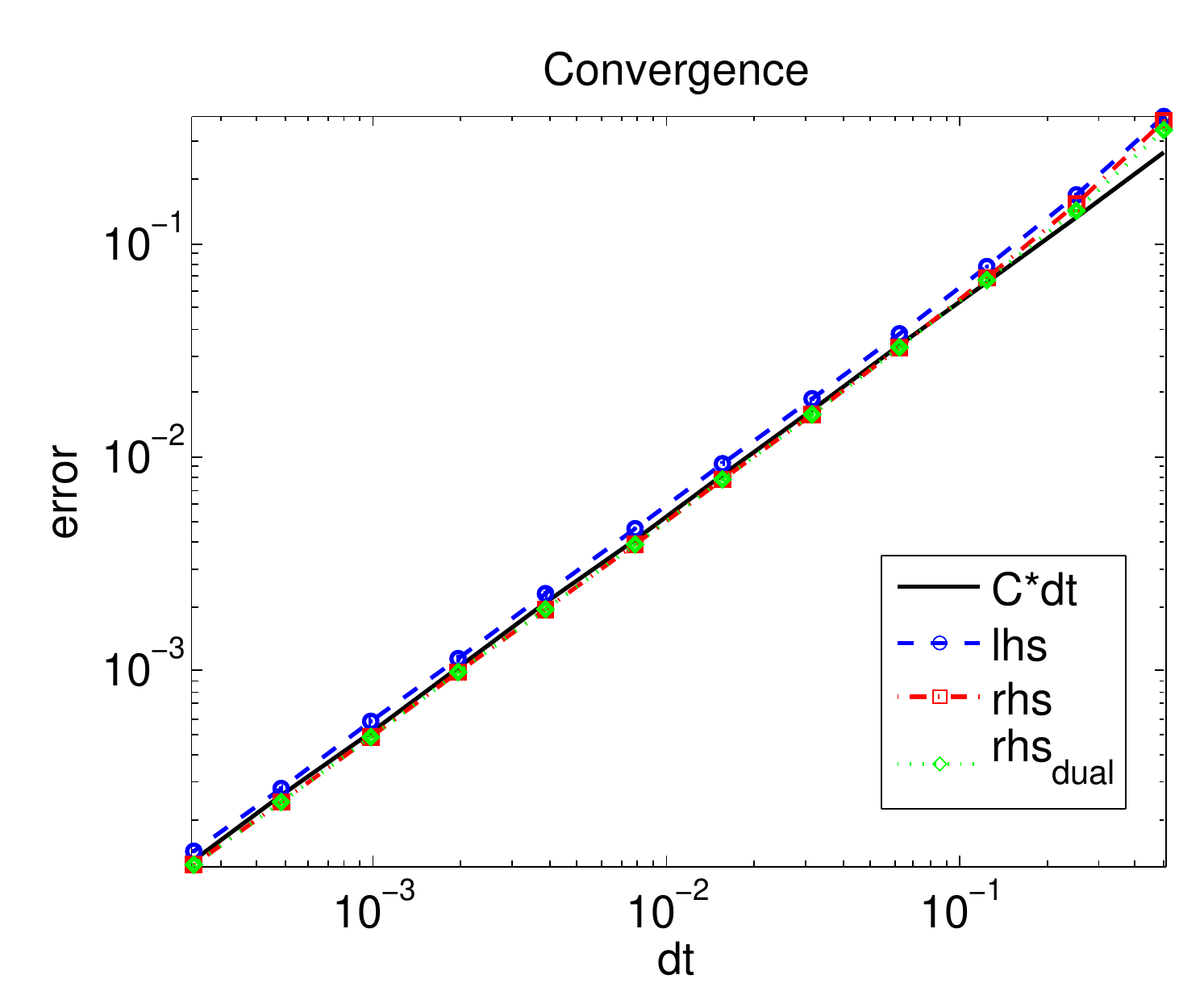}}
  \subfigure[Example 3]{
    \includegraphics[width=0.45\textwidth,viewport=70 220 500 580]
    {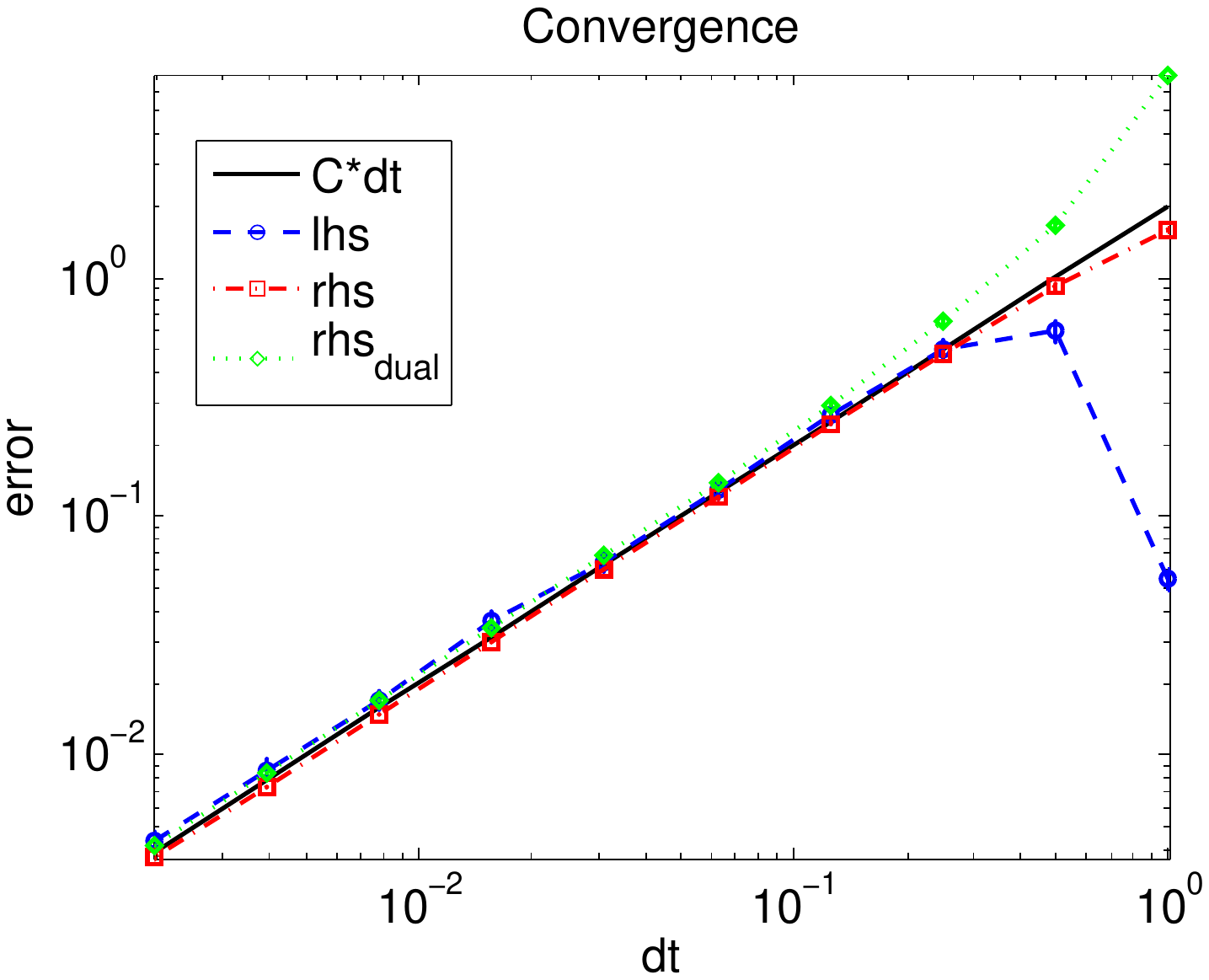}}
  \subfigure[Example 4]{
    \includegraphics[width=0.45\textwidth,viewport=70 220 500 580]
    {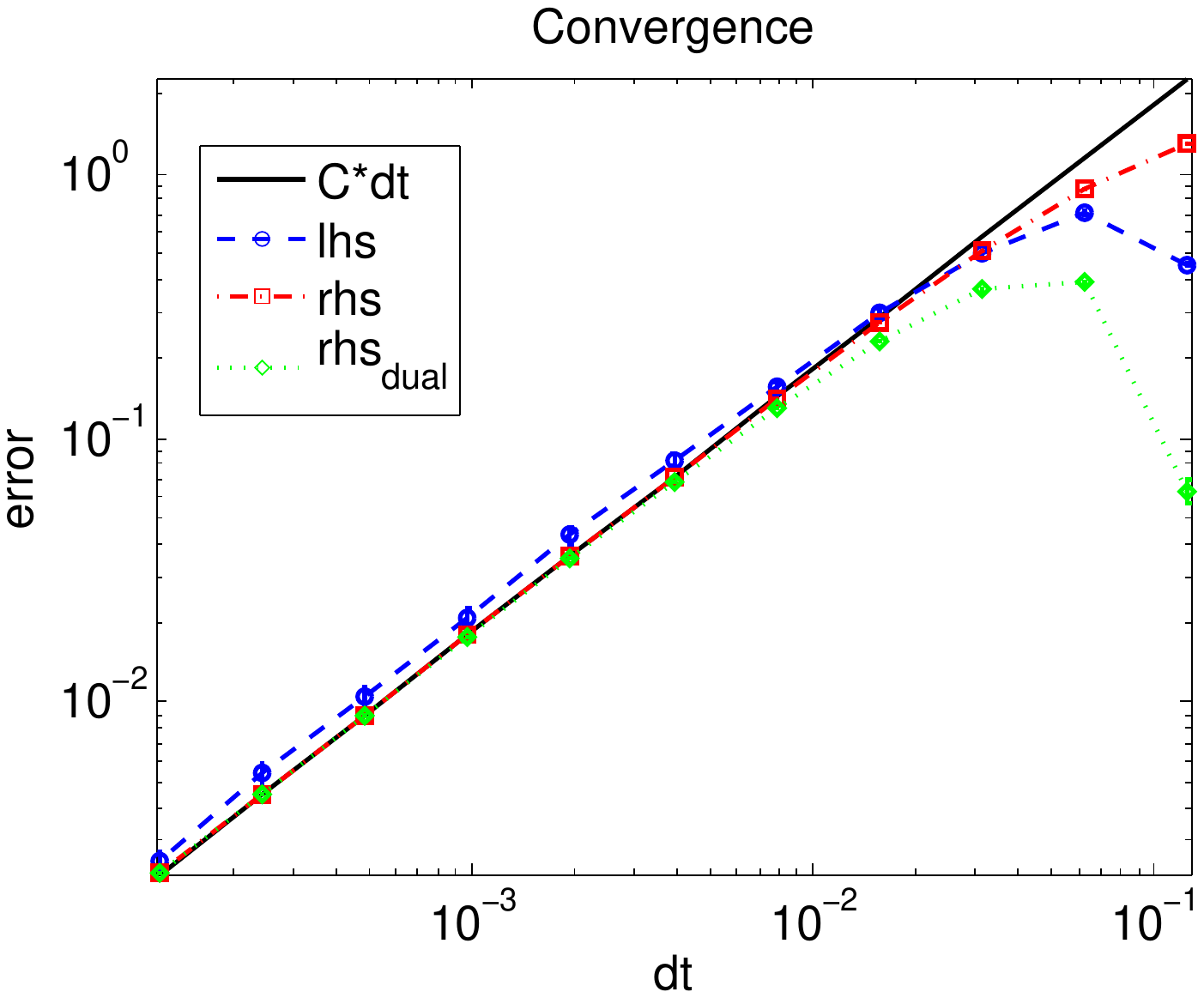}}
  \caption{Reactive decay: Convergence of the errors $lhs_{approx}$,
    $rhs$ and $rhs_{dual}$ with respect to the maximum time step.
    Here, $g(x)=x$ and the data is scaled by dividing with $u(x_0,0)$.
    For each data point the number of samples is controlled by the
    standard error such that $1.96 SE \leq 0.1 \mu$, where $\mu$ is
    the sample mean and $SE$ the standard error of the mean. Bars
    indicate the $95\%$ confidence intervals.}
  \label{fig:decay_convergence_mom1}
\end{figure}
\begin{figure}[hbpt]
  \centering
  \subfigure[Example 1]{
    \includegraphics*[width=0.45\textwidth,viewport=70 220 500 580]
    {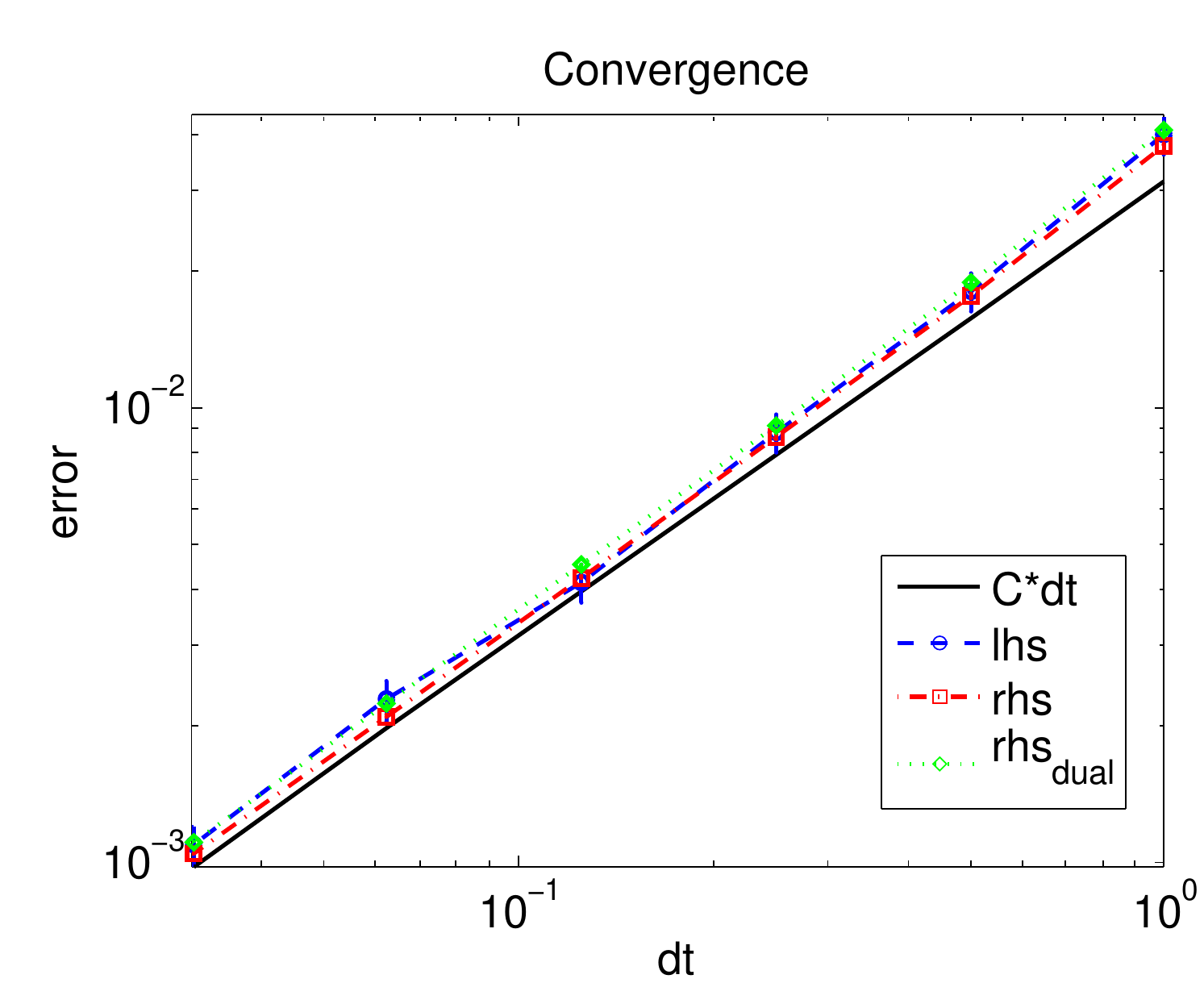}}
  \subfigure[Example 2]{
    \includegraphics*[width=0.45\textwidth,viewport=70 220 500 580]
    {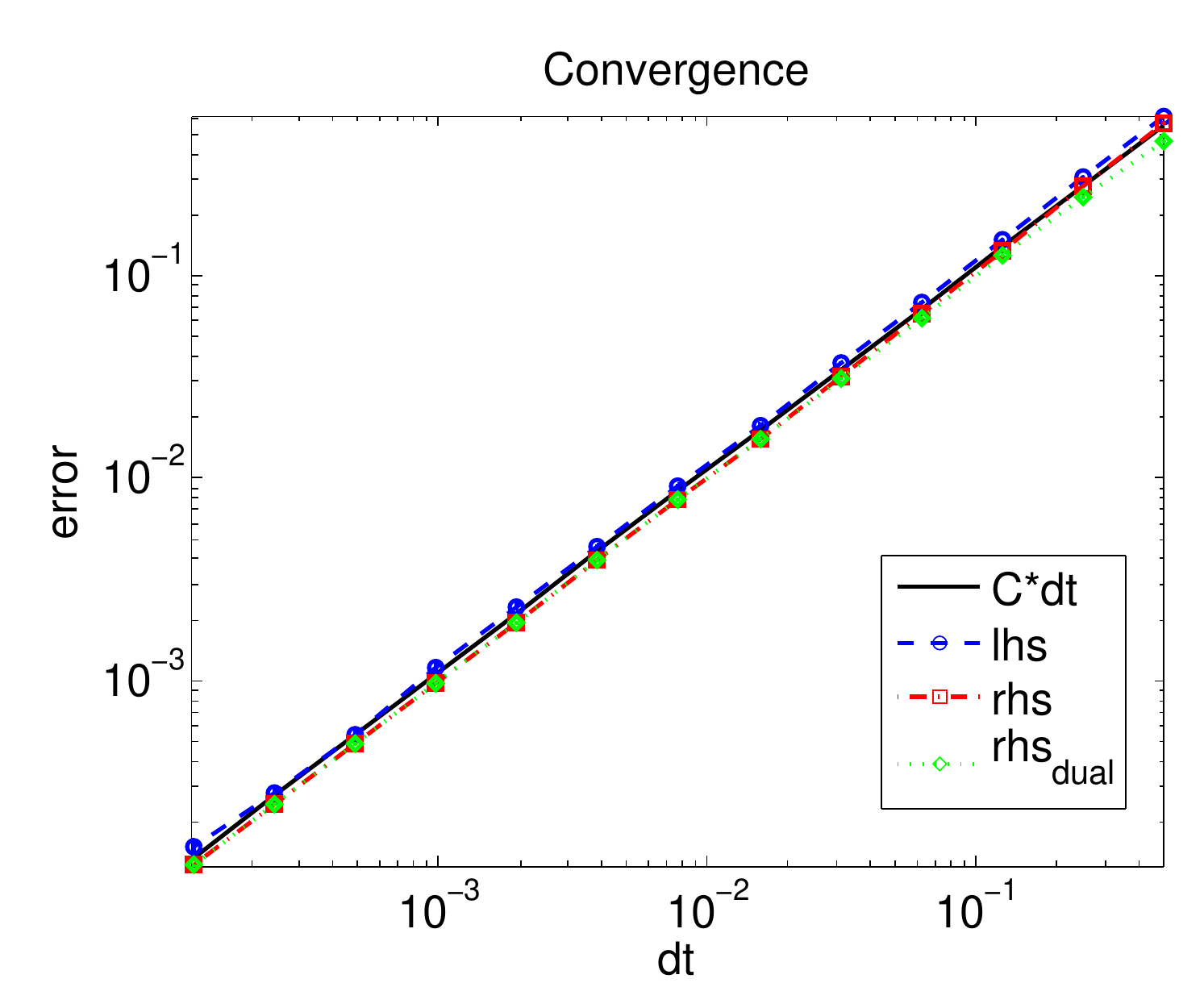}}
  \subfigure[Example 3]{
    \includegraphics[width=0.45\textwidth,viewport=70 220 500 580]
    {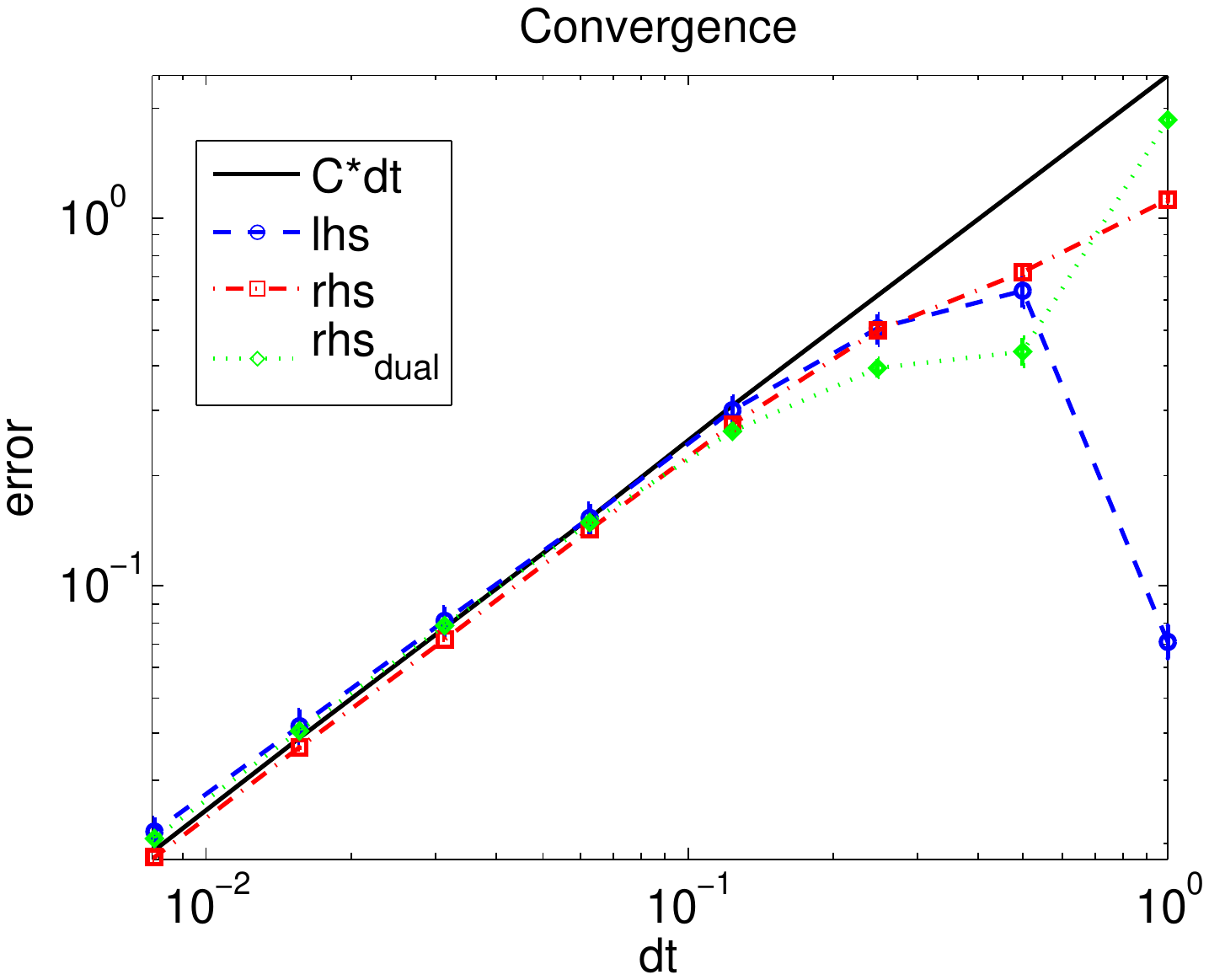}}
  \subfigure[Example 4]{
    \includegraphics[width=0.45\textwidth,viewport=70 220 500 580]
    {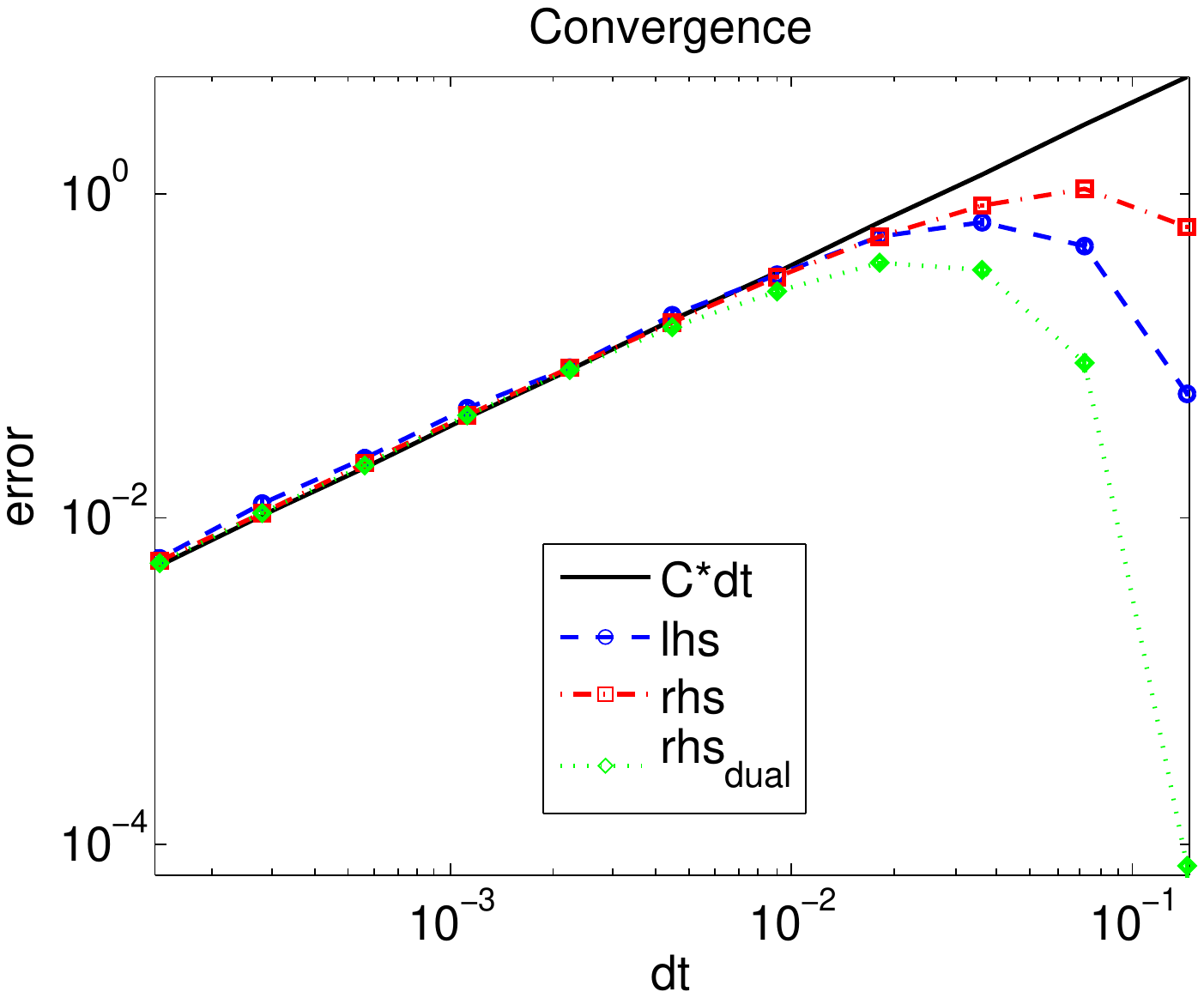}}
  \caption{Reactive decay: convergence for second moment of $X$.}
  \label{fig:decay_convergence_mom2}
\end{figure}
\begin{figure}[hbpt]
  \centering
  \subfigure[Example 1]{
    \includegraphics*[width=0.45\textwidth,viewport=70 220 500 580]
    {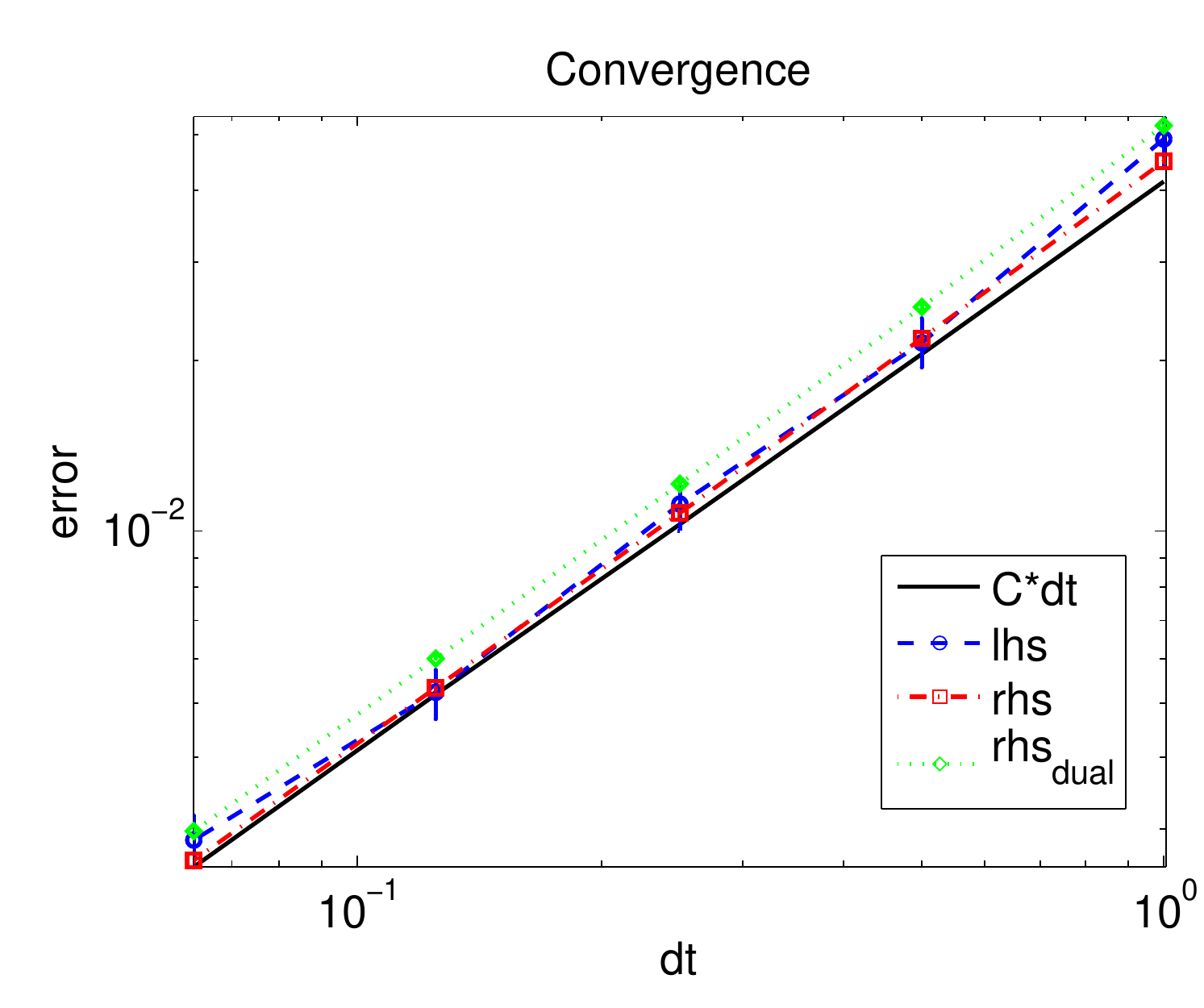}}
  \subfigure[Example 2]{
    \includegraphics*[width=0.45\textwidth,viewport=70 220 500 580]
    {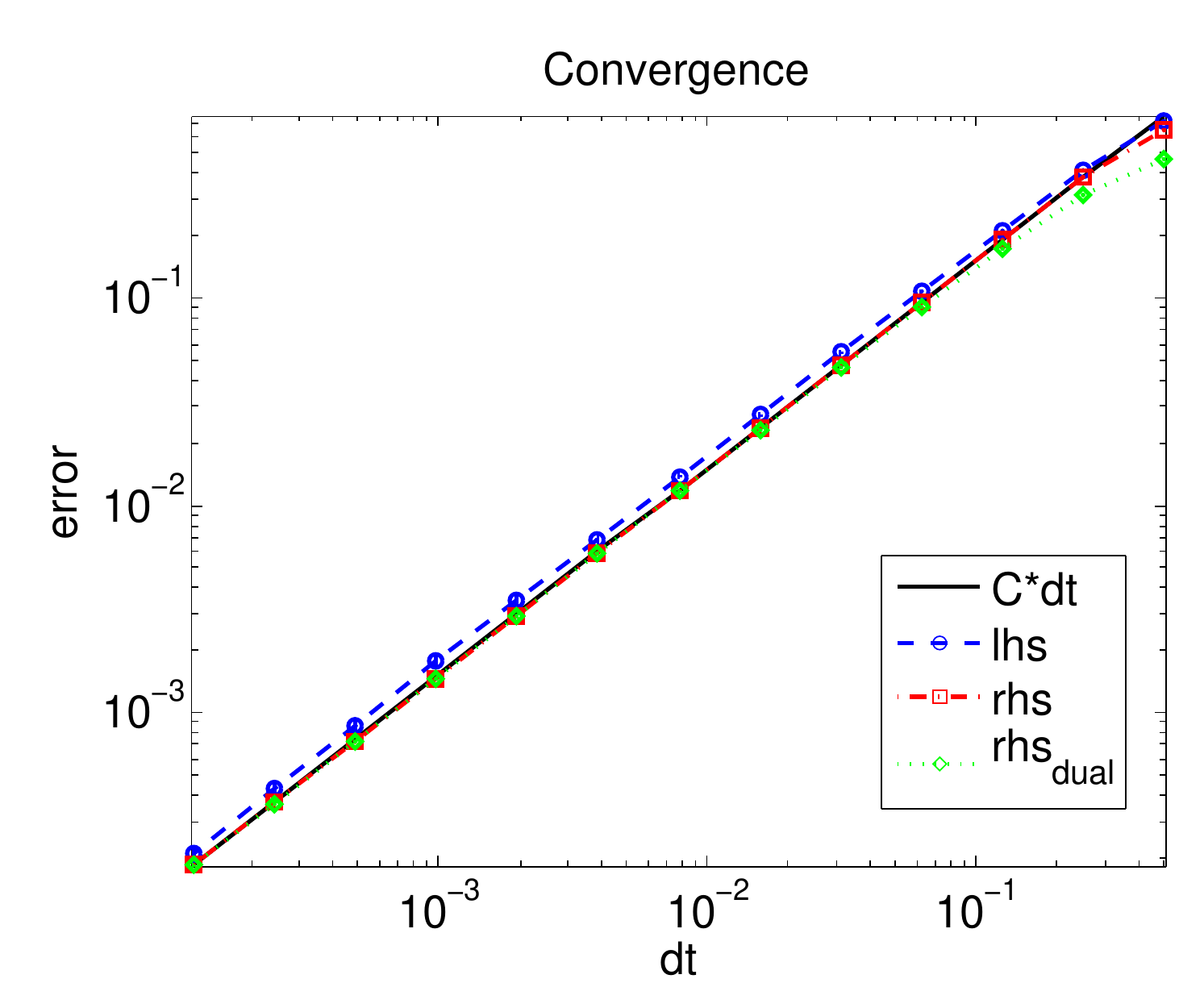}}
  \subfigure[Example 3]{
    \includegraphics[width=0.45\textwidth,viewport=70 220 500 580]
    {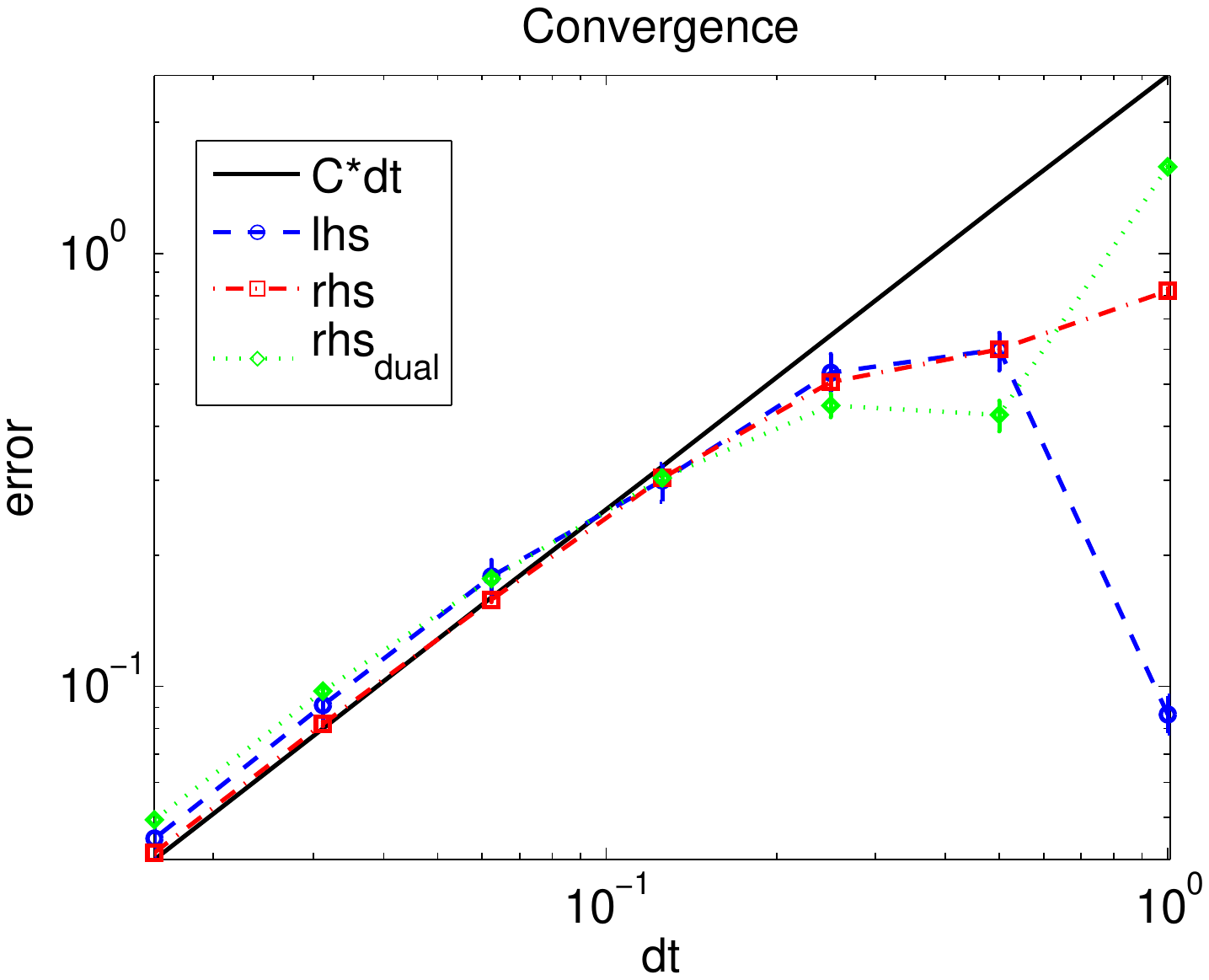}}
  \subfigure[Example 4]{
    \includegraphics[width=0.45\textwidth,viewport=70 220 500 580]
    {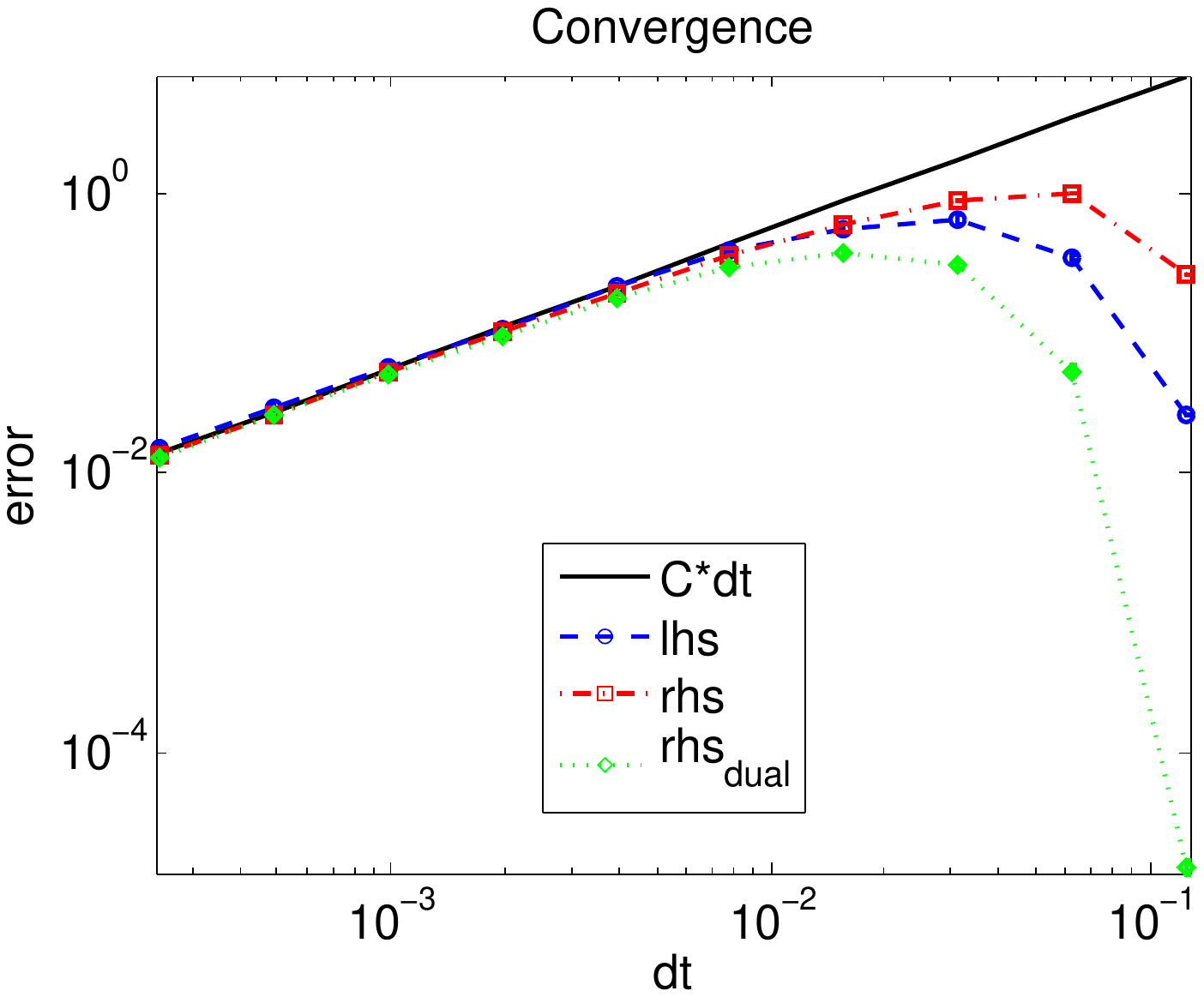}}
  \caption{Reactive decay: convergence for third moment of $X$.}
  \label{fig:decay_convergence_mom3}
\end{figure}

\begin{figure}[hbpt]
  \centering
  \subfigure[Example 1]{
    \includegraphics*[width=0.45\textwidth,viewport=70 220 500 580]
    {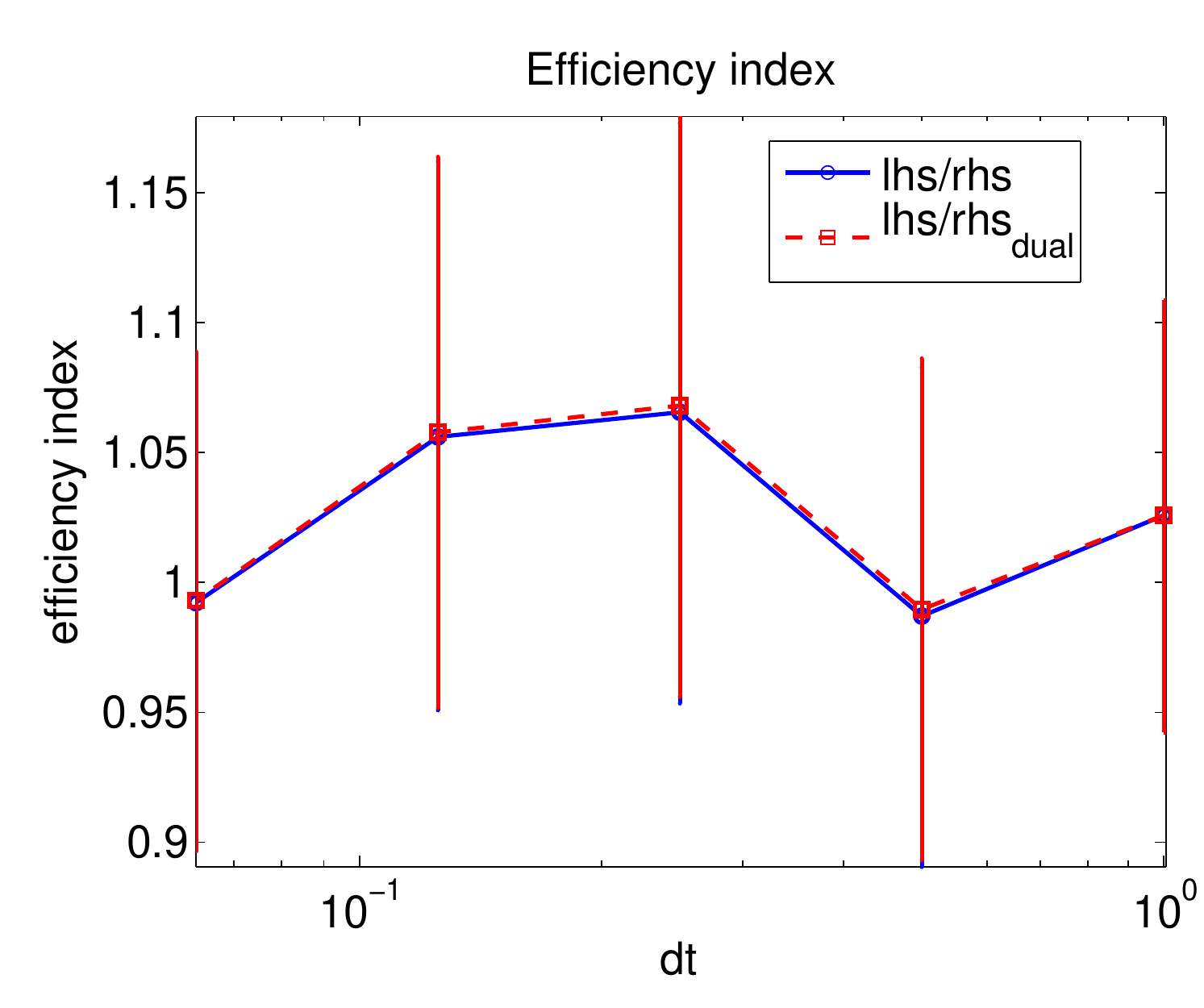}}
  \subfigure[Example 2]{
    \includegraphics*[width=0.45\textwidth,viewport=70 220 500 580]
    {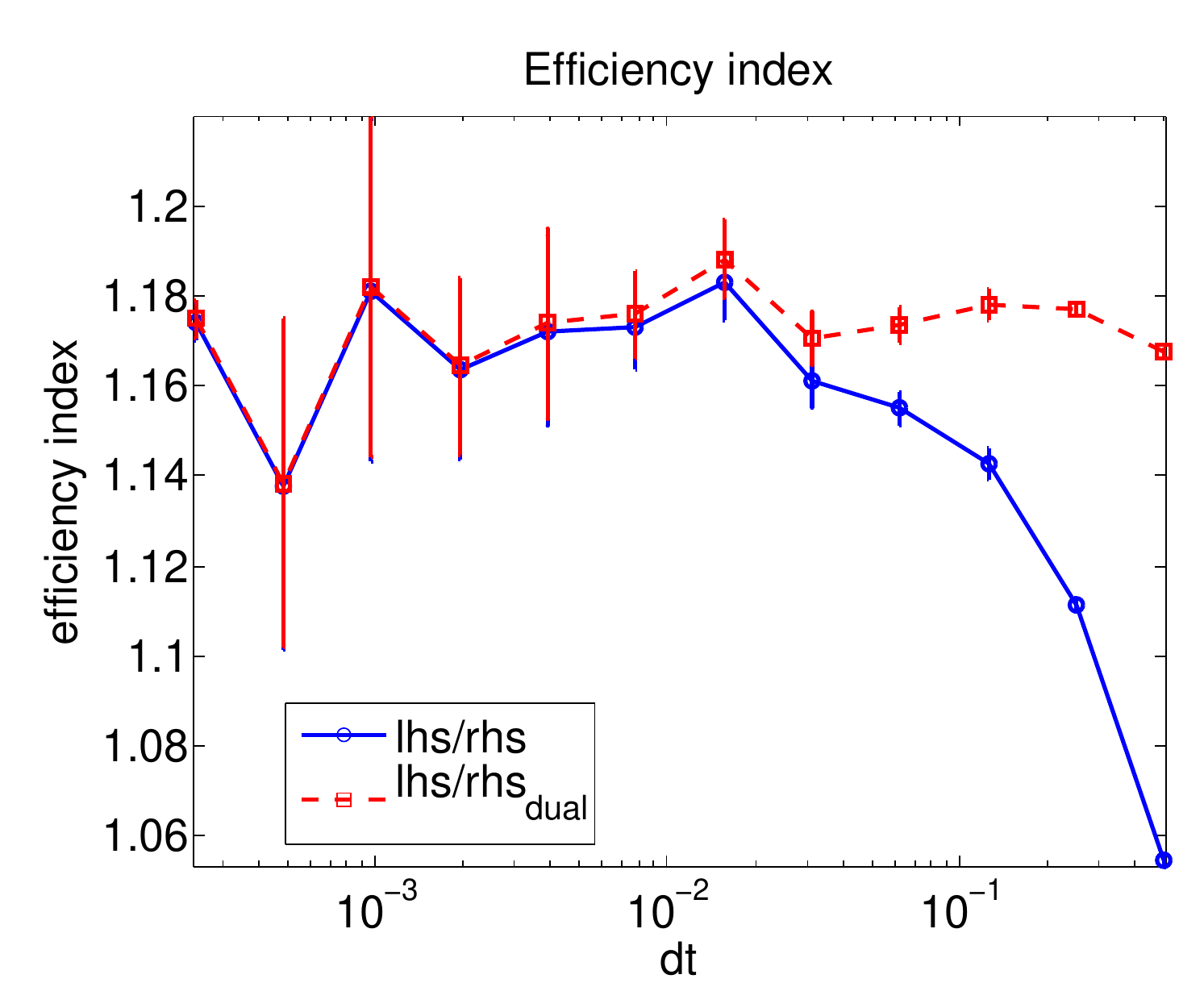}}
  \subfigure[Example 3]{
    \includegraphics[width=0.45\textwidth,viewport=70 220 500 580]
    {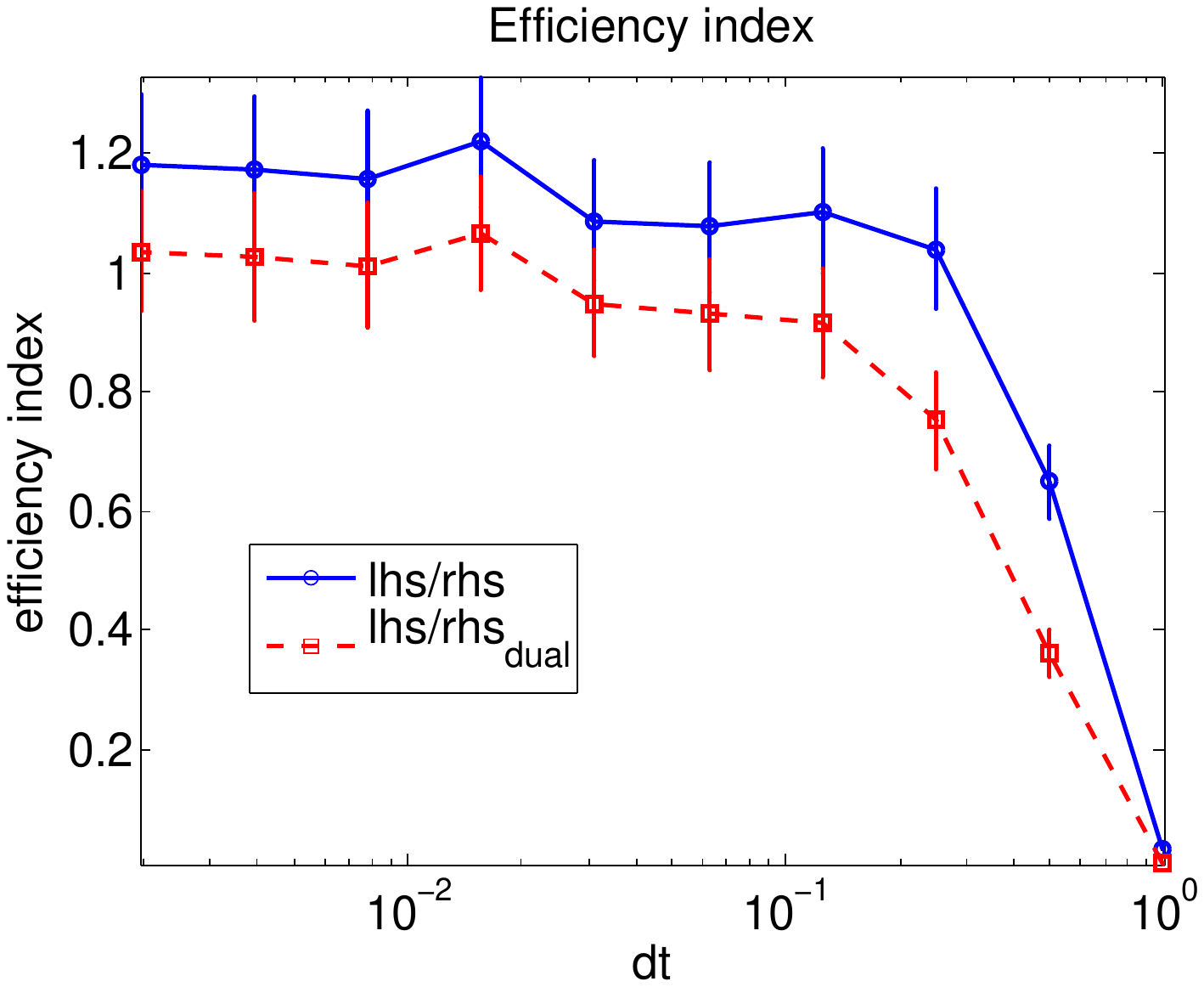}}
  \subfigure[Example 4]{
    \includegraphics[width=0.45\textwidth,viewport=70 220 500 580]
    {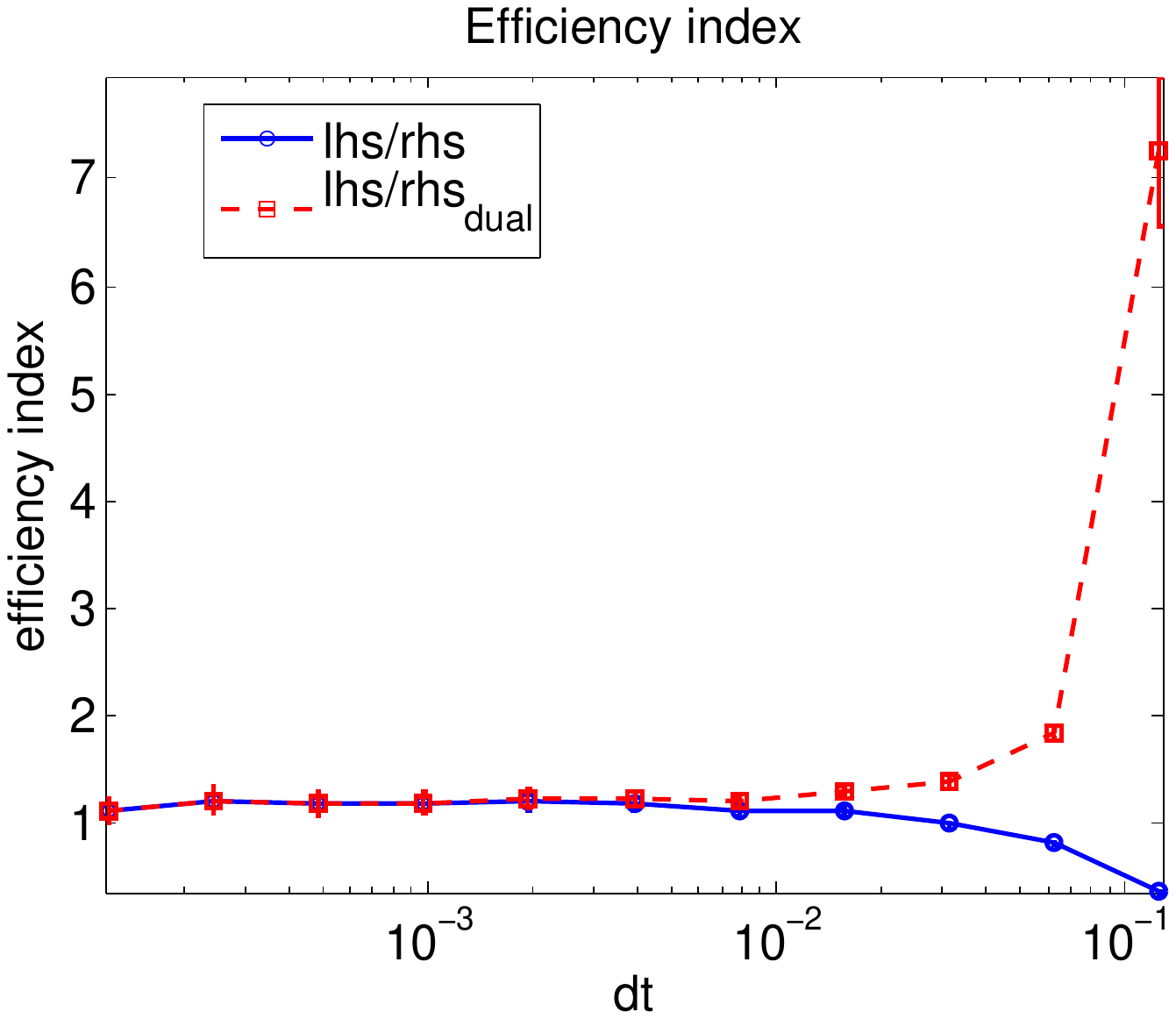}}
  \caption{Reactive decay: Efficiency index for first moment of $X$
    with $95\%$ confidence interval. The mean and confidence intervals were
    calculated by using 200 Bootstrap samples.}
  \label{fig:decay_efficiency_mom1}
\end{figure}
\begin{figure}[hbpt]
  \centering
  \subfigure[Example 1]{
    \includegraphics*[width=0.45\textwidth,viewport=70 220 500 580]
    {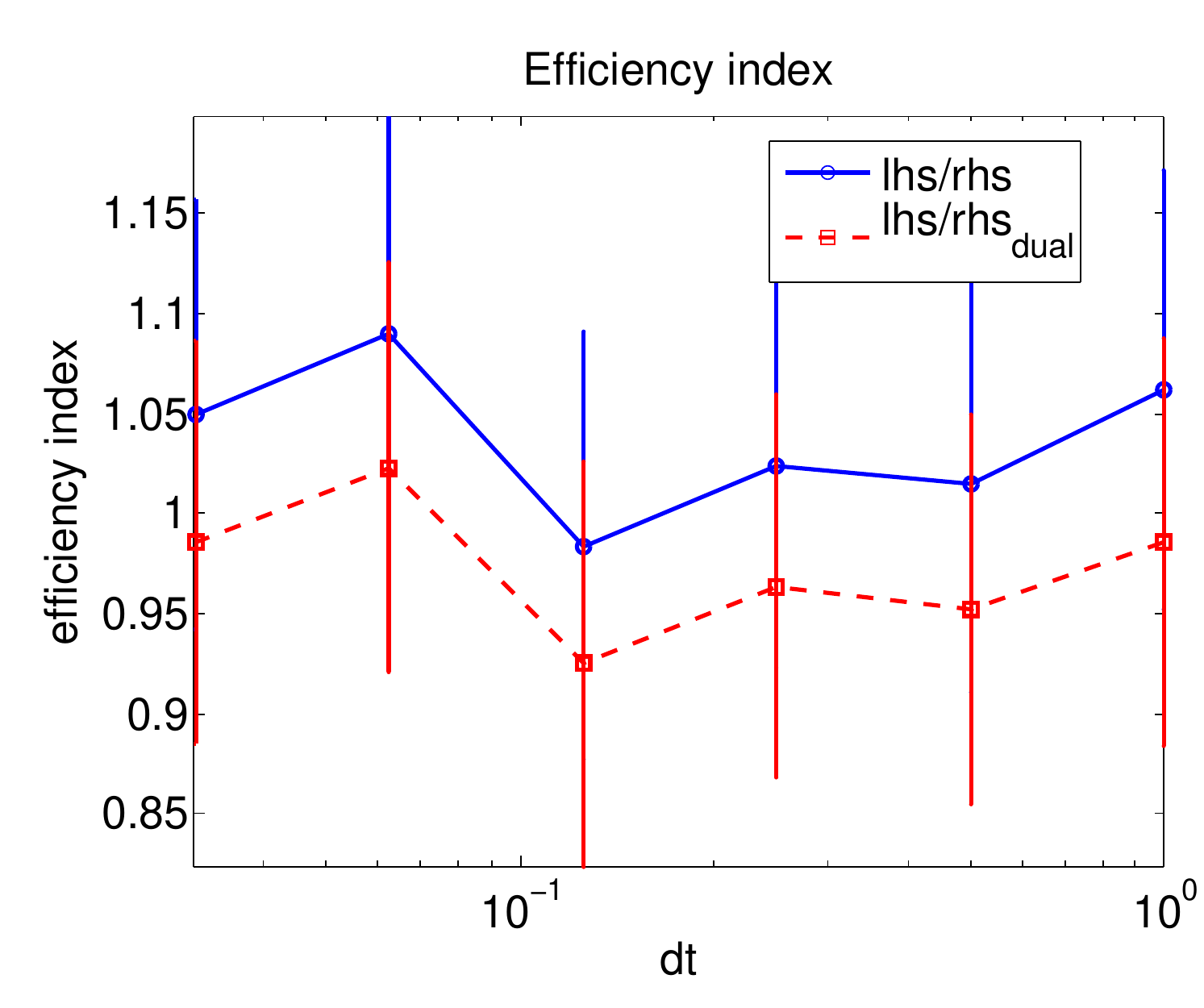}}
  \subfigure[Example 2]{
    \includegraphics*[width=0.45\textwidth,viewport=70 220 500 580]
    {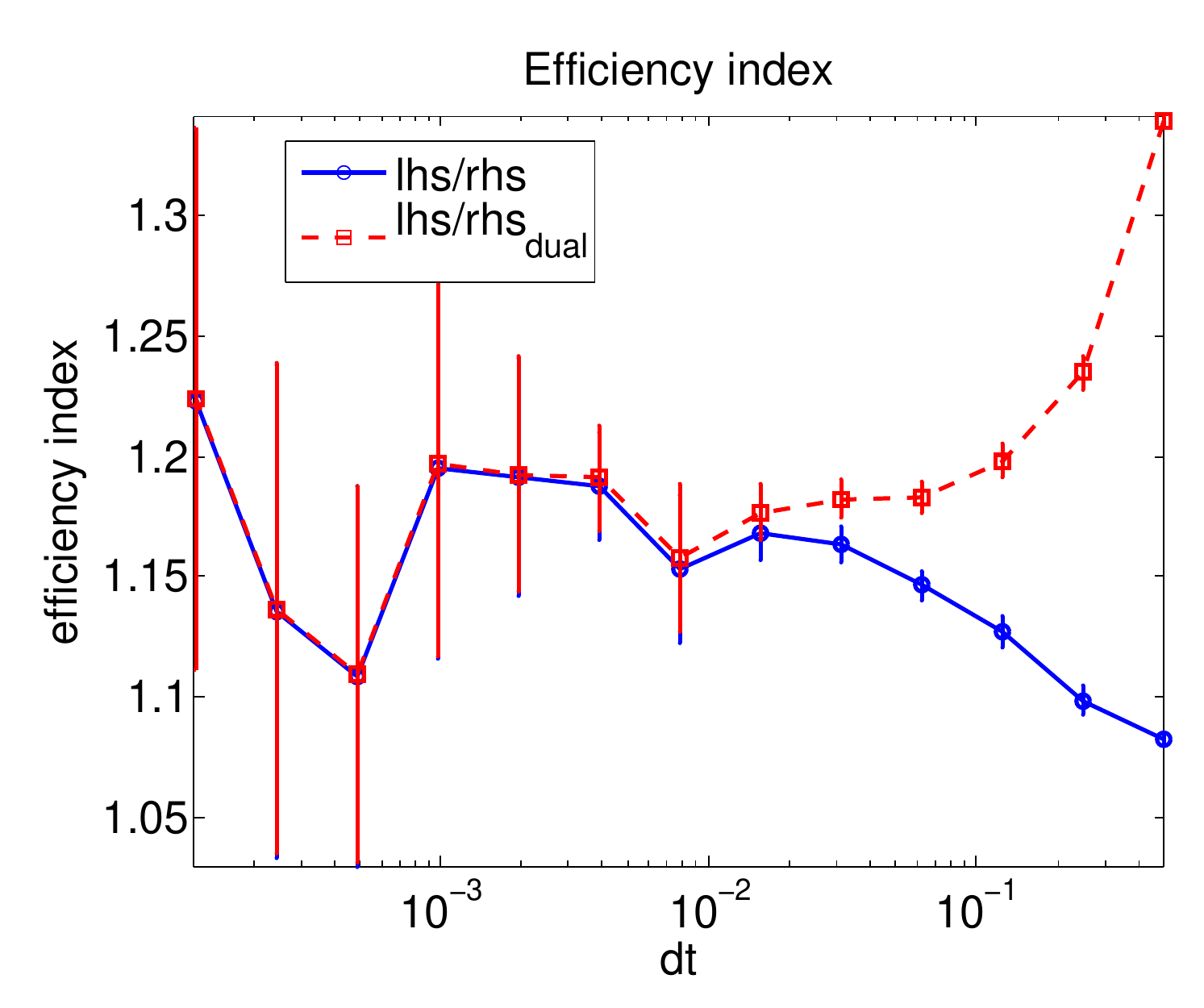}}
  \subfigure[Example 3]{
    \includegraphics[width=0.45\textwidth,viewport=70 220 500 580]
    {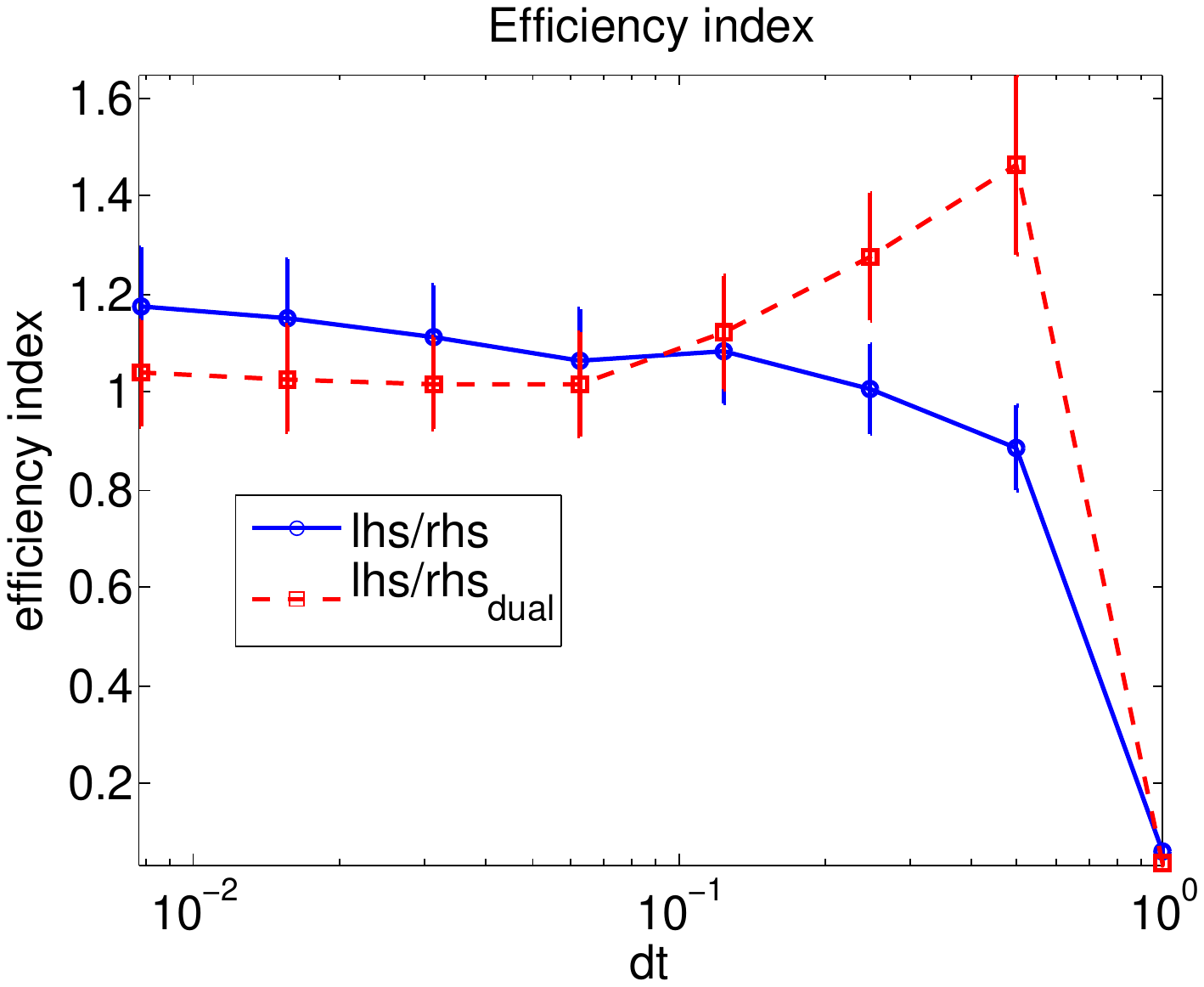}}
  \subfigure[Example 4]{
    \includegraphics[width=0.45\textwidth,viewport=70 220 500 580]
    {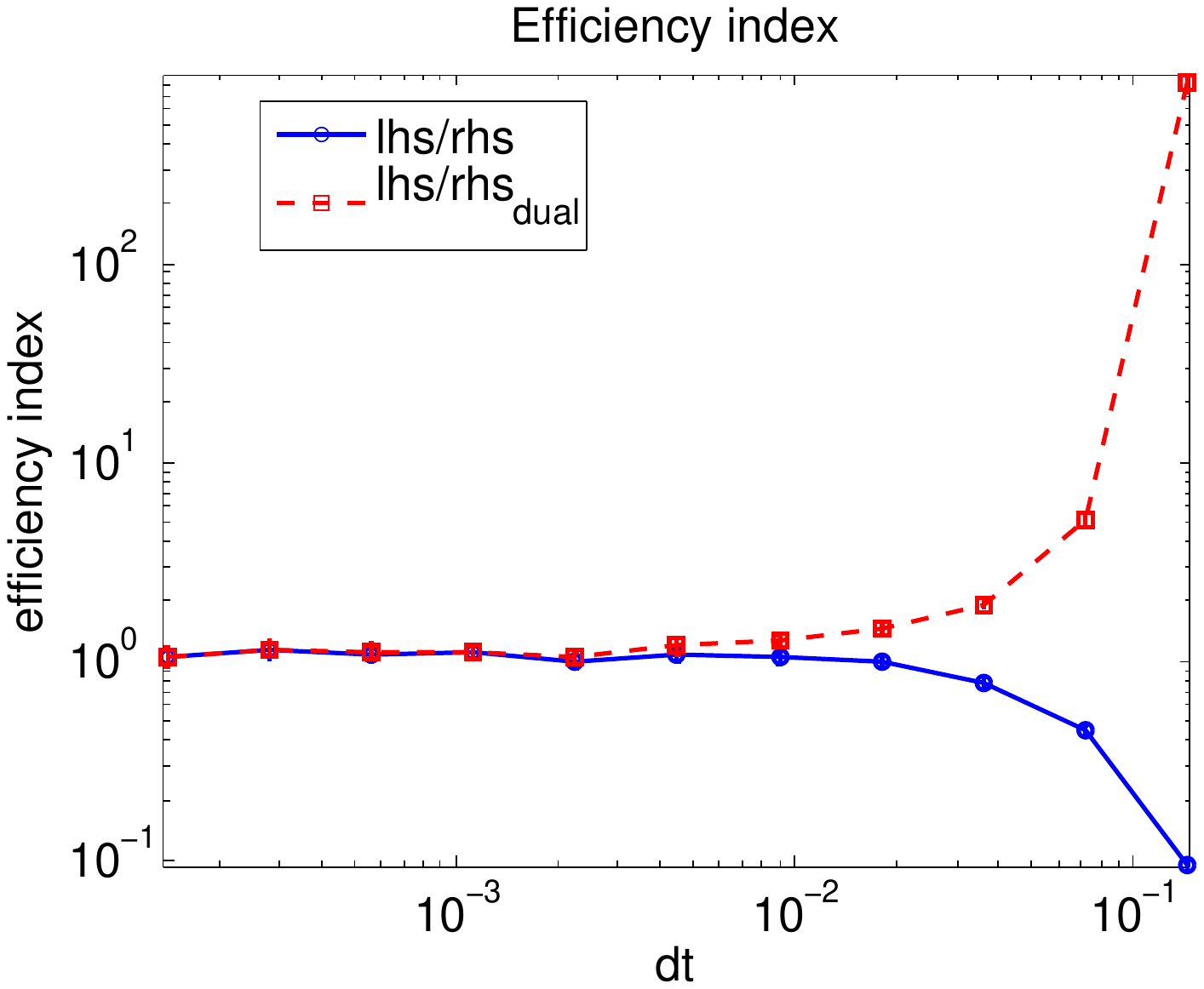}}
  \caption{Reactive decay: Efficiency index for second moment of $X$.}
  \label{fig:decay_efficiency_mom2}
\end{figure}
\begin{figure}[hbpt]
  \centering
  \subfigure[Example 1]{
    \includegraphics*[width=0.45\textwidth,viewport=70 220 500 580]
    {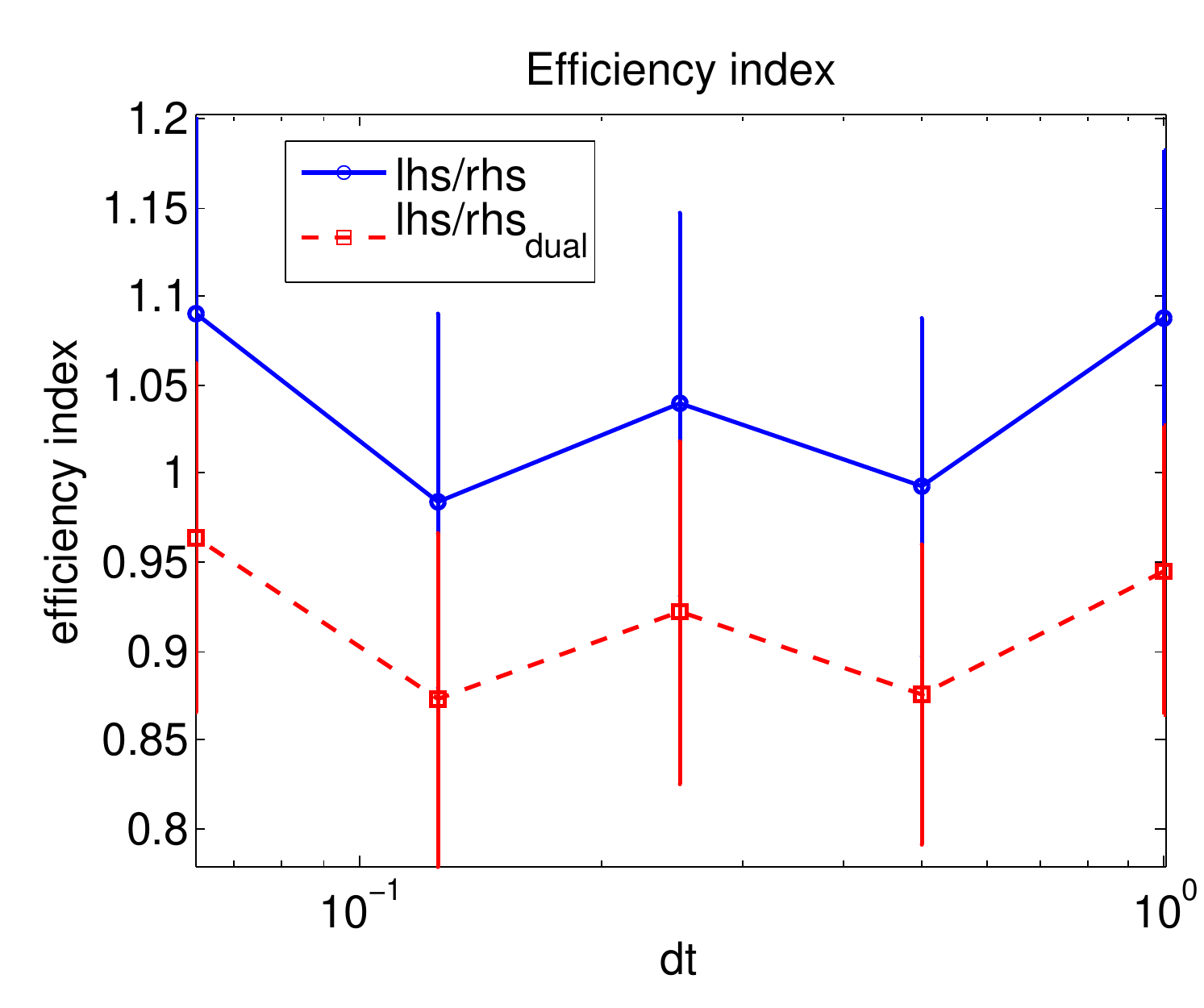}}
  \subfigure[Example 2]{
    \includegraphics*[width=0.45\textwidth,viewport=70 220 500 580]
    {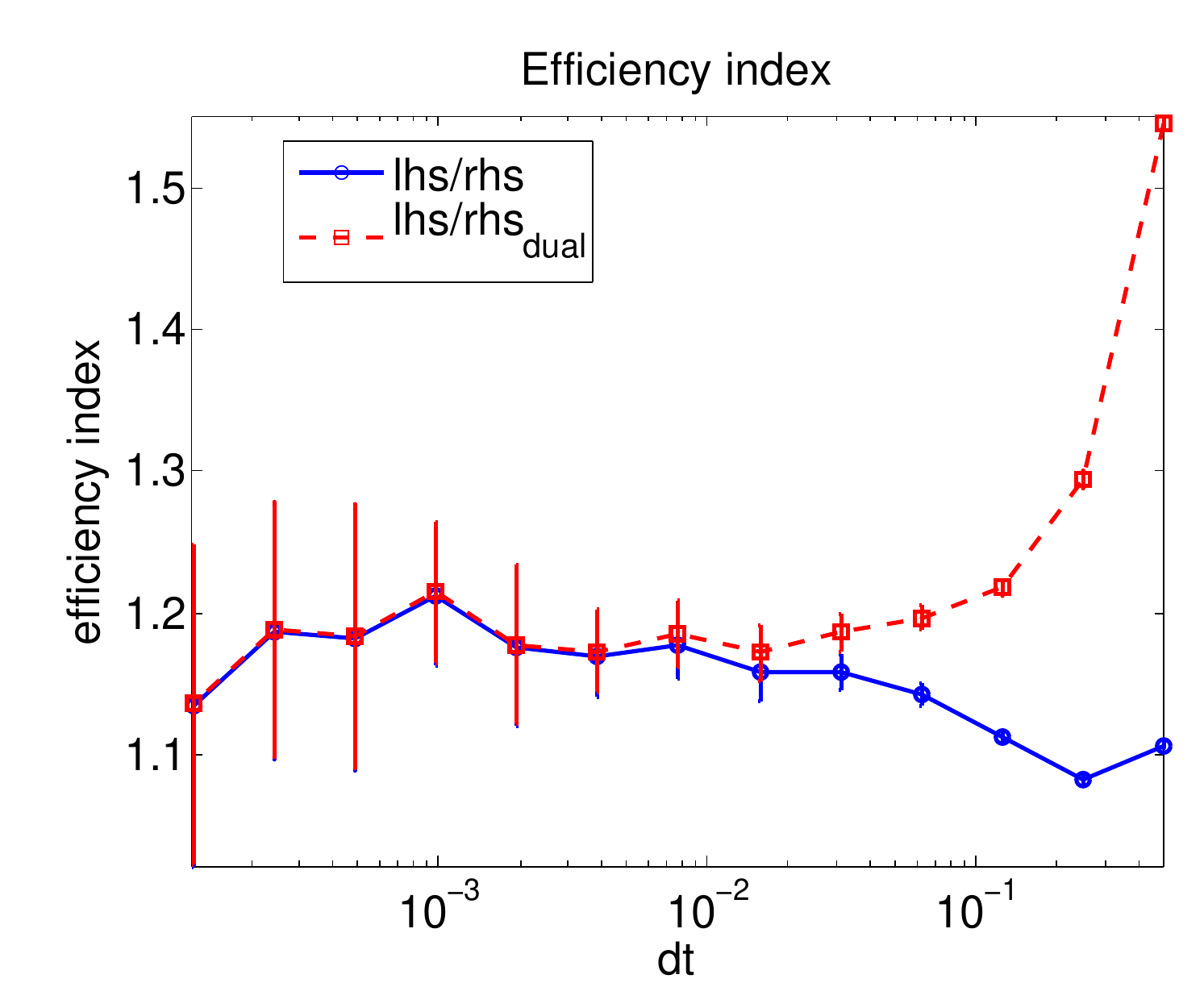}}
  \subfigure[Example 3]{
    \includegraphics[width=0.45\textwidth,viewport=70 220 500 580]
    {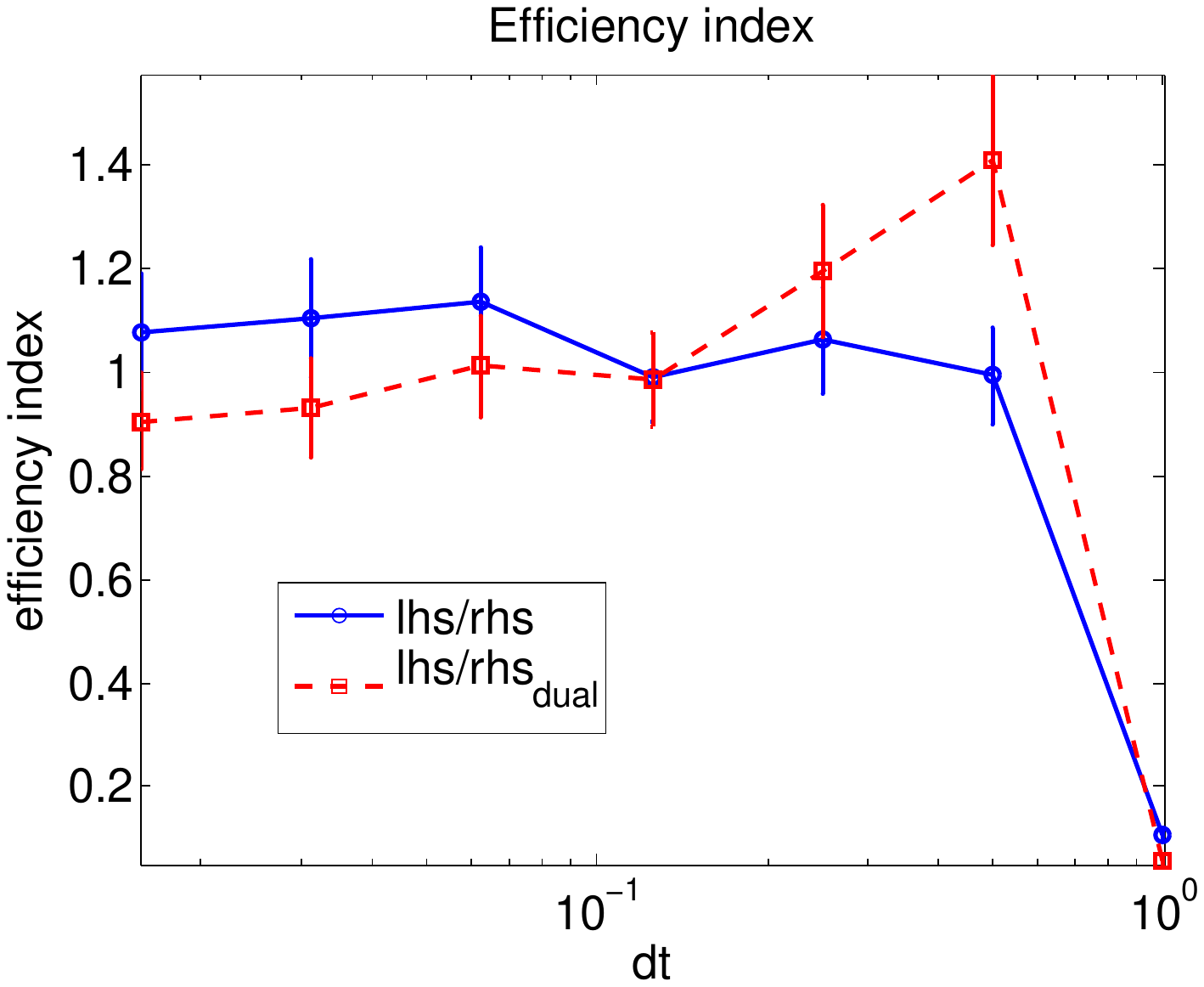}}
  \subfigure[Example 4]{
    \includegraphics[width=0.45\textwidth,viewport=70 220 500 580]
    {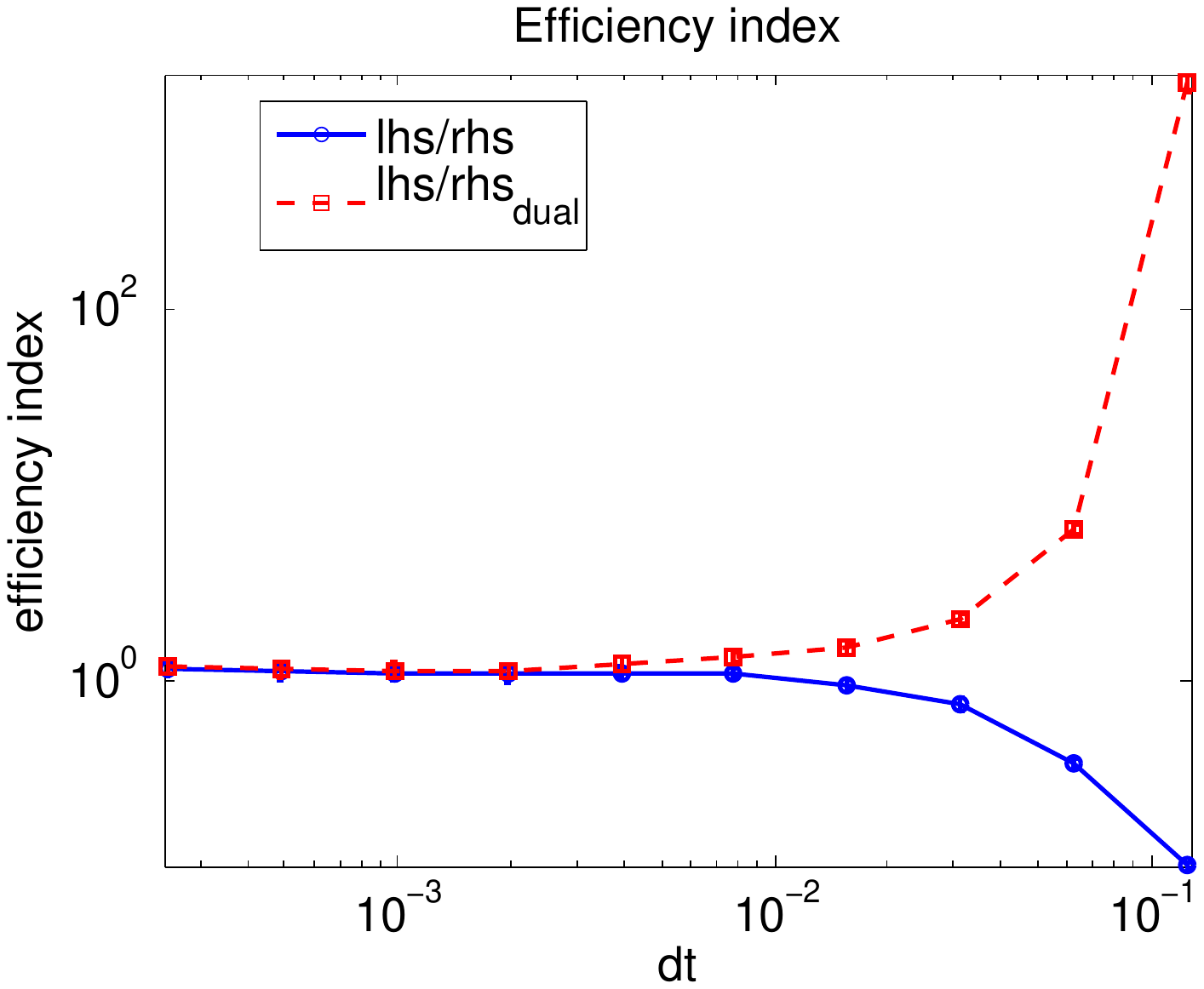}}
  \caption{Reactive decay: Efficiency index for third moment of $X$.}
  \label{fig:decay_efficiency_mom3}
\end{figure}

\begin{table}[hbpt]
  \centering
  \begin{tabular}{c|cccc}
    Example                 & 1  & 2    & 3    & 4 \\
    \hline
    Number of realizations  & 7.5$\cdot 10^4$ & 4.0$\cdot 10^1$
    & 1.3$\cdot 10^5$ & 3.6$\cdot 10^2$\\
    Current work            & 17 & 4097 & 513  & 11265\\ 
    Optimal work            & 16 & 4096 & 512  & 10866\\
    Uniform work            & 16 & 4096 & 512  & 10868\\
    Optimal work using dual & 16 & 4096 & 512  & 10867\\
    Uniform work using dual & 16 & 4096 & 512  & 10869
  \end{tabular}
  \caption{Current and estimated work for first moment. Values
    corresponding to smallest $\tau$ in 
    Figure \ref{fig:decay_convergence_mom1}.
    Two error densities are here used: one using the dual as in
    \eqref{eq:error_indicator} and one using the derivative of the 
    true value function.}
  \label{tab:decay_steps}
\end{table}

\subsection{Unstable dimer}
This stiff model was used in \eg{}
\cite{anderson,gillespie_leap_check} and has four reactions and three
species. The reactions are
\begin{equation*}
  \begin{aligned}
    X_1 &\to 0, & X_2 &\to 2X_1,\\
    2X_1 &\to X_2, & X_2 &\to X_3,
  \end{aligned}
\end{equation*}
described by the stoichiometric matrix and propensity function
\begin{equation*}
  \nu =
  \begin{pmatrix}
    -1 & -2  & 2 & 0\\
    0  & 1 &  -1 & -1\\
    0  &  0 &  0 &  1
  \end{pmatrix}
  \quad
  \text{and}
  \quad
  a(X) = 
  \begin{pmatrix}
    X_1 \\ 0.001 X_1 (X_1-1) \\ 0.5X_2 \\  0.04X_2
  \end{pmatrix},
\end{equation*}
respectively. In Figure \ref{fig:dimer_realizations}, where a few
realizations of the path $X_t:=(X_1,X_2,X_3)(t)$ are shown (in a
log-lin scale), it can be noted that there is a large difference in
time-scales; during a very short time most of the $X_1$-molecules will turn
into $X_2$-molecules. This difference in time scales has several
consequences: the step size during the transient phase may have a big
effect on the error, and the tau-leap method will cause negative
populations unless the step size is adjusted accordingly.

We here let $g(x)=x_1+x_2+x_3$, and the convergence of the
corresponding error and efficiency index can be seen in Figure
\ref{fig:dimer_convergence}. As in the previous example we see that
the error decreases linearly with $\tau$, as expected, with error
estimates being close to each other.

From Figure \ref{fig:dimer_indicators} we see that the error
density here play a big role, and that the optimal time stepping is
to choose very small time steps during the transient phase.  To see if
choosing a pre-leap check as in \cite{gillespie_leap_size} will give a
similar result we apply the leap-size condition \eqref{eq:leap_size}
with $\epsilon = 0.05$, which turns out to give almost the same error
density and proposed time steps.
In Table \ref{tab:dimer_steps1} a comparison between the current and
estimated work shows potential for great improvement using adaptive
time stepping.


\begin{figure}[hbpt]
  \centering
  \includegraphics*[width=0.45\textwidth,viewport=80 230 500 580]
  {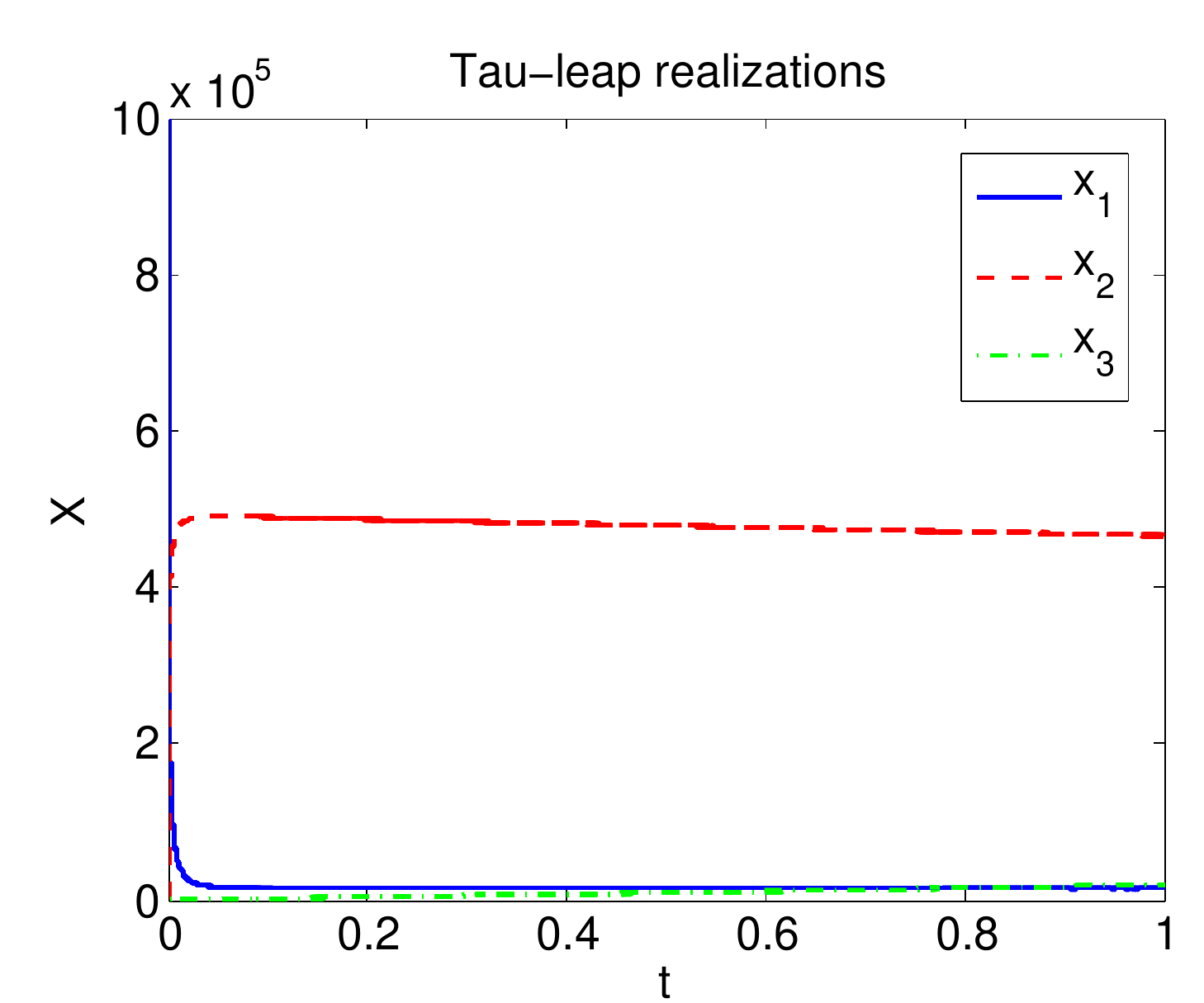}
  \includegraphics*[width=0.45\textwidth,viewport=80 230 500 580]
  {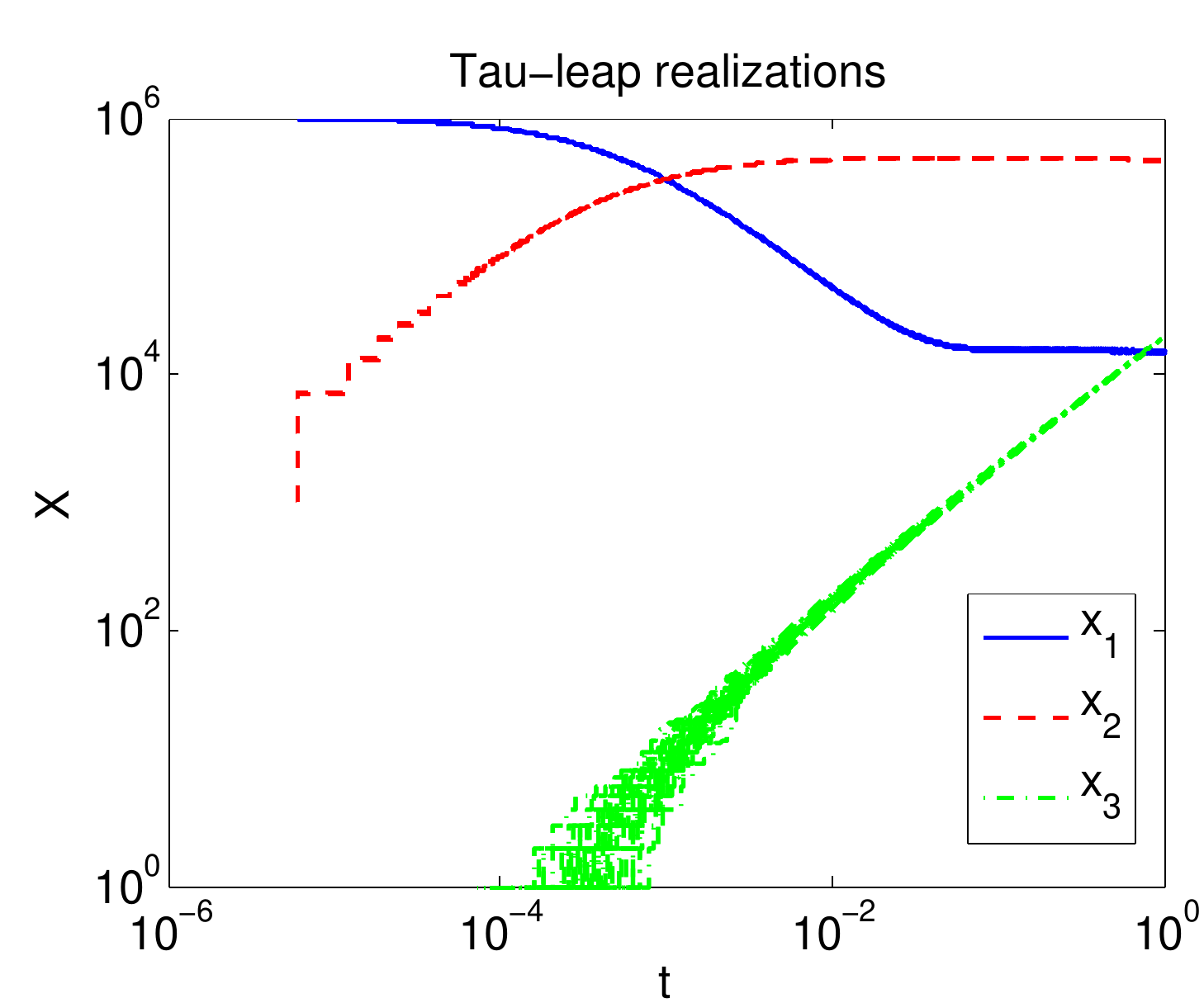}
  \caption{Realizations for unstable dimer shown in different axis scales.}
  \label{fig:dimer_realizations}
\end{figure}

\begin{figure}[hbpt]
  \centering
  \subfigure[Convergence]{
    \includegraphics*[width=0.45\textwidth,viewport=80 230 500 580]
    {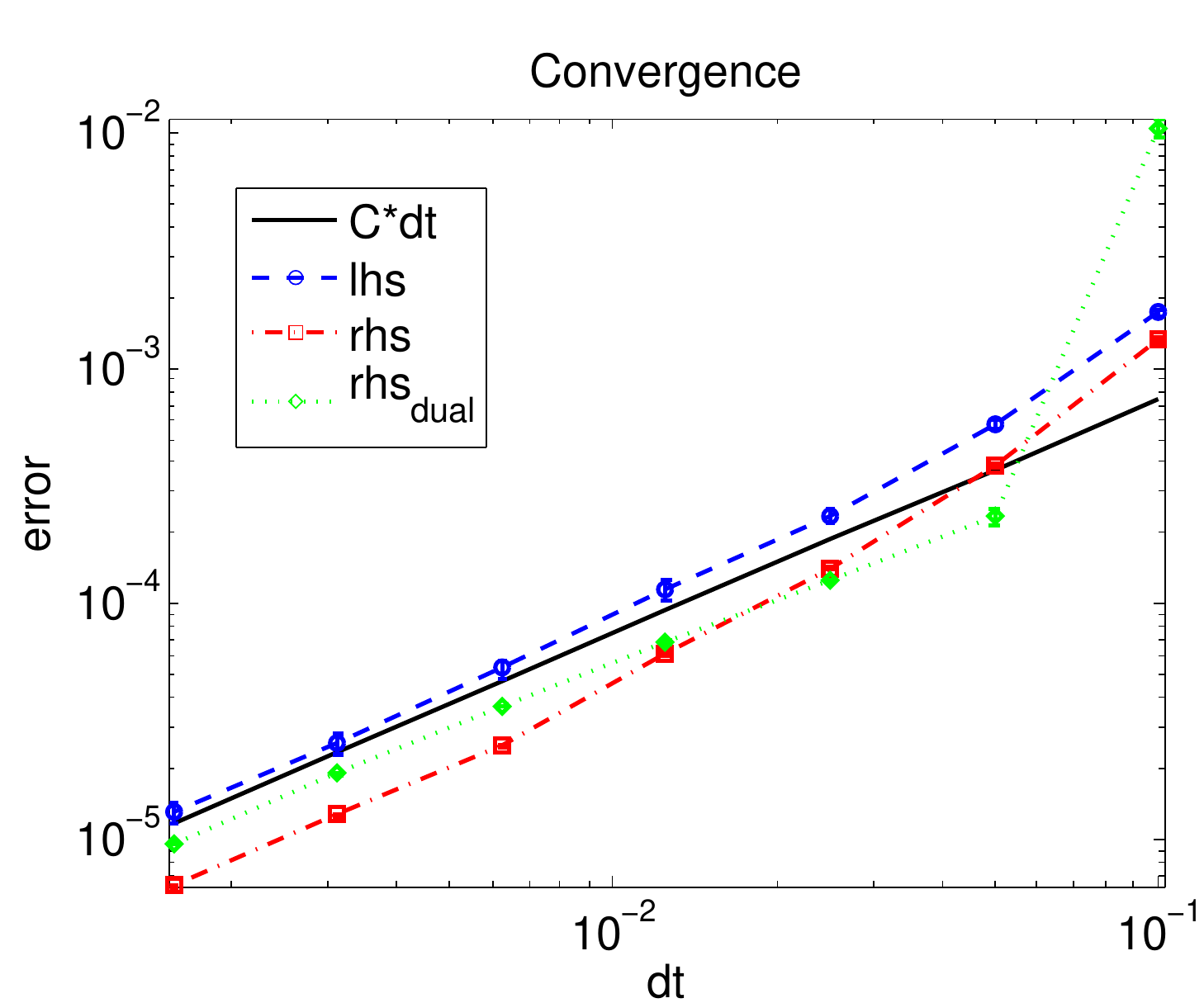}}
  \subfigure[Efficiency index]{
    \includegraphics*[width=0.45\textwidth,viewport=80 230 500 580]
    {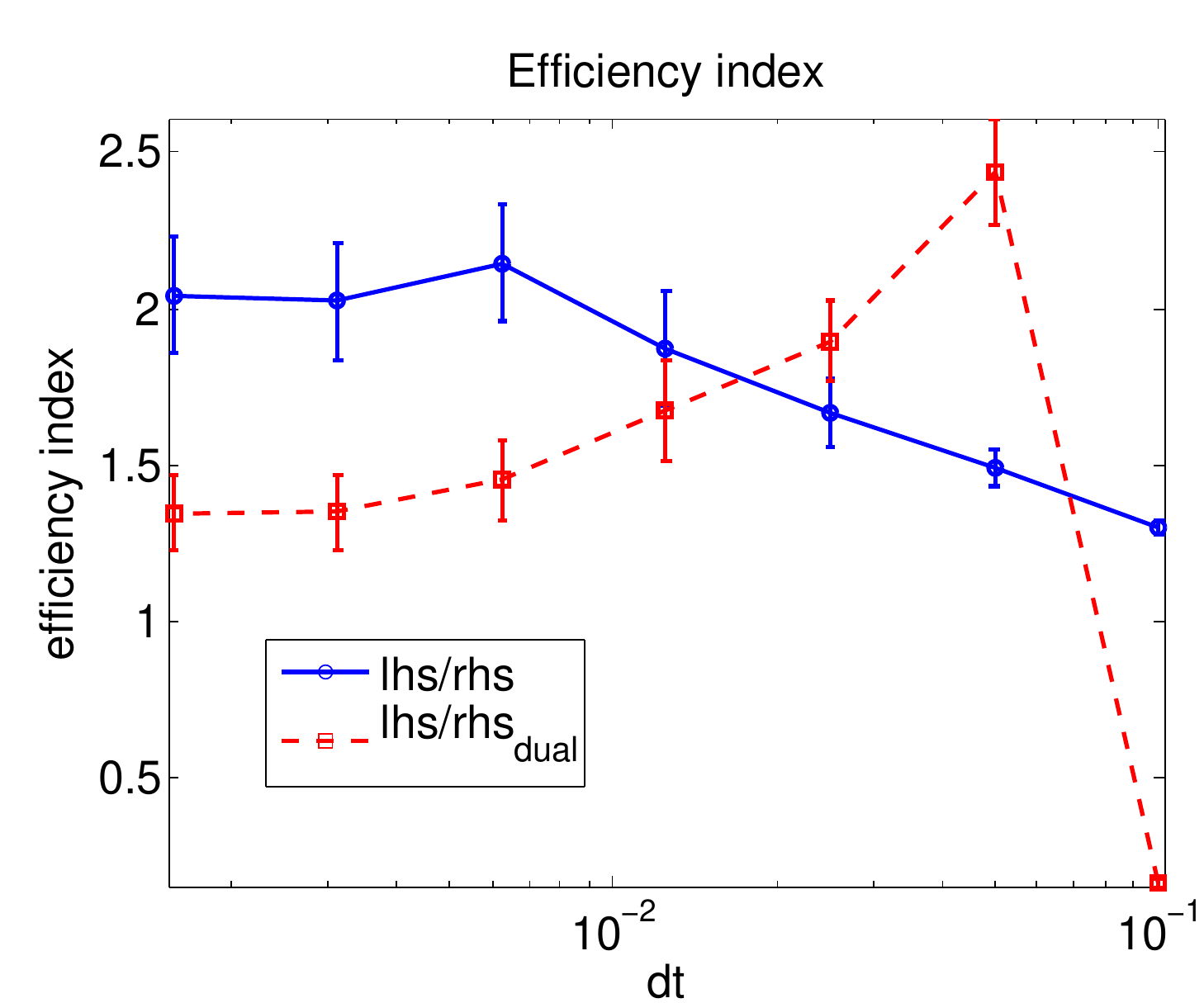}}
  \caption{Unstable dimer: Convergence and efficiency index.}
  \label{fig:dimer_convergence}
\end{figure}


\begin{figure}[hbpt]
  \centering
  \subfigure[Error densities]{
    \includegraphics*[width=0.45\textwidth,viewport=0 100 600 650]
    {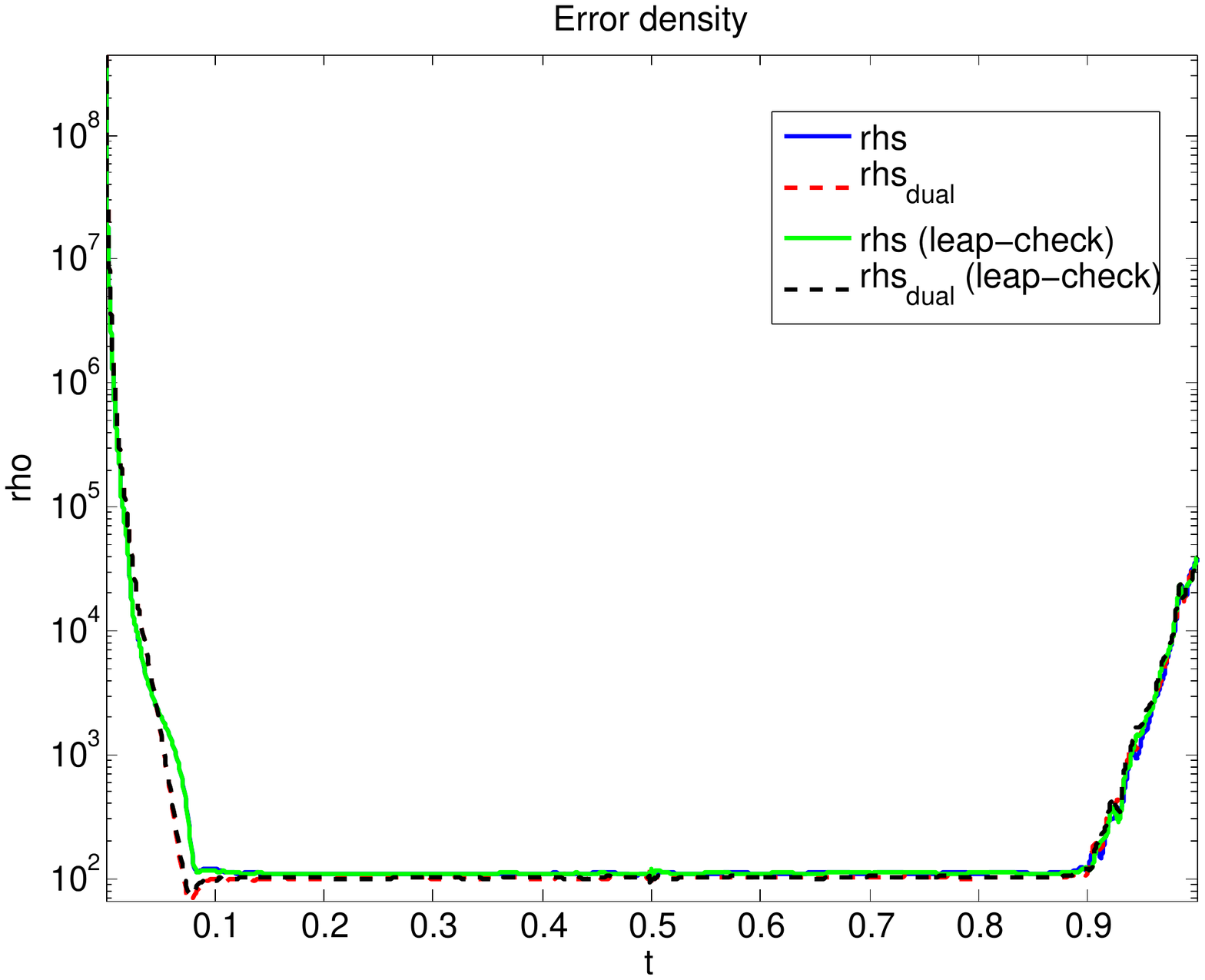}}
  \subfigure[Step sizes]{
    \includegraphics*[width=0.45\textwidth,viewport=0 100 600 650]
    {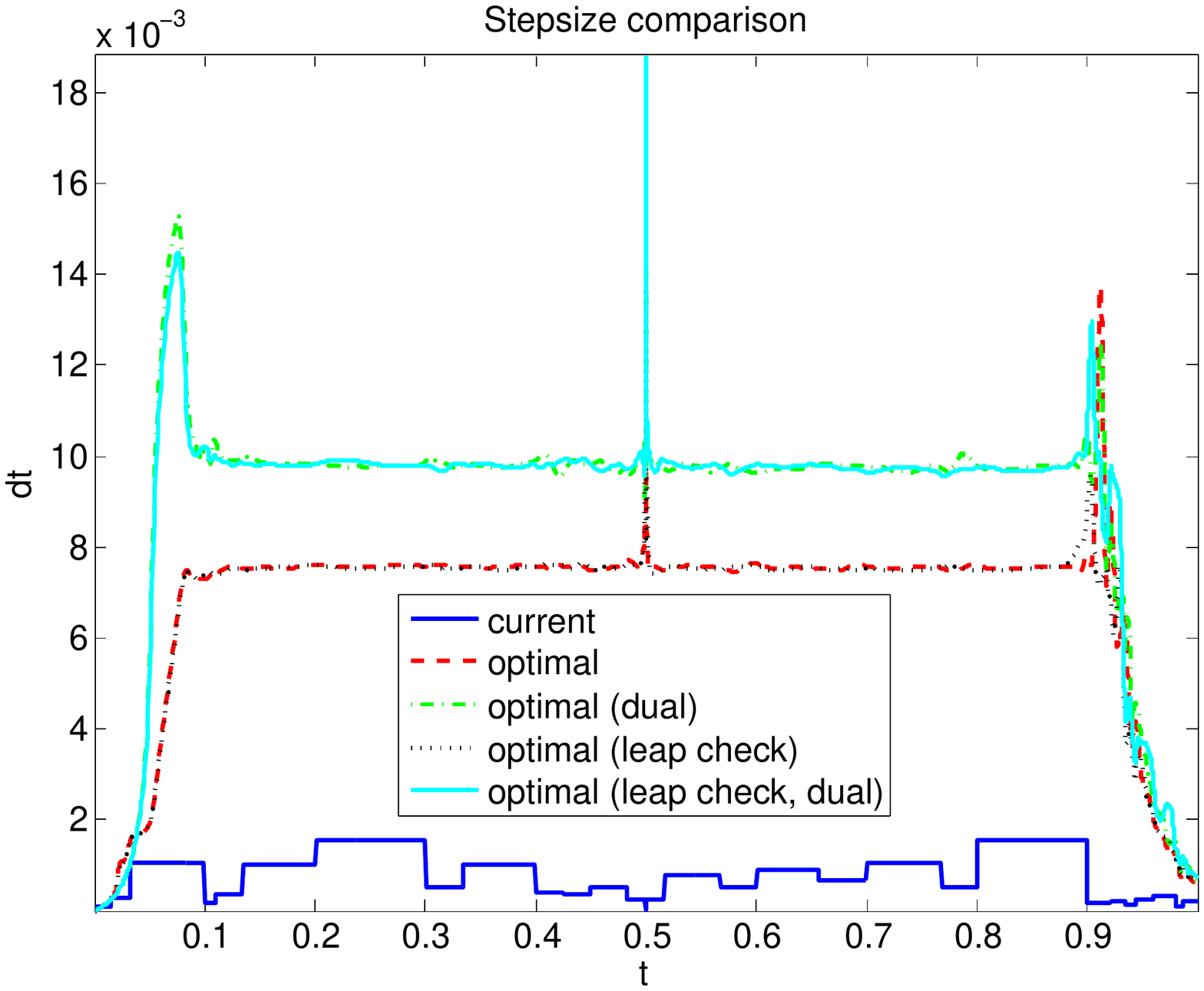}}
  \caption{Unstable dimer: Error densities and step sizes.}
  \label{fig:dimer_indicators}
\end{figure}


\begin{table}[hbpt]
  \centering
  \begin{tabular}{c|cccc}
    Leap-check & No & Yes\\
    \hline
    Current work & 2628 & 2637\\ 
    Optimal work & 573 & 575\\
    Uniform work & 40480 & 40501\\
    Optimal work using dual & 529 & 531\\
    Uniform work using dual & 42362 & 42313    
  \end{tabular}
  \caption{Unstable dimer: Current and estimated work for 880
    realizations. Values
    corresponding to smallest $\tau$ in 
    Figure \ref{fig:dimer_convergence}.
  }
  \label{tab:dimer_steps1}
\end{table}


\section{Conclusions}
We have in this work shown a weak global error representation for the
tau-leap method that can be accurately approximated by a computable leading
order term using a discrete dual weighted propensity residual, see
Lemma \ref{lem:tau_exp_error_rep} and Theorem
\ref{thm:aposteriori}. This type of \emph{a posteriori} error
estimates, using discrete dual functions, are the first for the
tau-leap method and are important tools for the ongoing work on developing
efficient adaptive time stepping algorithms, see \eg{}
\cite{mordecki,msst}.
Also, we have here shown an \emph{a priori} estimate on the relative
error of the tau-leap method, that is of order $\tau$ independently of the
number of particles in the system, see Theorem
\ref{thm:rel_error_bnd}.

Our results are based on an error representation using the value
function for the corresponding Kolmogorov backward equation, which for
jump processes is defined on a discrete lattice. Using extensions from
lattices onto real values and stochastic representations, we develop
weighted estimators for the value function and its derivatives, see
Lemma \ref{lem:cont_ext} and Lemma \ref{lem:growth} respectively.
The weighted estimators developed here give \emph{polynomial} bounds
on the value function and its derivatives and improve similar
$L$-infinity bounds in \cite{katsoulakis} that are \emph{exponential}
in the propensity function.


\bibliographystyle{plain}
\bibliography{../references}

\end{document}